\documentclass[reqno,11pt]{amsart}
 
\usepackage{amsmath}
\usepackage{amsthm,amssymb,color,comment}
\usepackage{bbm}
\usepackage{mathrsfs}
\usepackage{hyperref}
\usepackage[TS1,T1]{fontenc}
\usepackage[utf8]{inputenc}
\usepackage{dsfont}
\usepackage{tikz}
\usepackage{enumitem}
\usepackage{subfigure}
\usepackage{mathtools}
\usepackage{bbold}
\usepackage{esint}
\usepackage[sort]{cite}

\numberwithin{equation}{section}

\newtheorem{thm}{Theorem}[section]

\newtheorem{proposition}[thm]{Proposition}
\newtheorem{lem}[thm]{Lemma}

\newtheorem{Def}[thm]{Definition}

\theoremstyle{definition}
\newtheorem{Ass}[thm]{Assumption}
\newtheorem{rem}[thm]{Remark}

\DeclareMathOperator{\supp}{supp}

\DeclareMathOperator{\DIV}{div}
\DeclareMathOperator{\CURL}{curl}

\DeclareMathOperator{\loc}{loc}

\newcommand{\R}{\mathbb{R}}

\newcommand{\Rd}{\mathbb{R}^{d}}

\newcommand{\N}{\mathbb{N}}

\newcommand{\eps}{\varepsilon}

\newcommand{\diff}{\mathop{}\!\mathrm{d}}

\newcommand{\ueps}{u_{\varepsilon}}
\newcommand{\veps}{v_{\varepsilon}}
\newcommand{\seps}{s_{\varepsilon}}

\newcommand{\weak}{\rightharpoonup}

\makeatletter
\newcommand{\doublewidetilde}[1]{{%
  \mathpalette\double@widetilde{#1}%
}}
\newcommand{\double@widetilde}[2]{%
  \sbox\z@{$\m@th#1\widetilde{#2}$}%
  \ht\z@=.9\ht\z@
  \widetilde{\box\z@}%
}
\makeatother

\author{Jakub Skrzeczkowski}
\address{{\it Jakub Skrzeczkowski: } St John's College, University of Oxford, St Giles, Oxford, OX1 3JP, United Kingdom \& Mathematical Institute, University of Oxford, Woodstock Road, Oxford, OX2 6GG, United Kingdom}
\email{jakub.skrzeczkowski@maths.ox.ac.uk}

\begin{document}

\title[]{Global solutions to cross-diffusion systems with independent advections in one dimension}

\begin{abstract}
We consider cross-diffusion systems describing evolution of two species $u$ and $v$ moving according to Darcy's law with the pressure law $p(s) = \frac{1}{\alpha-1} s^{\alpha-1}$ where $s=u+v$. One of the most challenging questions in the field is the construction of solutions to the problem in the presence of additional advection fields, without imposing any artificial structure on the fields or the initial conditions. Although advection arises naturally in these models, it breaks the symmetry of the system and prevents application of techniques developed in recent years. Here, we provide a new approach to construct solutions in one space dimension that works in a unified manner for all pressure exponents $\alpha \in (0,\infty)$ and for arbitrary initial data. In~particular, in the regime $\alpha > 1$, this yields the first existence result of its kind, obtained without any structural assumptions. We construct the solutions as a limit of a vanishing viscosity approximation $(u_{\varepsilon}, v_{\varepsilon})$. The main challenge is to identify the limit of $u_{\varepsilon} \, \partial_x p(s_{\varepsilon})$, where $s_{\varepsilon} = u_{\varepsilon} + v_{\varepsilon}$. The key new insight is that possible oscillations of $u_\varepsilon$ and $\partial_x p(s_\varepsilon)$ are correlated, simplifying the Young measure analysis in the compensated compactness argument and allowing identification of the limit. Somewhat surprisingly, in contrast to the theory of $2\times2$ hyperbolic systems, the argument relies on only three entropy-entropy flux pairs. This is particularly useful for $\alpha>2$, where it is unclear whether additional entropies are available.
\end{abstract}



\maketitle

\setcounter{tocdepth}{1}

\section{Introduction}
Cross-diffusion systems have attracted considerable attention in recent years due to their analytical challenges \cite{MR4188329,MR5033040,MR3740386,MR4939528,MR4000848} and their applications in the applied sciences \cite{MR3948738,MR4902815,murakawa2015continuous}. These partial differential equations, first formulated in the 1970s and 1980s \cite{MR736508,MR821681,MR540951,MR887658}, still present many challenging and open mathematical problems. The aim of the present paper is to address one of them, namely, the existence of weak solutions to the following cross-diffusion system:
\begin{equation}\label{eq:general_cross_diffusion_intro}
\begin{split}
\partial_t u &= \partial_x(u\, \partial_x p(s)) + \partial_x(u \, \partial_x( V^1 + K^{1,1}\ast u + K^{1,2}\ast v)),\\
\partial_t v &= \partial_x(v\, \partial_x p(s)) + \partial_x(v \, \partial_x( V^2 + K^{2,1}\ast u + K^{2,2}\ast v)),
\end{split}
\end{equation}
describing the evolution of two species on $\R$ having nonnegative densities $u$, $v$, where
$$
s = u+v.
$$ 
Writing these equations as continuity equations $\partial_t u + \partial_x (u\, \mathcal{W}^1) = 0$, $\partial_t v + \partial_x (v\, \mathcal{W}^2) = 0$, we recognize that the velocities $\mathcal{W}^1$ and $\mathcal{W}^2$ are given by
\begin{align*}
\mathcal{W}^1 &=- \partial_x p(s) - \partial_x( V^1 + K^{1,1}\ast u + K^{1,2}\ast v),\\
\mathcal{W}^2&=- \partial_x p(s) - \partial_x( V^2 + K^{2,1}\ast u + K^{2,2}\ast v).
\end{align*}
The first term corresponds to Darcy's law \cite{darcy1856fontaines} with the pressure  
\begin{equation}\label{eq:darcy_law_introduction}
p(s) = \begin{cases}
\frac{1}{\alpha - 1} s^{\alpha-1} &\mbox{ if } \alpha \neq 1,\\
\log s &\mbox{ if } \alpha = 1,
\end{cases} \qquad \qquad \mbox{ for } \alpha \in (0,\infty),
\end{equation}
and it models the fact that species avoid overcrowded regions. The case $\alpha>1$ seems to be more relevant from the point of view of applications as the total density $s$ is expected to move with finite speed of propagation by an analogy to the porous medium equation \cite[Chap.~14]{MR2286292}. On the other hand, for $\alpha \leq 1$ one expects that $s$ is immediately globally supported as for heat equation or fast-diffusion equation, see \cite[Theorem 2.4]{MR797051} and \cite{MR760987}. Moreover, the so-called incompressible limit $\alpha \to \infty$ shows that these models are related to the free boundary models of tissue growth for large $\alpha$ \cite{MR3162474,MR429164}.\\

The additional terms with functions $V^{i}:\R \to \R$, $K^{i,j}:\R \to \R$ are responsible for an independent advection and nonlocal interaction between the densities. We recall that $K^{i,j}~\ast~u(x) = \int_{\R} K^{i,j}(x-y)\,u(y)\diff y$ and analogously for $K^{i,j}\ast v$. The inclusion of these effects in the models is classical: nonlocal interaction terms account for repulsion and attraction between species \cite{MR2279324,MR2111591,MR1698215,MR4902815}, while advection captures spatial heterogeneity. We do not assume any particular structure of these functions except some mild regularity listed in Assumption \ref{ass:velocity_kernels_main}. It will be convenient to define
\begin{align*}
\mathcal{V}^1[u,v] &= V^1 + K^{1,1}\ast u + K^{1,2}\ast v,\\
\mathcal{V}^2[u,v] &= V^2 + K^{2,1}\ast u + K^{2,2}\ast v,
\end{align*}
so that the system \eqref{eq:general_cross_diffusion_intro} can be written as 
\begin{equation}\label{eq:general_cross_diffusion_intro_short_velocity}
\begin{split}
\partial_t u &= \partial_x(u\, \partial_x p(s)) + \partial_x(u \, \partial_x \mathcal{V}^1[u,v]),\\
\partial_t v &= \partial_x(v\, \partial_x p(s)) + \partial_x(v \, \partial_x \mathcal{V}^2[u,v]).
\end{split}
\end{equation}
If one would like to think about something simpler, the system
\begin{align*}
\partial_t u &= \partial_x(u\, \partial_x(u+v)) - \partial_x u,\\
\partial_t v &= \partial_x(v\, \partial_x(u+v)) + \partial_x v,
\end{align*}
i.e. $p(s) = s$, $\alpha = 2$, $K^{i,j} = 0$, $V^1(x)=-x$ and $V^2(x)=x$, is covered by the general theory developed in this paper and already presents several of the main difficulties that we need to address (though not all of them, since the case $\alpha>2$ is more involved; see the proof of Proposition~\ref{prop:the_case_of_support_on_the_line_identification_of_YM} for $\alpha\le 2$ in Section \ref{subsect:alternative_proof_alpha_at_most_3} and Section~\ref{subsect:the_case_that_YM_supp_on_a_nonzero_line} for the treatment of arbitrary~$\alpha$).\\

The target of this work is to construct weak solutions without imposing any structural conditions on $V^{i}$, $K^{i,j}$ or conditions guaranteeing that $u$, $v$ or $u+v$ are strictly positive. In particular, we do not make any assumption on the support of the initial conditions, the pressure law exponent and velocity fields (except mild regularity assumptions which could be probably improved). We also do not include the self-diffusion terms $\partial^2_x u$, $\partial^2_x v$ in \eqref{eq:general_cross_diffusion_intro_short_velocity}. \\

The standard approach to construct weak solutions is to consider some regularization of \eqref{eq:general_cross_diffusion_intro_short_velocity}, say by adding an artificial viscosity term 
\begin{equation}\label{eq:general_cross_diffusion_intro_viscosity}
\begin{split}
\partial_t \ueps &= \partial_x(\ueps\, \partial_x p(\seps)) + \partial_x(\ueps \, \partial_x \mathcal{V}^1[\ueps,\veps] ) + \eps\, \partial^2_x \ueps,\\
\partial_t \veps &= \partial_x(\veps\, \partial_x p(\seps)) + \partial_x(\veps \, \partial_x \mathcal{V}^2[\ueps,\veps]) + \eps\, \partial^2_x \veps,
\end{split}
\end{equation}
with $\seps=\ueps+\veps$ and study the limit $\eps \to 0$. The difficult terms are nonlinear ones, namely $\ueps\, \partial_x p(\seps)$, $\veps \, \partial_x p(\seps)$. Without advection (i.e. $\mathcal{V}^1 = \mathcal{V}^2 = 0$) the system has attracted a~large group of researchers and the understanding is now quite complete in arbitrary number of dimensions. In particular, there are several techniques to construct weak solutions: Aronson-Benilan estimates \cite{MR524760} can be used to obtain a uniform bound on $\partial^2_x p(\seps)$ \cite{MR4000848}, BV estimates on $\frac{\ueps}{\seps}$, $\frac{\veps}{\seps}$ coming from a transport equation satisfied by these quantities lead to strong compactness of $\ueps$ and $\veps$ in one dimension (probably this argument was used in \cite{MR2652018} for the first time but see \cite{MR4072681,MR3870087} as well) while direct energy considerations show that $\{\partial_x p(\seps)\}$ is strongly compact in $L^2((0,T)\times\Rd)$ for $\alpha \geq 1$ \cite{MR4712820, MR4179253, MR4880211}. In contrast, for a single equation rather than a system, the problem with advection is well understood \cite{MR4720218}.\\

The additional advection in \eqref{eq:general_cross_diffusion_intro_short_velocity} prevents application of any of these arguments and new ideas are needed. To the best of our knowledge, the following are the only results available in the literature for \eqref{eq:general_cross_diffusion_intro_short_velocity}. All of them hold only in one space dimension.
\begin{itemize}
\item \cite{MR3795211} constructs weak solutions to \eqref{eq:general_cross_diffusion_intro_short_velocity} with $\alpha>1$ without nonlocal interaction $K^{i,j}=0$. The initial conditions are initially separated and the Authors impose an inequality between the velocities $\partial_x V^1$ and $\partial_x V^2$, which prevents the two populations from mixing. 
\item \cite{meszaros2025cross} constructs for the first time weak solutions to \eqref{eq:general_cross_diffusion_intro_short_velocity} with $\alpha=1$ allowing arbitrary velocity field (nonlocal interactions are not explicitly mentioned there but it does not seem a difficult issue to include them in the analysis). The main drawback of the work is the assumption on the initial conditions $\log(u^0/v^0) \in BV(\mathbb{T})$ (so both $u^0$ and $v^0$ have to be supported on the same set). The main observation of \cite{meszaros2025cross} is that $f= \log(u/v)$ solves a transport equation so the BV regularity of $f$ is propagated in time. This gives strong convergence of $u$ and $v$ which is sufficient to conclude the proof of existence of solutions.
\item \cite{elbar2025cross} extends the work of \cite{meszaros2025cross} to the fast-diffusion case $\alpha\in(0,1)$ under similar assumptions on the initial conditions. 
\item Shortly before the submission of this preprint, the Authors of \cite{meszaros2025cross} posted a new work \cite{alpar_guy_newwork} in which the assumption $\log(u^0/v^0) \in BV(\mathbb{T})$ is removed, thereby allowing for general initial configurations in the regime $\alpha \leq 1$, which is a significant extension.
\end{itemize}
We remark that all these works cannot deal with the case $\alpha>1$ where the populations evolve with finite speed of propagation which is expected in mathematical biology as explained above. On the other hand, the work presented here covers the whole regime $\alpha \in (0,\infty)$ and all initial configurations by a completely different method than the one of \cite{alpar_guy_newwork, elbar2025cross, meszaros2025cross}. \\

In the present work we do not address the problem of obtaining uniform BV estimates for ${\ueps}$ and ${\veps}$, which appears to be a very challenging and worthwhile problem in its own right. Instead, we prove directly that the nonlinear terms $\ueps\, \partial_x p(\seps)$ and $\veps \, \partial_x p(\seps)$ converge weakly to $u\, \partial_x p(s)$ and $v\,\partial_x p(s)$, respectively; that is, we identify their weak limits. We use the approach of compensated compactness developed by Tartar and Murat in the context of hyperbolic conservation laws \cite{MR506997,MR584398}. This is formulated in the language of Young measures (introduced in Section \ref{subsect:young_measures}) which encode weak limits of nonlinear functions. The general idea is to derive additional PDEs satisfied by nonlinear functions of the solution, the so-called entropy-entropy flux pairs. Applying the Div-Curl Lemma (see Section~\ref{subsect:div-curl-lemma}) then yields constraints on the Young measure which, provided sufficiently many such pairs are available, allow one to deduce convergence. We remark that the approach has been successfully applied (among many others) to hyperbolic conservation laws \cite{MR684413,MR1284790, MR1383202}, mixed-type elliptic-hyperbolic systems \cite{MR2403603}, fast-reaction limits in reaction-diffusion systems \cite{MR4588320,MR4384332}, and regularization of forward-backward diffusion problems \cite{MR1299852,MR4384332}. Last but not least, we remark that Young measures are not only used to identify the limit, but also to define the concept of solutions (measure-valued solutions), where the limit cannot be identified \cite{MR4584297,MR2805464,MR3567640}.\\

In contrast to the aforementioned works, the present paper requires significant new developments in how the compensated compactness framework is exploited. The Young measure considered here captures oscillations of two sequences, ${\ueps/\seps}$ and ${\partial_x \seps^{\alpha}}$, which do not converge strongly and for which there is no simple relationship (for instance, via the divergence). Therefore, the system under consideration should be compared with $2\times 2$ hyperbolic conservation laws \cite{MR684413,MR1284790,MR1383202}, where the argument typically relies on countably many entropy-entropy flux pairs.
In our setting, such a family exists (Remark~\ref{rem:cons_law_uk+1_over_sk}), but only for $\alpha \leq 2$ (Remark~\ref{rem:integrability_of_derivative_sum_with_advection}), which excludes many cases of interest; in particular, large values of $\alpha$ are important in tumor-growth models \cite{MR3162474,MR429164}. Consequently, we carry out a careful analysis of the identities obtained via compensated compactness in order to complete the argument. A~more detailed sketch of the proof is provided in Section~\ref{sect:sketch_of_the_proof}. \\

We will make the following assumptions in all results throughout the paper. They are certainly not optimal. For instance, one could easily allow for time-dependent $V^i$ and $K^{i,j}$ by assuming additional time regularity. Depending on $\alpha$, one could also relax some assumptions on the kernels $K^{i,j}$ by using additional integrability information on the solution, see e.g. \cite{carrillo2024well} for $\alpha=1$.
\begin{Ass}\label{ass:velocity_kernels_main}
We assume that:
\begin{itemize}
\item  $\partial_x V^i$, $\partial_x K^{i,j}$, $\partial^2_x V^i$, $\partial^2_x K^{i,j}$  are in $L^{\infty}(\R)$,
\item the initial conditions $u^0, v^0 \in L^1(\R) \cap L^{\infty}(\R)$ and $u^0\, |x|, v^0\, |x| \in L^1(\R)$. 
\end{itemize}
\end{Ass}
The weak solutions to \eqref{eq:general_cross_diffusion_intro_short_velocity} that we construct are defined as follows. We interpret the terms $u\, \partial_x p(s)$, $v\, \partial_x p(s)$ as 
\begin{equation}\label{eq:pressure_term_rewritten_derivative_intro}
u\, \partial_x p(s) = u \, p'(s) \, \partial_x s = u\, s^{\alpha-2} \, \partial_x s = \frac{u}{s} \,  s^{\alpha-1} \, \partial_x s = \frac{1}{\alpha} \, \frac{u}{s} \, \partial_x s^{\alpha},
\end{equation}
since in our analysis it is useful to deal with derivatives of terms which are not singular (this is helpful mostly in the proof of Proposition \ref{prop:weak_compact_partial_x_seps}). Above, we always define $\frac{u}{s} = 0$ on the set $\{s = 0\}$. In fact, this does not matter because we show that on the set $\{s=0\}$ we have in fact $\partial_x s^{\alpha} = 0$ cf. \ref{conv:der_s_alpha_zero_in_vacuum} in Proposition \ref{prop:basic_convergences}.
\begin{Def}\label{def:weak_sol}
We say that a pair $u, v \in L^{\infty}(0,T; L^1(\R))$ with $s:=u+v$ is a weak solution to \eqref{eq:general_cross_diffusion_intro_short_velocity} with initial conditions $u^0, v^0 \in L^1(\R)$ if $\partial_x s^{\alpha} \in L^2((0,T)\times\R)$ and for all $\varphi, \phi \in C_c^{\infty}([0,T)\times\R)$ we have
$$
\int_0^T \int_{\R} u\, \partial_t \varphi \diff x \diff t + \int_{\R} u^0(x) \, \varphi(0,x) \diff x = \int_0^T \int_{\R} \left(\frac{1}{\alpha} \, \frac{u}{s}\, \partial_x s^{\alpha} + u \, \partial_x \mathcal{V}^1[u,v]\right)\, \partial_x \varphi \diff x \diff t,
$$
$$
\int_0^T \int_{\R} v\, \partial_t \phi \diff x \diff t + \int_{\R} v^0(x) \, \phi(0,x) \diff x = \int_0^T \int_{\R} \left(\frac{1}{\alpha} \, \frac{v}{s}\, \partial_x s^{\alpha} + v \, \partial_x \mathcal{V}^2[u,v]\right)\, \partial_x \phi \diff x \diff t.
$$
\end{Def}
The main result of the paper reads:
\begin{thm}\label{thm:main}
Under Assumption \ref{ass:velocity_kernels_main}, there exists a weak solution to \eqref{eq:general_cross_diffusion_intro_short_velocity} in the sense of Definition \ref{def:weak_sol} which satisfies $u, v \in L^{\infty}(0,T; L^p(\R))$ for all $p\in[1,\infty]$, $\partial_x s^{\frac{\alpha}{2}} \in L^2((0,T)\times\R)$ and $s\,|x| \in L^{\infty}(0,T; L^1(\R))$. It is obtained as a limit of solutions to \eqref{eq:general_cross_diffusion_intro_viscosity}. More precisely, there exists a subsequence (not relabelled) such that 
\begin{itemize}
\item $\ueps \weak u$, $\veps \weak v$ weakly/weakly$^*$ in $L^p((0,T)\times\R)$ for all $p\in[1,\infty]$,
\item $\seps \to s$ a.e. and strongly in $L^p((0,T)\times\R)$ for all $p\in[1,\infty)$,
\item $\partial_x \seps^{{\alpha}} \to \partial_x s^{{\alpha}}$ strongly in $L^2((0,T)\times\R)$,
\item $\ueps \to u$, $\veps \to v$ strongly in $L^p(\mathcal{S})$ for all $p\in[1,\infty)$ and a.e. on $\mathcal{S}$, where the set $
\mathcal{S}:= [0,T]~\times~\R \setminus \{(t,x): s(t,x) > 0 \mbox{ and } 
\partial_x\mathcal{V}^1[u,v] = \partial_x\mathcal{V}^2[u,v]\}
$.
\end{itemize}
\end{thm}
The same result can be proved on a periodic domain $\mathbb{T}$. It is also not difficult to extend this result to systems with additional growth terms of the form 
\begin{equation}\label{eq:general_cross_diffusion_intro_short_velocity_noncons}
\begin{split}
\partial_t u &= \partial_x(u\, \partial_x p(s)) + \partial_x(u \, \partial_x \mathcal{V}^1[u,v]) + u\, G^1(s),\\
\partial_t v &= \partial_x(v\, \partial_x p(s)) + \partial_x(v \, \partial_x \mathcal{V}^2[u,v]) + v\, G^2(s),
\end{split}
\end{equation}
where $G^i: [0,\infty)\to \R$ and $G^i \in L^{\infty}(\R)$. In tissue-growth processes, the additional source term acts to inhibit cell proliferation at high pressure levels \cite{MR3162474}. We stress, however, that we cannot handle the case where $G$ depends on the individual species $u$ and $v$, as our method does not yield strong compactness of $\{\ueps\}$ and $\{\veps\}$, unless the set $\mathcal{S}$ is the whole space $[0,T]\times \mathbb{R}$ (this is the case, for instance, when $\partial_x\mathcal{V}^1= -1$ and $\partial_x\mathcal{V}^2 = 1$). We briefly outline the necessary modifications required to extend Theorem \ref{thm:main} to equation \eqref{eq:general_cross_diffusion_intro_short_velocity_noncons} in Section~\ref{sect:extension_nonconservative}.   

\begin{rem}\label{rem:Darcy_law_weights}
We may also consider Darcy’s law \eqref{eq:darcy_law_introduction} with the pressure depending not on the total density $s = u + v$, but instead on a weighted sum $\widetilde{s} = c_u\,u + c_v\,v$. This formulation is motivated by the fact that pressure may reflect physical characteristics such as cell size. For instance, five cells of type $u$ may occupy the same volume as ten cells of type $v$; see \cite{MR4907570}. This case can be reduced to the standard setting $c_u = c_v = 1$. Indeed, if $u, v$ solve the system where the pressure depends on the weighted sum $\widetilde{s}$ with kernels $\widetilde{K^{i,j}}$ and velocities $\widetilde{V^i}$, then $c_u\,u, c_v\,v$ solve the original system~\eqref{eq:general_cross_diffusion_intro} with kernels ${K^{i,1}} = \frac{1}{c_u} \widetilde{K^{i,1}},
{K^{i,2}} = \frac{1}{c_v} \widetilde{K^{i,2}}$ (the velocity fields $\widetilde{V^i}$ do not change).
\end{rem}

{\bf Notation.} We use standard notation for function spaces. For $\Omega \subset \mathbb{R}$ and $1~\le~p,q~\le~\infty$, we denote by $L^p(0,T;L^q(\Omega))$ the usual Bochner-Lebesgue space. More generally, for an open set $U \subset \mathbb{R}^n$, we write $L^p(U)$ and $W^{1,p}(U)$ for the Lebesgue and Sobolev spaces, with negative-order spaces $W^{-1,p}(U)$ defined as a dual space of $W^{1,p'}_0(U)$ where $\frac{1}{p}+\frac{1}{p'}=1$; as usual, $H^k(U)=W^{k,2}(U)$. In particular, we shall use spaces on space–time domains such as $L^p((0,T)\times\mathbb{R})$ or $H^{-1}((0,T)~\times~\mathbb{R})$. If $1~<~p,q~<~\infty$, bounded sequences in these spaces are weakly compact; if $p=\infty$ and/or $q=\infty$, they are weakly$^*$ compact. In the case $p=1$ or $q=1$, boundedness is not sufficient and we use the Dunford–Pettis theorem \cite[Theorem~1.38]{MR1857292} to obtain weak compactness. We denote by $\mathcal{D}'((0,T)\times\R)$ the space of distributions and by $\mathcal{M}([0,T]\times\mathbb{R})$ the space of bounded Radon measures, identified with the dual of $C^0([0,T]\times\mathbb{R})$ (space of continuous functions vanishing at infinity). In this space, bounded sets are weakly$^*$ compact. We write $L^p(0,T;L^q_{\loc}(\Omega))$ for functions belonging to $L^p(0,T;L^q(K))$ for every compact $K\subset\Omega$. Local versions in time and/or space are defined analogously for Lebesgue, Sobolev (e.g. $H^{-1}_{\loc}((0,T)\times\mathbb{R})$), and measure spaces (e.g. $\mathcal{M}_{{\loc}}([0,T]\times\mathbb{R})$). Weak and strong convergence in these local spaces are understood in the usual sense, i.e. convergence holds in the corresponding space restricted to every compact subset. Whenever the sequence depends on the parameter $\eps$, we always assume that $\eps \in (0,1)$ and write $\{\ueps\}$, omitting the subscript (that is, we do not write $\{\ueps\}_{\eps \in (0,1)}$). Many sequences will converge only weakly, so the limits of nonlinear expressions cannot, in general, be identified with the nonlinear function of the limit; to denote such weak limits, we use the overline notation. For example, if $u_\eps \weak u$ weakly, we write $\overline{f(u)}$ to denote the weak limit of $\{f(u_\eps)\}$, which in general differs from $f(u)$; see Proposition \ref{prop:basic_convergences} for more examples. Finally, if $E$ is a Banach space, $E^*$ denotes its dual and $C(0,T; E)$ is a Banach space of continuous functions from $[0,T]$ to~$E$. 

{\bf Structure of the paper.} Section~\ref{sect:sketch_of_the_proof} contains a detailed sketch of the proof. Section~\ref{sect:a_priori_estimates_compactness} gathers the preparatory steps. We first present the standard a priori estimates for \eqref{eq:general_cross_diffusion_intro_viscosity} in Section~\ref{subsect:aprioriestimates}. In Section~\ref{subsect:basic_compactness} we establish the corresponding weak compactness results. Sections~\ref{subsect:cons_laws_all_alpha} and~\ref{subsect:cons_law_u2_over_s} introduce two new conservation laws, while in Section~\ref{subsect:dissipation_measure_gradient_squared} we quantify the defect measure $|\partial_x \seps^{\alpha} - \partial_x s^{\alpha}|^2$ as $\eps\to0$. Section~\ref{sect:young_measures_compensated_compactness} develops the more advanced tools required for the proof of Theorem~\ref{thm:main}. In Section~\ref{subsect:young_measures} we introduce Young measures while in Section~\ref{subsect:div-curl-lemma} we discuss the compensated compactness approach. In Section~\ref{subsect:existence_of_slns_in_lang_of_yms} we reduce the proof of the existence of a weak solution to showing that a certain integral involving the Young measure vanishes. The proof of this final statement together with the proof of Theorem \ref{thm:main} is given in Section~\ref{sect:proof_main_result} and is divided into three steps, carried out in Sections~\ref{subsect:support_is_a_line}, \ref{subsect:the_case_that_YM_supp_on_a_nonzero_line} and \ref{sect:main_proof_after_two_steps}. For $\alpha \leq 2$, the second step can be proved by a simpler argument, presented in Section~\ref{subsect:alternative_proof_alpha_at_most_3}. Finally, in Section~\ref{sect:extension_nonconservative}, we outline the necessary modifications required to extend the reasoning to the nonconservative case \eqref{eq:general_cross_diffusion_intro_short_velocity_noncons}.

\section{Sketch of the proof}\label{sect:sketch_of_the_proof}

First of all, in Section \ref{subsect:aprioriestimates} we obtain standard a priori estimates for solutions to \eqref{eq:general_cross_diffusion_intro_viscosity}. Namely, $\{\ueps\}$ and $\{\veps\}$ are bounded in $L^{\infty}(0,T; L^p(\R))$ for all $p \in [1,\infty]$, while $\{\partial_x \seps^{\alpha}\}$ is bounded in $L^2((0,T)\times\R)$. Then, in Section \ref{subsect:basic_compactness}, the Aubin-Lions lemma yields strong compactness of $\{\seps\}$ in $L^{p}((0,T)\times\R)$ for all $p \in [1,\infty)$. On the other hand, the Banach-Alaoglu theorem provides weak compactness of $\{\ueps\}$, $\{\veps\}$, and $\{\partial_x \seps^{\alpha}\}$ in their respective spaces. The remaining difficulty is to pass to the limit in the term $\frac{\ueps}{\seps}\, \partial_x \seps^{\alpha}$, which is the product of two sequences that converge only weakly.\\

It seems quite challenging to obtain better estimates that would justify the convergence directly. Therefore, we will try to identify the limit without deriving stronger estimates. We consider the Young measure $\pi_{t,x}$ of the sequence  $(\frac{\ueps}{\seps} - \frac{u}{s},\, \partial_x \seps^{\alpha} - \partial_x s^{\alpha})$, see Section~\ref{subsect:young_measures} for a rigorous explanation. For each $(t,x)$, this is a measure on $\R^2$, which is the set of possible values of the sequence. Roughly speaking, whenever $\{f(t,x,\frac{\ueps}{\seps} - \frac{u}{s},\, \partial_x \seps^{\alpha} - \partial_x s^{\alpha})\}$ is weakly compact in $L^1_{\loc}((0,T)\times\R)$, we have the representation (up to a subsequence)
\begin{equation}\label{eq:representation_weak_limits_sketch_of_the_proof_YMs}
f\left(t,x,\frac{\ueps}{\seps} - \frac{u}{s},\, \partial_x \seps^{\alpha} - \partial_x s^{\alpha}\right) \weak \int_{\R^2} f(t,x,\lambda_1, \lambda_2) \diff \pi_{t,x}(\lambda_1, \lambda_2) \mbox{ weakly in } L^1_{\loc}((0,T)\times\R).
\end{equation}
This type of difference appears naturally in studies of how much the sequences fail to converge strongly, see e.g. the defect measure in \cite{MR1034481}. To motivate our choice of the sequence, we write 
\begin{align*}
\frac{\ueps}{\seps}\, \partial_x \seps^{\alpha}- \frac{u}{s}\, \partial_x s^{\alpha} &= \left(\frac{\ueps}{\seps} - \frac{u}{s} \right) \, \left(\partial_x \seps^{\alpha} - \partial_x s^{\alpha} \right) + \,\left( \partial_x \seps^{\alpha} - \partial_x s^{\alpha}\right)\, \frac{u}{s} +\left(\frac{\ueps}{\seps}- \frac{u}{s} \right) \,  \partial_x s^{\alpha}\\
&\weak \int_{\R^2} \lambda_1\, \lambda_2 \diff \pi_{t,x}(\lambda_1,\lambda_2),
\end{align*}
where the last two terms converge to zero by weak convergence (some care is needed if $s=0$, but in that case we prove that $\partial_x s^{\alpha}=0$, see \ref{conv:der_s_alpha_zero_in_vacuum} in Proposition \ref{prop:basic_convergences}). Hence, the proof of existence is concluded once we show that $\int_{\R^2} \lambda_1\, \lambda_2 \diff \pi_{t,x}(\lambda_1,\lambda_2) = 0$ for a.e. $(t,x)$.\\

In order to prove this identity, we will use the equations satisfied by $\ueps$, $\veps$, and some nonlinear functions of them to obtain information on $\pi_{t,x}$. Roughly speaking, the product of weakly converging sequences is usually not the weak limit of the product, but certain combinations of products may have this property. The theory of compensated compactness collects results in this spirit; see Section \ref{subsect:div-curl-lemma} for a more in-depth discussion. For instance, consider two continuity equations 
$$
\partial_t U_{\eps}= \partial_x F_{\eps} + R_{\eps}, \qquad \partial_t V_{\eps} = \partial_x G_{\eps} + S_\eps \qquad \mbox{ on } (0,T)\times \R,
$$
with $\{R_{\eps}\}$ and $\{S_{\eps}\}$ compact in $H^{-1}_{\loc}((0,T)\times\R)$. Under some integrability assumptions, the celebrated Div-Curl Lemma implies that
$$
F_{\eps}\, V_{\eps} - U_{\eps}\, G_{\eps} \weak \overline{F}\, \overline{V} - \overline{U}\, \overline{G}
$$ 
in the sense of distributions, where $\overline{F}$, $\overline{V}$, $\overline{U}$, $\overline{G}$ are the weak limits in $L^2_{\loc}((0,T)\times\R)$ of $\{F_{\eps}\}$, $\{V_{\eps}\}$, $\{U_{\eps}\}$, $\{G_{\eps}\}$, respectively. Our strategy is to express all the weak limits in terms of the Young measure in order to obtain constraints on it.\\

The key observation proved in Section \ref{subsect:support_is_a_line} is that, at points where the integral does not vanish, the Young measure associated with the sequence is supported on at most a straight line passing through $(0,0)$. Intuitively, this means that oscillations (preventing strong compactness) of $\frac{\ueps}{\seps}$ and $\partial_x \seps^{\alpha}$ are correlated. This property follows from two observations:
\begin{itemize}
\item $\int_{\R^2} \lambda_1\,\lambda_2 \diff \pi_{t,x}(\lambda_1, \lambda_2)= \mathcal{F}\,\int_{\R^2} \lambda_1^2 \diff \pi_{t,x}(\lambda_1, \lambda_2)$, where $\mathcal{F}=\mathcal{F}(t,x)$ is defined in \eqref{eq:definition_quantity_F_necessary_later}. This follows from the compensated compactness approach applied directly to the two PDEs in \eqref{eq:general_cross_diffusion_intro_viscosity}.
\item $\int_{\R^2} \lambda_2^2 \diff \pi_{t,x}(\lambda_1, \lambda_2) \leq \mathcal{F}\, \int_{\R^2} \lambda_1\,\lambda_2 \diff \pi_{t,x}(\lambda_1, \lambda_2)$, which follows from the fact that the weak limit of $|\partial_x \seps^{\alpha}-\partial_x s^{\alpha}|^2$ coincides with the weak limit of $\mathcal{F}\,\left(\frac{\ueps}{\seps}\, \partial_x \seps^{\alpha}- \frac{u}{s}\, \partial_x s^{\alpha}\right)$; see Section~\ref{subsect:dissipation_measure_gradient_squared}. The latter essentially exploits the equation for the sum $\seps$, in which we can pass to the limit $\eps \to 0$ (the limit of every term can be identified). 
\end{itemize}
These two observations imply that there is equality in the Cauchy-Schwarz inequality 
$$
\int_{\R^2} \lambda_1^2 \diff \pi_{t,x}(\lambda_1, \lambda_2) \, \int_{\R^2} \lambda_2^2 \diff \pi_{t,x}(\lambda_1, \lambda_2) = \left| \int_{\R^2} \lambda_1 \, \lambda_2 \diff \pi_{t,x}(\lambda_1, \lambda_2) \right|^2,
$$
which implies that either $\int_{\R^2} \lambda_1^2 \diff \pi_{t,x}(\lambda_1, \lambda_2) = 0$ or $\lambda_2 = \mathcal{F}\, \lambda_1$. \\

In order to prove that $\int_{\R^2} \lambda_1\, \lambda_2 \diff \pi_{t,x}(\lambda_1,\lambda_2) = 0$, we only need to consider the case $\lambda_2 = \mathcal{F}\, \lambda_1$, where $\mathcal{F}$ is as above. This imposes a restriction on the support and has important consequences for the analysis of weak limits. By the representation \eqref{eq:representation_weak_limits_sketch_of_the_proof_YMs}, we have that the sequences
$$
\left\{ f\left(t,x,\frac{\ueps}{\seps} - \frac{u}{s},\, \partial_x \seps^{\alpha} - \partial_x s^{\alpha}\right) \right\}, \left\{ f\left(t,x,\frac{\ueps}{\seps} - \frac{u}{s}, \mathcal{F}\, \left(\frac{\ueps}{\seps} - \frac{u}{s}\right) \right) \right\},
$$
have the same weak limits whenever they are weakly compact in $L^1_{\loc}((0,T)\times\R)$. In particular, we can eliminate $\partial_x \seps^{\alpha}$ from the identities. We exploit this observation in Section \ref{subsect:the_case_that_YM_supp_on_a_nonzero_line}, together with a new conservation law from Section \ref{subsect:cons_laws_all_alpha}, in the form
\begin{equation}\label{eq:sketch_of_the_proof_first_cons_law}
\partial_t E_{\eps}- \partial_x \big((E_{\eps} + \seps)\,\partial_x p(\seps) + \ueps  \log \ueps\, \partial_x \mathcal{V}^1[\ueps, \veps] + \veps  \log \veps\, \partial_x \mathcal{V}^2[\ueps, \veps]\big) =g_{\eps} + h_{\eps},
\end{equation}
where $E_{\eps} =\ueps \log \ueps + \veps \log \veps$, $h_{\eps} \to 0$ in $L^2(0,T; H^{-1}(\R))$ and $\{g_{\eps}\}_{\eps \in (0,1)}$ is bounded in $L^1((0,T)\times\R)$. We apply compensated compactness again as explained above (although now the remainder is not compact in $H^{-1}_{\loc}((0,T)\times\R)$ and one has to use Theorem \ref{thm:div_curl_lemma_version_CR}). The resulting identity is technically challenging to analyze due to its nonlinear structure arising from the logarithmic terms. We therefore perform several decompositions and apply delicate estimates to the identity obtained via compensated compactness. Combining these with the relation $\lambda_2 = \mathcal{F}\, \lambda_1$, we conclude that the Young measure is in fact a Dirac mass concentrated at $(0,0)$. When $\alpha \leq 2$, we can use another conservation law (see Section \ref{subsect:cons_law_u2_over_s})
\begin{equation}\label{eq:sketch_of_the_proof_second_cons_law}
\partial_t \left(\frac{\ueps^2}{\seps}\right) - \partial_x\left(\frac{1}{\alpha}\,\frac{\ueps^2}{\seps^2}\, \partial_x \seps^{\alpha} \! + \! \frac{\ueps^2}{\seps}\,  \partial_x \mathcal{V}^1[\ueps, \veps]\! -\! \frac{\ueps^3}{3 \, \seps^2}\, \partial_x (\mathcal{V}^1[\ueps, \veps] \!-\! \mathcal{V}^2[\ueps, \veps])  \right)\! = \! g_{\eps} +  h_{\eps},
\end{equation}
where a small difference here is that $\{g_{\eps}\}$ is bounded only in $L^1(0,T; L^1_{\loc}(\R))$. The computations leading to the final conclusion, exploiting \eqref{eq:sketch_of_the_proof_second_cons_law}, are substantially simpler (see Section~\ref{subsect:alternative_proof_alpha_at_most_3} for this alternative argument). The final proof of Theorem \ref{thm:main} is presented in Section~\ref{sect:main_proof_after_two_steps}. \\

Finally, once we know that $\int_{\R^2} \lambda_1\,\lambda_2 \diff \pi_{t,x}(\lambda_1, \lambda_2)=0$, we can prove the convergences claimed in Theorem \ref{thm:main} by the usual energy considerations; see Section \ref{subsect:existence_of_slns_in_lang_of_yms}. Last but not least, we remark that this programme succeeds only because the Young measure was shown to be supported on a line, which substantially simplifies the identities under consideration. Without this observation, the computations would fail, since we would be unable to characterize the weak limits of joint nonlinear functions of $\frac{\ueps}{\seps}$ and $\partial_x \seps^{\alpha}$. 

\section{A priori estimates and some compactness}\label{sect:a_priori_estimates_compactness}

\subsection{A priori estimates}\label{subsect:aprioriestimates} First, we obtain a priori estimates for the solutions of~\eqref{eq:general_cross_diffusion_intro_viscosity}. While they are quite standard compared to the rest of the paper, they seem to be new in the fast-diffusion regime (particularly, the $L^{\infty}$ estimates). Moreover, one has to be careful as we work in the full range of exponents $\alpha \in (0, \infty)$ so for instance, when $\alpha \in (0,\frac{1}{3}]$, it is well-known that the energy $\int_{\R} \seps^{\alpha} \diff x$ is no longer integrable, nor is $\int_{\R} \seps\, |x|^2 \diff x$, which can be deduced by looking at the Barenblatt profile, see e.g. \cite[eq. (2.4)]{MR1986060}.
\begin{proposition}\label{prop:uniform_estimates_viscosity}
Let $\ueps$, $\veps$ be the solution of \eqref{eq:general_cross_diffusion_intro_viscosity}. Then, for all $T>0$, the following sequences are bounded: 
\begin{enumerate}[label=(E\arabic*)]
\item\label{est:moment} $\{\ueps |x|\}$, $\{\veps |x|\}$ in $L^{\infty}(0,T; L^1(\R))$, 
\item\label{est:gradient_from_entropy} $\{\partial_x \seps^{\frac{\alpha}{2}}\}$ in $L^2(0,T; L^2(\R))$,
\item\label{est:crucial_estimate_gradient_power_alpha} $\{\partial_x \seps^{\alpha}\}$ in $L^2(0,T; L^2(\R))$,
\item\label{est:conservation} $\{\ueps\}$, $\{\veps\}$ in $L^{\infty}(0,T; L^p(\R))$ for all $p\in [1,\infty]$,
\item\label{est:ulogu_vlogv} $\{\ueps \log \ueps\}$, $\{\veps \log \veps \}$ in $L^{\infty}(0,T; L^p(\R))$ for all $p\in [1,\infty]$,
\item\label{est:ulogu_vlogv_moment} $\{\ueps \log \ueps\,|x|^{\frac{1}{2}}\}$, $\{\veps \log \veps|x|^{\frac{1}{2}} \}$ in $L^{\infty}(0,T; L^1(\R))$ ,
\item\label{est:entropy_derivatives_u_v_eps_signular} $\{\sqrt{\eps}\, \partial_x \sqrt{\ueps}\}$, $\{\sqrt{\eps}\, \partial_x \sqrt{\veps}\}$ in $L^2((0,T)\times\R)$,
\item\label{est:energy_derivatives_u_v_eps_signular} $\{\sqrt{\eps} \, \partial_x \ueps\}$, $\{\sqrt{\eps}\, \partial_x \veps\}$ in $L^2((0,T)\times \R)$,
\item\label{est:partial_t_neg_ss} $\{\partial_t \ueps\}$, $\{\partial_t \veps\}$ in $L^2(0,T; H^{-1}(\R))$,
\item\label{est:partial_x_seps_without_any_power} $\{\partial_x \seps\}$ in $L^2((0,T)\times \R)$ for $\alpha\leq2$.
\end{enumerate}
\end{proposition}

We will need the following computation borrowed from \cite[Prop. 2.1]{elbar2025cross}. 
\begin{lem}\label{lem:estimate_derivative_seps_1_alpha}
Let $\alpha \in (0, \frac{1}{3}]$. Then, there is a constant $C=C(\alpha, \|s^0\|_{L^1_x}, T)$ such that 
$$
\int_0^T \int_{\R} |\partial_x \seps^{1-\alpha}| \diff x \diff t \leq C\, \| \partial_x \seps^{\frac{\alpha}{2}}  \|_{L^2_{t,x}}^{1 + \frac{1-3\alpha}{1+\alpha}}.
$$
\end{lem} 
\begin{proof}
We compute using the conservation of mass $\int_{\R} \seps(t,x) \diff x = \int_{\R} s^0(x) \diff x$:
\begin{align*}
&\int_0^T \int_{\R} \seps^{2-\alpha} \diff x \diff t \leq \|s^0\|_{L^1_x} \, \int_0^T \|\seps^{1-\alpha}\|_{L^{\infty}(\R)} \diff t \leq \int_0^T \int_{\R} |\partial_x \seps^{1-\alpha}| \diff x \diff t \\
&\leq  \|s^0\|_{L^1_x} \,(1-\alpha) \, \int_0^T \int_{\R} \seps^{-\alpha}\, |\partial_x \seps| \diff x \diff t 
\leq \|s^0\|_{L^1_x} \,\frac{2\,(1-\alpha)}{\alpha} \, \int_0^T \int_{\R} \seps^{1-\frac{3}{2}\alpha}\, |\partial_x \seps^{\frac{\alpha}{2}}| \diff x \diff t \\
&\leq \|s^0\|_{L^1_x} \,\frac{2\,(1-\alpha)}{\alpha} \, \left( \int_0^T \int_{\R} \seps^{2-3 \alpha} \diff x \diff t \right)^{\frac{1}{2}} \, \| \partial_x \seps^{\frac{\alpha}{2}}  \|_{L^2_{t,x}}.
\end{align*}
We apply Hölder's inequality with measure $\seps \diff x \diff t$ to estimate the first integral:
$$
 \int_0^T \int_{\R} \seps^{2-3 \alpha} \diff x \diff t \leq \left(\int_0^T \int_{\R} \seps^{2-\alpha} \diff x \diff t\right)^{\frac{1-3\alpha}{1-\alpha}} \, \left(T\, \int_{\R} s^0(x) \diff x  \right)^{\frac{2\alpha}{1-\alpha}}.
$$
Hence, since $1-\frac{1}{2}\,\frac{1-3\alpha}{1-\alpha} = \frac{1+\alpha}{2(1-\alpha)}$, we obtain
$$
\left(\int_0^T \int_{\R} \seps^{2-\alpha} \diff x \diff t  \right)^{\frac{1+\alpha}{2(1-\alpha)}} \leq C\, \| \partial_x \seps^{\frac{\alpha}{2}}  \|_{L^2_{t,x}}.
$$
We then use the first computation again to finally obtain
$$
\int_0^T \int_{\R} |\partial_x \seps^{1-\alpha}| \diff x \diff t \leq C\,  \left(\int_0^T \int_{\R} \seps^{2-\alpha} \diff x \diff t\right)^{\frac{1}{2}\,\frac{1-3\alpha}{1-\alpha}} \, \| \partial_x \seps^{\frac{\alpha}{2}}  \|_{L^2_{t,x}} \leq  C\, \| \partial_x \seps^{\frac{\alpha}{2}}  \|_{L^2_{t,x}}^{1 + \frac{1-3\alpha}{1+\alpha}}.
$$
\end{proof}
We will also need the following lemma proved in \cite[Lemma D.2]{carrillo2024stein}.
\begin{lem}\label{lem:control_negative_parts_of_log}
 Let $W:\R\to [0,\infty)$ be such that $e^{-W} \in L^1(\R)$. Then, for every measurable function $\rho: \R^d \to \R^+$ we have 
\begin{equation}\label{eq:pointwise_lower_bound}
\rho\, \log \rho + \rho\,W  \geq - \frac{2}{e} \, e^{-W/2}. 
\end{equation}
\end{lem}
\begin{proof}[Proof of Proposition \ref{prop:uniform_estimates_viscosity}]
We divide the proof into a few steps. As a preliminary observation, we note that by the conservation of mass $M:= \int_{\R} \seps(t,x) \diff x = \int_{\R} s^0(x) \diff x$, $\int_{\R} \ueps(t,x) \diff x = \int_{\R} u^0(x) \diff x$, $\int_{\R} \veps(t,x) \diff x = \int_{\R} v^0(x) \diff x$. In particular, Assumption \ref{ass:velocity_kernels_main} implies that all the sequences $\{\partial_x \mathcal{V}^1[\ueps,\veps]\}$, $\{\partial_x \mathcal{V}^2[\ueps,\veps]\}$, $\{\partial^2_x \mathcal{V}^1[\ueps,\veps]\}$, $\{\partial^2_x \mathcal{V}^2[\ueps,\veps]\}$ are uniformly bounded in the space $L^{\infty}((0,T)\times\R)$ and we denote the bound by $\| \partial_x \mathcal{V}^1\|_{L^{\infty}}$, $\| \partial_x \mathcal{V}^2\|_{L^{\infty}}$, $\| \partial^2_x \mathcal{V}^1\|_{L^{\infty}}$, $\| \partial^2_x \mathcal{V}^2\|_{L^{\infty}}$, respectively.\\

\underline{Step 1: Proof of \ref{est:moment}, \ref{est:gradient_from_entropy} and \ref{est:entropy_derivatives_u_v_eps_signular}.} We multiply equations in \eqref{eq:general_cross_diffusion_intro_viscosity} by $\log \ueps$ and $\log \veps$, respectively, and sum them up to obtain
\begin{equation*}
\begin{split}
&\partial_t \int_{\R} \ueps \log \ueps \diff x + \partial_t \int_{\R} \veps \log \veps \diff x + \int_{\R} p'(\seps)\, |\partial_x \seps|^2 \diff x + \eps\, \int_{\R} \frac{|\partial_x \ueps|^2}{\ueps} \diff x + \eps\, \int_{\R} \frac{|\partial_x \veps|^2}{\veps} \diff x \\ &=  - \int_{\R} \partial_x \mathcal{V}^1[\ueps,\veps] \, \partial_x \ueps \diff x - \int_{\R} \partial_x \mathcal{V}^2[\ueps,\veps] \, \partial_x \veps \diff x \leq (\| \partial_x^2 \mathcal{V}^1\|_{L^{\infty}} + \| \partial_x^2 \mathcal{V}^2\|_{L^{\infty}}) \, M.
\end{split}
\end{equation*}
We note that $p'(\seps)^{\frac{1}{2}}\, \partial_x \seps = \seps^{\frac{\alpha}{2}-1} \, \partial_x \seps = \frac{2}{\alpha} \partial_x \seps^{\frac{\alpha}{2}}$. Hence, we obtain
\begin{equation}\label{eq:entropy_bound_alpha_between_0_and_13}
\partial_t \int_{\R} \ueps \log \ueps \diff x + \partial_t \int_{\R} \veps \log \veps \diff x +\frac{4}{\alpha^2} \| \partial_x \seps^{\frac{\alpha}{2}} \|_{L^2_x}^2 + \eps\, \int_{\R} \frac{|\partial_x \ueps|^2}{\ueps} \diff x + \eps\, \int_{\R} \frac{|\partial_x \veps|^2}{\veps} \diff x \leq C,
\end{equation}
for a constant $C$ independent of $\eps$. Next, we sum up \eqref{eq:general_cross_diffusion_intro_viscosity} and we multiply by $|x|$ to obtain
\begin{equation}\label{eq:estimate_first_moment_for_alpha_at_most_one_third}
\begin{split}
\partial_t \int_{\R} \seps \,|x| \diff x =& - \int_{\R} \seps\, \partial_x p(\seps) \, \frac{x}{|x|} \diff x - \int_{\R} (\ueps\, \partial_x \mathcal{V}^1[\ueps,\veps]+ \veps\, \partial_x \mathcal{V}^2[\ueps,\veps]) \, \frac{x}{|x|} \diff x\\ & - \eps\int_{\R} \partial_x \seps \, \frac{x}{|x|} \diff x.
\end{split}
\end{equation}
We now consider the last two terms on the (RHS). For the first one
\begin{equation*}
\begin{split}
- \eps\int_{\R} \partial_x \seps \, \frac{x}{|x|} \diff x &= 2\, \eps\, \seps(t,0) \leq 2\, \eps \int_0^{\infty} |\partial_x \ueps| \diff x +2\,\eps \int_0^{\infty} |\partial_x \veps| \diff x \\
&\leq \frac{\eps}{2} \int_{\R} \frac{|\partial_x \ueps|^2}{\ueps} \diff x + \frac{\eps}{2} \int_{\R} \frac{|\partial_x \veps|^2}{\veps} \diff x + 2\,\eps \, M,
\end{split}
\end{equation*}
while for the term with the velocity fields
\begin{equation*}
 - \int_{\R} (\ueps\, \partial_x \mathcal{V}^1[\ueps,\veps]+ \veps\, \partial_x \mathcal{V}^2[\ueps,\veps]) \, \frac{x}{|x|} \diff x \leq (\| \partial_x \mathcal{V}^1\|_{L^{\infty}} + \| \partial_x \mathcal{V}^2\|_{L^{\infty}})\, M.
\end{equation*}
Hence, from \eqref{eq:estimate_first_moment_for_alpha_at_most_one_third}, we deduce
\begin{equation}\label{eq:final_identity_for_evol_1st_moment_before_splitting_2_cases}
\partial_t \int_{\R} \seps \,|x| \diff x \leq - \int_{\R} \seps\, \partial_x p(\seps) \, \frac{x}{|x|} \diff x + \frac{\eps}{2} \int_{\R} \frac{|\partial_x \ueps|^2}{\ueps} \diff x + \frac{\eps}{2} \int_{\R} \frac{|\partial_x \veps|^2}{\veps} \diff x  + C,
\end{equation}
where $C$ does not depend on $\eps \in (0,1)$. We will now combine \eqref{eq:entropy_bound_alpha_between_0_and_13} and \eqref{eq:final_identity_for_evol_1st_moment_before_splitting_2_cases} in order to conclude the proof. We need to split for two cases: $\alpha \leq 1$ and $\alpha > 1$.\\

\underline{Step 1.1: $\alpha \leq 1$.} Note that $ \seps\, \partial_x p(\seps) = \frac{1}{\alpha}\, \partial_x \seps^{\alpha}$ so that we have
$$
- \int_{\R} \seps\, \partial_x p(\seps) \, \frac{x}{|x|} \diff x = -\frac{1}{\alpha} \int_{x>0} \partial_x \seps^{\alpha} \diff x + \frac{1}{\alpha} \int_{x<0} \partial_x \seps^{\alpha} \diff x = \frac{2}{\alpha} \seps^{\alpha}(t,0).
$$
Observe that $\alpha \leq 1$ implies $\alpha \leq \frac{1+\alpha}{2}$ so we can estimate
$$
 \frac{2}{\alpha} \seps^{\alpha}(t,0) \leq \frac{2}{\alpha}  + \frac{2}{\alpha} \seps^{\frac{1+\alpha}{2}}(t,0) \leq  \frac{2}{\alpha}  + \frac{2}{\alpha} \int_{\R} |\partial_x \seps^{\frac{1+\alpha}{2}}| \diff x \leq \frac{2}{\alpha}  + \frac{2\,(1+\alpha)}{\alpha^2} \int_{\R} \seps^{\frac{1}{2}}\, |\partial_x \seps^{\frac{\alpha}{2}}| \diff x.
$$
Hence, integrating in time the estimates above and applying Cauchy-Schwarz inequality with a parameter $\delta$:
\begin{equation*}
\begin{split}
-\! \int_0^t \int_{\R} \seps\, \partial_x p(\seps) \, \frac{x}{|x|} \diff x \diff \tau &\leq 
\frac{2\,t}{\alpha} + \frac{2\,(1+\alpha)}{\alpha^2} \int_0^t \int_{\R} \seps^{\frac{1}{2}}\, |\partial_x \seps^{\frac{\alpha}{2}}| \diff x \diff \tau  \\
&\leq C(t,\alpha, M, \delta) + \delta \, \| \partial_x \seps^{\frac{\alpha}{2}}  \|_{L^2((0,t)\times\R)}^{2}.
\end{split}
\end{equation*}
Plugging this estimate into \eqref{eq:final_identity_for_evol_1st_moment_before_splitting_2_cases}, choosing $\delta = \frac{2}{\alpha^2}$, and integrating in time yields
\begin{equation}\label{eq:first_moment_estimate_after_all_terms_being_estimated}
 \int_{\R} \seps(t,x) \,|x| \diff x  \leq C + \frac{2}{\alpha^2}\, \| \partial_x \seps^{\frac{\alpha}{2}}  \|_{L^2((0,t)\times\R)}^2 + \frac{\eps}{2}\int_0^t \int_{\R} \left(\frac{|\partial_x \ueps|^2}{\ueps} + \frac{|\partial_x \veps|^2}{\veps}\right) \diff x \diff \tau
\end{equation} 
for some new constant $C=C(\alpha, T, M, \int_{\R} s^0 \,|x| \diff x)$. Integrating \eqref{eq:entropy_bound_alpha_between_0_and_13} in time and summing up with \eqref{eq:first_moment_estimate_after_all_terms_being_estimated} we deduce
\begin{equation}\label{eq:ulogu_vlogv_alpha_leq_13_final_estimate}
\begin{split}
\int_{\R} \ueps(t,x) \log & \, \ueps(t,x) \diff x + \int_{\R} \veps(t,x) \log \veps(t,x) \diff x + \int_{\R} \seps(t,x) \,|x| \diff x \\
&+\frac{2}{\alpha^2}\, \| \partial_x \seps^{\frac{\alpha}{2}} \|_{L^2((0,t)\times\R)}^2 + \frac{\eps}{2}\, \int_0^t \int_{\R} \left(\frac{|\partial_x \ueps|^2}{\ueps} + \frac{|\partial_x \veps|^2}{\veps}\right) \diff x \diff \tau  \\
&\leq C + \int_{\R} u^0(x) \log u^0(x) \diff x + \int_{\R} v^0(x) \log v^0(x) \diff x .
\end{split}
\end{equation}
By Lemma \ref{lem:control_negative_parts_of_log}, $\int_{\R} \left(\ueps(t,x)\log \ueps(t,x) +  \ueps(t,x)\,\frac{|x|}{2}\right) \diff x$ is bounded from below by a numerical constant and analogously for $\veps$. This concludes the proof.\\

\underline{Step 1.2: $\alpha > 1$.} We sum up \eqref{eq:general_cross_diffusion_intro_viscosity}, we multiply by $p(\seps)$ and we use Cauchy-Schwarz inequality to obtain
\begin{equation}\label{eq:identity_estimates_alpha_geq_1_energy}
\begin{split}
\frac{\alpha}{\alpha-1} \partial_t \int_{\R} \seps^{\alpha} \diff x + \int_{\R} \seps\, |\partial_x p(\seps)|^2 \diff x = - \!\int_{\R}(\ueps \, \partial_x \mathcal{V}^1[\ueps,\veps]+\veps \, \partial_x \mathcal{V}^2[\ueps,\veps]) \partial_x p(\seps) \diff x\\
\leq \frac{1}{2}(\| \partial_x \mathcal{V}^1\|^2_{L^{\infty}} + \| \partial_x \mathcal{V}^2\|^2_{L^{\infty}})\, M + \frac{1}{2} \, \int_{\R} \seps\, |\partial_x p(\seps)|^2 \diff x.
\end{split}
\end{equation}
It follows that $\{\sqrt{\seps}\, \partial_x p(\seps)\}$ is bounded in $L^2((0,T)\times\R)$. Combining with the bound on $\{\sqrt{\seps}\}$ in $L^2((0,T)\times\R)$, we deduce that the integral $\int_0^T\int_{\R} \left|\seps\, \partial_x p(\seps) \, \frac{x}{|x|} \right| \diff x \diff \tau$ is uniformly bounded, so from \eqref{eq:final_identity_for_evol_1st_moment_before_splitting_2_cases} we obtain a bound on $\{\seps\, |x|\}$ in $L^{\infty}(0,T; L^1(\R))$. The remaining assertions follow now from \eqref{eq:entropy_bound_alpha_between_0_and_13} together with Lemma \ref{lem:control_negative_parts_of_log}, which allows us to control the negative part of the logarithm. \\

\underline{Step 2. Proof of \ref{est:conservation}.} Let $M({\gamma})= \sup_{t\in [0,T]} \left(\int_{\R} \seps^{\gamma} \diff x\right)^{\frac{1}{\gamma}}$. We will first prove that for some constant $C=C(M, \alpha, T)$ independent of $\gamma$ and all $\gamma \geq 15\,\alpha + 2$, $t\in[0,T]$
\begin{equation}\label{eq:Alikakos_estimate_to_run_iterations}
M(\gamma)^{\gamma} \leq \int_{\R} (s^0)^{\gamma} \diff x + C +C\,\gamma^{5} \, M\left( \frac{\gamma}{2}\right)^{\gamma}.
\end{equation}
Having \eqref{eq:Alikakos_estimate_to_run_iterations}, it is well-known that Alikakos' iterations yield the $L^{\infty}$ bounds \cite{Alikakos1979} and the final claim follows by the $L^p$ interpolation. To see that \eqref{eq:Alikakos_estimate_to_run_iterations} implies the $L^{\infty}$ bounds, we first observe that we may assume 
$$
\int_{\R} (s^0)^{\gamma} \diff x + C \leq  C\, \gamma^{5} \,M\left(\frac{\gamma}{2}\right)^{\gamma}
 $$
 for all $\gamma \geq \gamma^*$ for some $\gamma^* \geq 1$. Otherwise, there exists a sequence $\{ \gamma_k\}_{k \geq 0}$ such that $\gamma_k \to +\infty$ as $k \to +\infty$ so that
    $$
\int_{\R} (s^0)^{\gamma_k} \diff x + C \geq  C\, \gamma_k^{5} \,M\left(\frac{\gamma_k}{2}\right)^{\gamma_k}
    $$
from which the uniform boundedness of $\{\seps\}$ follows. Hence, for all $\gamma \geq \max(15\alpha+2,\gamma^*)$
$$
M(\gamma) \leq  (2\,C)^{\frac{1}{\gamma}}\, \gamma^{\frac{5}{\gamma}} \,M\left(\frac{\gamma}{2}\right) = (\tilde{C}\, \gamma)^{\frac{5}{\gamma}} \,M\left(\frac{\gamma}{2}\right)
$$
for a new constant $\tilde{C}$. Let $i \geq 1$ such that $2^i \geq \max(p^*,15\alpha+2)$. Then
   $$
        M( 2^k)  \leq  ( \tilde{C}\, 2^{k} )^{5/{2^k}} M(2^{k-1})  \leq ( \tilde{C}\, 2^{k} )^{5/{2^k}}  ( \tilde{C}\, 2^{(k-1)} )^{5/{2^{k-1}}} M(2^{k-2}) \leq M(2^{i-1}) \prod_{j=i}^k ( \tilde{C}\, 2^{j})^{5/2^{j}} .
   $$
    It follows that $M(2^k)$ is bounded independently of $k$ because
    $$
\lim_{k \to \infty} \prod_{j=i}^k ( \tilde{C}\, 2^{j})^{5/2^{j}} < \infty \iff \sum_{j=i}^\infty \frac{5 \log (\tilde{C}\, 2^{j})}{2^j} < \infty.
    $$
Now, we prove \eqref{eq:Alikakos_estimate_to_run_iterations}. We multiply the PDE for $\seps$ by $\seps^{\gamma-1}$
\begin{align*}
\partial_t \int_{\R} \seps^{\gamma} \diff x +\gamma\, (\gamma-1) \int_{\R}& \seps^{\alpha+\gamma-3}\,|\partial_x \seps|^{2} \diff x\\
& \leq   -\gamma \int_{\R} (\ueps\, \partial_x \mathcal{V}^1[\ueps,\veps]+ \veps\, \partial_x \mathcal{V}^2[\ueps,\veps]) \, \partial_x \seps^{\gamma-1} \diff x\\
&\leq \gamma\,(\gamma-1)\,(\| \partial_x \mathcal{V}^1\|_{L^{\infty}} + \| \partial_x \mathcal{V}^2\|_{L^{\infty}})\, \int_{\R} \seps^{\gamma-1} \, |\partial_x \seps| \diff x,
\end{align*}
where we neglected the term $\eps \int_{\R} \partial_x \seps\, \partial_x \seps^{\gamma-1} \diff x \geq 0$. We write $\seps^{\gamma-1} = \seps^{\frac{\alpha+\gamma-3}{2}}\, \seps^{\frac{-\alpha+\gamma+1}{2}}$ so applying Cauchy-Schwarz with a parameter $\delta_1>0$ to be chosen
\begin{multline*}
(\| \partial_x \mathcal{V}^1\|_{L^{\infty}} + \| \partial_x \mathcal{V}^2\|_{L^{\infty}}) \int_{\R} \seps^{\gamma-1} \, |\partial_x \seps| \diff x \leq \\ \leq \frac{\delta_1\,(\| \partial_x \mathcal{V}^1\|_{L^{\infty}} + \| \partial_x \mathcal{V}^2\|_{L^{\infty}})^2}{2} \int_{\R} \seps^{\alpha+\gamma-3}\,|\partial_x \seps|^{2} \diff x + \frac{1}{2\,\delta_1}\int_{\R} \seps^{-\alpha+\gamma+1} \diff x.
\end{multline*}
Choosing $\delta_1$ so that $\frac{\delta_1}{2}\,(\| \partial_x \mathcal{V}^1\|_{L^{\infty}} + \| \partial_x \mathcal{V}^2\|_{L^{\infty}})^2 \leq \frac{1}{2}$ we obtain
\begin{equation}\label{eq:evolution_of_the_gamma_moment_before_doing_the_interpolation_without_the_rubbish}
\partial_t \int_{\R} \seps^{\gamma} \diff x  + \frac{\gamma\, (\gamma-1)}{2} \int_{\R} \seps^{\alpha+\gamma-3}\,|\partial_x \seps|^{2} \diff x \leq \frac{\gamma\,(\gamma-1)}{2\,\delta_1} \int_{\R} \seps^{-\alpha+\gamma+1} \diff x.
\end{equation}
We claim that \eqref{eq:evolution_of_the_gamma_moment_before_doing_the_interpolation_without_the_rubbish} implies that $\{ \seps \}$ is bounded in $L^{\infty}(0,T; L^q(\R))$ for all $q\in[1,\infty)$ (so that $M(\gamma)$ is finite for all $\gamma\geq 1$). We split the proof for two cases. 
\begin{itemize}
\item If $\alpha\leq \frac{1}{3}$, we have by Lemma \ref{lem:estimate_derivative_seps_1_alpha} and \ref{est:gradient_from_entropy} that $\{\seps^{1-\alpha}\}$ is bounded in $L^1(0,T; L^{\infty}(\R))$ so we estimate $\int_{\R} \seps^{-\alpha+\gamma+1} \diff x \leq \| \seps^{1-\alpha} \|_{L^{\infty}(\R)} \int_{\R} \seps^{\gamma} \diff x$ and we apply Gr\"onwall's inequality.
\item If $\alpha > \frac{1}{3}$, we use that $\{\partial_x \seps^{\frac{\alpha+1}{2}}\}$ is bounded in $L^2(0,T; L^1(\R))$ by \ref{est:gradient_from_entropy} and $\partial_x \seps^{\frac{\alpha+1}{2}} = \frac{\alpha+1}{2}\, \seps^{\frac{\alpha-1}{2}}\, \partial_x \seps=\frac{\alpha+1}{\alpha}\, \seps^{\frac{1}{2}} \, \partial_x \seps^{\frac{\alpha}{2}}$ so that $\{ \seps^{\frac{\alpha+1}{2}}\}$ is bounded in $L^2(0,T; L^{\infty}(\R))$ and Hölder's inequality with measure $\seps \diff x$ implies
$$
\int_{\R} \seps^{-\alpha+\gamma+1} \diff x \leq \|  \seps^{\frac{\alpha+1}{2}} \|_{L^{\infty}(\R)} \int_{\R} \seps^{\gamma + \frac{1-3\alpha}{2} } \diff x \leq 
\|  \seps^{\frac{\alpha+1}{2}} \|_{L^{\infty}(\R)} \left( \int_{\R} \seps^{\gamma} \diff x \right)^{\kappa}\, M^{1-\kappa},
$$
where $\kappa = \frac{\gamma + \frac{-1-3\alpha}{2}}{\gamma-1}\in(0,1)$. Then, \eqref{eq:evolution_of_the_gamma_moment_before_doing_the_interpolation_without_the_rubbish}   gives that $\partial_t \left(\int_{\R} \seps^{\gamma} \diff x\right)^{1-\kappa} \leq C\, \|  \seps^{\frac{\alpha+1}{2}} \|_{L^{\infty}(\R)}$ for some constant $C$ which after integrating in time proves the claim.
\end{itemize}
In particular, we deduce that for all $\alpha \in (0,\infty)$, 
\begin{equation}\label{eq:claim_s_to_power_3_halves_is_in_l2}
\{\seps^{\frac{3}{2}}\} \mbox{ is bounded in } L^2((0,T)\times\R).
\end{equation}
Next, we come back to \eqref{eq:evolution_of_the_gamma_moment_before_doing_the_interpolation_without_the_rubbish} and we estimate
\begin{equation}\label{eq:estimate_Linf_bound_integral_-alph+gamma_first_splitting}
 \int_{\R} \seps^{-\alpha+\gamma+1} \diff x = \int_{\R} \seps^{\frac{\gamma+\alpha}{2} + 1}\, \seps^{\frac{\gamma-3\alpha}{2}} \diff x \leq \| \seps^{\frac{\gamma+\alpha}{2} + 1} \|_{L^{\infty}(\R)} \, \int_{\R} \seps^{\frac{\gamma-3\alpha}{2}} \diff x.
\end{equation}
The term $\| \seps^{\frac{\gamma+\alpha}{2} + 1} \|_{L^{\infty}(\R)}$ can be controlled by means of the dissipation. Indeed, 
\begin{equation}\label{eq:estimate_L_inf_norm_term_half_of_gamma_via_dissipation}
\| \seps^{\frac{\gamma+\alpha}{2} + 1} \|_{L^{\infty}(\R)} \leq  \| \partial_x \seps^{\frac{\gamma+\alpha}{2} + 1} \|_{L^{1}(\R)} \leq \left(\frac{\gamma+\alpha}{2} + 1\right) \|\seps^{\frac{3}{2}} \|_{L^2(\R)} \, \| \seps^{\frac{\gamma+\alpha-3}{2}}\, \partial_x \seps \|_{L^2(\R)},
\end{equation}
Therefore, as $\gamma \geq \alpha$, $\gamma \geq 1$ we have $\frac{\gamma+\alpha}{2} + 1 \leq 2\,\gamma$ so that from \eqref{eq:estimate_Linf_bound_integral_-alph+gamma_first_splitting} we deduce
\begin{align*}
 \int_{\R} \seps^{-\alpha+\gamma+1} \diff x &\leq 2 \, \gamma\,  \| \seps^{\frac{3}{2}} \|_{L^2(\R)} \, \| \seps^{\frac{\gamma+\alpha-3}{2}}\, \partial_x \seps \|_{L^2(\R)} \, \int_{\R} \seps^{\frac{\gamma-3\alpha}{2}} \diff x \\
&\leq \frac{\delta_1}{2}  \| \seps^{\frac{\gamma+\alpha-3}{2}}\, \partial_x \seps \|_{L^2(\R)}^2 + \frac{2\, \gamma^2}{\delta_1}\, \| \seps^{\frac{3}{2}} \|^2_{L^2(\R)} \, \left( \int_{\R} \seps^{\frac{\gamma-3\alpha}{2}} \diff x \right)^2
\end{align*}
for the same $\delta_1$ as above. Coming back to \eqref{eq:evolution_of_the_gamma_moment_before_doing_the_interpolation_without_the_rubbish}, we obtain
\begin{equation}\label{eq:evolution_gamma_moment_last_step_before_Holder}
\partial_t \int_{\R} \seps^{\gamma} \diff x  \leq \frac{\gamma^4}{\delta_1^2} \, \|\seps^{\frac{3}{2}} \|^2_{L^2(\R)} \, \left( \int_{\R} \seps^{\frac{\gamma-3\alpha}{2}} \diff x \right)^2.
\end{equation}
We finally use Hölder's inequality with respect to the measure $\seps \diff x$ to obtain
\begin{equation}\label{eq:Holder_to_adjust_the_exponent_to_gamma_over_2_for_alikakos}
\int_{\R} \seps^{\frac{\gamma-3\alpha}{2}} \diff x \leq \left(\int_{\R} \seps^{\frac{\gamma}{2}} \diff x\right)^{\beta} \, M^{1-\beta}, \qquad \qquad \beta=\frac{\frac{\gamma-3\alpha}{2}-1}{\frac{\gamma}{2}-1},
\end{equation}
where $\beta \in (\frac{4}{5},1]$ since $\gamma \geq 15\alpha + 2$. From \eqref{eq:evolution_gamma_moment_last_step_before_Holder} we deduce
\begin{align*}
\partial_t \int_{\R} \seps^{\gamma} \diff x &\leq \frac{\gamma^4}{\delta_1^2} \, \| \seps^{\frac{3}{2}} \|^2_{L^2(\R)} \, \left( \int_{\R} \seps^{\frac{\gamma}{2}} \diff x \right)^{2\beta}\, M^{2(1-\beta)}\\
&\leq 
\frac{M^{2(1-\beta)}}{\delta_1^2} \,\| \seps^{\frac{3}{2}} \|^2_{L^2(\R)} \, \left(1 +  \gamma^{\frac{4}{\beta}} \,\left( \int_{\R} \seps^{\frac{\gamma}{2}} \diff x \right)^{2}\right).
\end{align*}
We finally observe that $M^{2(1-\beta)} \leq 1+ M^2$ and $\gamma \geq 1$ and $\beta \geq \frac{4}{5}$ imply $\gamma^{\frac{4}{\beta}} \leq \gamma^5$ so that
\begin{equation}\label{eq:final_bound_to_start_alikakos}
\partial_t \int_{\R} \seps^{\gamma} \diff x \leq \frac{1+M^{2}}{\delta_1^2}\,\|\seps^{\frac{3}{2}} \|^2_{L^2(\R)} \, \left(1 +  \gamma^{5} \,\left( \int_{\R} \seps^{\frac{\gamma}{2}} \diff x \right)^{2}\right).
\end{equation}
We integrate in time and letting $C = \frac{1+M^{2}}{\delta_1^2}\,\|\seps^{\frac{3}{2}} \|^2_{L^2((0,T)\times \R)}$ (which is uniformly bounded in $\eps$ by \eqref{eq:claim_s_to_power_3_halves_is_in_l2}) we obtain \eqref{eq:Alikakos_estimate_to_run_iterations} concluding the proof.\\

\underline{Step 3. Proof of \ref{est:ulogu_vlogv}, \ref{est:ulogu_vlogv_moment}.} Clearly, $\{ \ueps \log \ueps \}$, $\{\veps \log \veps\}$ are bounded in $L^{\infty}((0,T)\times \R)$ by \ref{est:conservation}. Hence, to obtain \ref{est:ulogu_vlogv} we need to prove that $\{ \ueps \log \ueps \}$, $\{\veps \log \veps\}$ are bounded in $L^{\infty}(0,T; L^1(\R))$. We observe that by Lemma \ref{lem:control_negative_parts_of_log} with $|y|^- = -\min(0,y)$
\begin{multline*}
\int_{\R} \ueps |\log \ueps| \diff x = \int_{\R} \ueps \log \ueps \diff x + 2 \int_{\R} \ueps |\log \ueps|^- \diff x \leq \\ \leq \int_{\R} \ueps \log \ueps \diff x + 2\left( \int_{\R} \ueps \, |x| \diff x + \frac{2}{e} \int_{\R} e^{-|x|} \diff x\right).
\end{multline*}
The same inequality holds for $\veps$. Therefore, the claim follows from \eqref{eq:entropy_bound_alpha_between_0_and_13} and \ref{est:moment}.\\

To prove \ref{est:ulogu_vlogv_moment}, we split the integral for three sets as follows:
\begin{multline*}
\int_{\R}  \ueps |\log \ueps| |x|^{\frac{1}{2}} \diff x \leq \int_{\ueps \geq 1}  \ueps |\log \ueps| |x|^{\frac{1}{2}} \diff x+  \int_{\ueps \in [e^{-|x|^{1/2}}, 1]}  \ueps |\log \ueps| |x|^{\frac{1}{2}} \diff x  \\+ \int_{\ueps \leq e^{-|x|^{1/2}}}  \ueps |\log \ueps| |x|^{\frac{1}{2}} \diff x =: I_1 + I_2 + I_3.
\end{multline*}
For $I_1$, we can estimate $|\log \ueps| \leq 1+ \ueps \leq 1+\|\ueps\|_{L^{\infty}}$ by \ref{est:conservation} and $|x|^{\frac{1}{2}} \leq 1+|x|$ so we can use \ref{est:moment} to control the integral. For $I_2$, the set of integration implies $|\log \ueps| \leq |x|^{\frac{1}{2}}$ so we can use again \ref{est:moment} to control the term. Finally, to bound $I_3$, we write $\ueps |\log \ueps| |x|^{\frac{1}{2}} = \sqrt{\ueps}\, \sqrt{\ueps} |\log \ueps|\,  |x|^{\frac{1}{2}}   \leq C\, |x|^{\frac{1}{2}}\, e^{-|x|^{\frac{1}{2}}/2}$ where $C = \sup_{x\in[0,1]} \sqrt{x}\, |\log x|$. The proof for $\veps$ is the same.\\

\underline{Step 4. Proof of \ref{est:crucial_estimate_gradient_power_alpha} and \ref{est:energy_derivatives_u_v_eps_signular}.} For \ref{est:crucial_estimate_gradient_power_alpha}, we observe that 
$$
\partial_x \seps^{\alpha} = \alpha\, \seps^{\alpha-1} \, \partial_x \seps = \alpha\, \seps^{\frac{\alpha}{2}} \, \seps^{\frac{\alpha}{2}-1} \, \partial_x \seps = 2\,\seps^{\frac{\alpha}{2}}\, \partial_x \seps^{\frac{\alpha}{2}}
$$ 
so we conclude directly by \ref{est:gradient_from_entropy} and \ref{est:conservation}. For \ref{est:energy_derivatives_u_v_eps_signular}, we have 
$$
|\partial_x \ueps| \leq \sqrt{\ueps}\, \frac{\partial_x \ueps}{\sqrt{\ueps}}\leq \|\seps\|_{L^{\infty}((0,T)\times\R)}^{\frac{1}{2}}\, \frac{\partial_x \ueps}{\sqrt{\ueps}}
$$
and we conclude by \ref{est:conservation} and \ref{est:entropy_derivatives_u_v_eps_signular}.

\underline{Step 5. Proof of \ref{est:partial_t_neg_ss}.} We have $\partial_t \ueps = \partial_x F_{\eps}$ where 
$$
F_{\eps} = \ueps \, \partial_x p(\seps) + \ueps\, \partial_x \mathcal{V}^1[\ueps,\veps]+ \eps\, \partial_x \ueps
$$ and it is sufficient to prove that $\{F_{\eps}\}$ is uniformly bounded in $L^2((0,T)\times\Omega)$. We have $|\ueps\, \partial_x p(\seps)|\leq  |\seps\, \partial_x p(\seps)| = \frac{1}{\alpha} |\partial_x \seps^{\alpha}|$ so the claim follows from \ref{est:crucial_estimate_gradient_power_alpha}, \ref{est:conservation}, \ref{est:energy_derivatives_u_v_eps_signular} and from Assumption \ref{ass:velocity_kernels_main} on the velocity fields. The reasoning for $\partial_t \veps$ is analogous.

\underline{Step 6. Proof of \ref{est:partial_x_seps_without_any_power}.} Since $\alpha \leq 2$ we use that $\partial_x \seps = \frac{2}{\alpha}\,\seps^{1-\frac{\alpha}{2}}\,\partial_x \seps^{\frac{\alpha}{2}}$ so $\{\partial_x \seps\}$ is bounded in $L^2((0,T)\times\R)$ by \ref{est:gradient_from_entropy} and \ref{est:conservation}.
\end{proof}
\subsection{Basic compactness}\label{subsect:basic_compactness} We extract here a subsequence which will enjoy several compactness properties. This is a direct consequence of Proposition \ref{prop:uniform_estimates_viscosity}.
\begin{proposition}\label{prop:basic_convergences}
There exists a subsequence (not relabelled) such that
\begin{enumerate}[label=(C\arabic*)]
\item\label{conv:sum_strong} $\seps \to s$ a.e. and strongly in $L^p((0,T)\times\R)$ for all $p\in [1,\infty)$,
\item\label{conv:u_v_weak} $\ueps \weak u$, $\veps \weak v$ weakly/weakly$^*$ in $L^p((0,T)\times\R)$ for all $p\in [1,\infty]$,
\item\label{conv:gradient_sum} $\partial_x \seps^{\alpha} \weak \partial_x s^{\alpha}$, $\partial_x \seps^{\frac{\alpha}{2}} \weak \partial_x s^{\frac{\alpha}{2}}$ weakly in $L^2((0,T)\times\R)$,
\item\label{conv:time_derivative} $\partial_t \seps \weak \partial_t s$ weakly in $L^2(0,T; H^{-1}(\R))$,
\item\label{conv:velocity_fields} $\partial_x \mathcal{V}^1[\ueps,\!\veps] \! \to\! \partial_x \mathcal{V}^1[u,\!v], \partial_x\mathcal{V}^2[\ueps,\!\veps]\! \to \! \partial_x \mathcal{V}^2[u,\!v]$ a.e. and strongly in $L^{p}_{\loc}((0,T)\!\times\!\R)$ for $p\in[1,\infty)$,
\item\label{conv:der_s_alpha_zero_in_vacuum} $ |\partial_x \seps^{\alpha} - \partial_x s^{\alpha}|\, \mathds{1}_{s=0} \to 0$ strongly in $L^1(0,T; L^1_{\loc}(\R))$ and $\partial_x s^{\alpha} = 0$ a.e. on $\{s = 0\}$,
\item\label{conv:integrated_time_der_sum} $\int_{-\infty}^x \partial_t \seps \diff y \weak \int_{-\infty}^x \partial_t s \diff y$ weakly in $L^2((0,T)\times\R)$,
\item\label{conv:square_product_uv} $\ueps^2 \weak \overline{u^2}$, $\ueps^3 \weak \overline{u^3}$, $\ueps^4 \weak \overline{u^4}$, $\veps^2 \weak \overline{v^2}$, $\ueps\,\veps \weak \overline{u\,v}$ weakly/weakly$^*$ in $L^{p}((0,T)\times\R)$ for $p\in[1,\infty]$ and for some functions $\overline{u^2}, \overline{u^3}, \overline{u^4}, \overline{v^2}, \overline{u\,v}$, 
\item\label{conv_logarithmic_fcns} $\ueps\log\ueps \weak \overline{u\log u}$,  $\frac{\ueps^2}{\seps}\log\ueps \weak \overline{\frac{u^2}{s}\log u}$, $\ueps^2\log\ueps \weak \overline{u^2\log u}$, $\veps\log \veps \weak \overline{v \log v}$, $\ueps \veps \log \veps \weak \overline{ u\,v \log v}$, $\frac{\ueps}{\seps} \veps \log \veps \weak \overline{ \frac{u}{s}\,v \log v}$  weakly/weakly$^*$ in $L^{p}((0,T)\times\R)$ for all $p\in[1,\infty]$ and for some functions $\overline{u\log u}$, $\overline{\frac{u^2}{s}\log u}$, $\overline{u^2\log u}$, $\overline{v \log v}$, $\overline{ u\,v \log v}$, $\overline{ \frac{u}{s}\,v \log v}$,
\item\label{conv_nabla_s_alpha_squared_in_measures} $|\partial_x \seps^{\alpha}|^2 \weak |\overline{\partial_x s^{\alpha}|^2}$ weakly in $\mathcal{M}_{\loc}([0,T]\times\R)$ for some $\overline{|\partial_x s^{\alpha}|^2} \in \mathcal{M}_{\loc}([0,T]\times\R)$,
\item\label{conv_product_u_derivative_s_alpha} $\ueps\, \partial_{x} \seps^{\alpha} \weak \overline{u\,\partial_x s^{\alpha}}$, $\ueps\, \partial_x p(\seps) \weak \overline{u\, \partial_x p(s)}$ weakly in $L^2((0,T)\times\R)$ for some functions $\overline{u\,\partial_x s^{\alpha}}, \overline{u\, \partial_x p(s)} \in L^2((0,T)\times\R)$ satisfying $\alpha\,s\,\overline{u\, \partial_x p(s)} = \overline{u\,\partial_x s^{\alpha}}$,
\item\label{conv:ulogu_vlogv_times_der} $\ueps \log \ueps \, \partial_x p(\seps) \weak \overline{u\log u \, \partial_x p(s)}$, $\veps \log \veps \, \partial_x p(\seps) \weak \overline{v\log v \, \partial_x p(s)}$, $\ueps \log \ueps \, \partial_x \seps^{\alpha} \weak \overline{u\log u \, \partial_x s^{\alpha}}$, $\veps \log \veps \, \partial_x \seps^{\alpha} \weak \overline{v\log v \, \partial_x s^{\alpha}}$ weakly in $L^2((0,T)\times\R)$ for some functions $\overline{u\log u \, \partial_x p(s)}$, $\overline{v\log v \, \partial_x p(s)}$, $\overline{u\log u \, \partial_x s^{\alpha}}$, $\overline{v\log v \, \partial_x s^{\alpha}}$ satisfying $\alpha\,s \, \overline{u\log u \, \partial_x p(s)}  = \overline{u\log u \, \partial_x s^{\alpha}}$ and $\alpha\,s \, \overline{v\log v \, \partial_x p(s)}  = \overline{v\log v \, \partial_x s^{\alpha}}$,
\item\label{conv:fractions} for all $k, l \in \N$, $l \leq k \leq 4$, we have $\frac{\ueps^k}{\seps^l} \weak \overline{\,\frac{u^k}{s^l}\,}$ weakly$^*$ in $L^{\infty}((0,T)\times\R)$ for some functions $\overline{\,\frac{u^k}{s^l}\,}$ satisfying ${s^{l-m}} \overline{\,\frac{u^k}{s^l}\,} = \overline{\,\frac{u^k}{s^m}\,} $ for all $m \in \N$, $0 \leq m \leq l$ (in particular, $s\, \overline{\,\frac{u}{s}\,} =  u$)
\item\label{conv_product_u_derivative_s_alpha_divided_by_s} $\frac{\ueps}{\seps}\, \partial_{x} \seps^{\alpha} \weak \overline{\frac{u}{s}\,\partial_x s^{\alpha}}, \frac{\veps}{\seps}\, \partial_{x} \seps^{\alpha} \weak \overline{\frac{v}{s}\,\partial_x s^{\alpha}}$, $\ueps \, \partial_{x} \seps^{\alpha} \weak \overline{u\,\partial_x s^{\alpha}}, {\frac{\ueps^2}{\seps^2}\, \partial_x \seps^{\alpha}} \weak \overline{{\frac{u^2}{s^2}\, \partial_x s^{\alpha}}}$ weakly in $L^2((0,T)\times\R)$ for some functions $\overline{\frac{u}{s}\,\partial_x s^{\alpha}}, \overline{\frac{v}{s}\,\partial_x s^{\alpha}}, \overline{u\,\partial_x s^{\alpha}}, \overline{{\frac{u^2}{s^2}\, \partial_x s^{\alpha}}} \in L^2((0,T)\times\R)$ satisfying $\overline{\frac{u}{s}\,\partial_x s^{\alpha}} + \overline{\frac{v}{s}\,\partial_x s^{\alpha}} = \partial_x s^{\alpha}$ and $s\,\overline{\frac{u}{s}\,\partial_x s^{\alpha}} =\overline{u\,\partial_x s^{\alpha}}$.
\end{enumerate}
\end{proposition}
Concerning the overline notation, we use the it to denote weak limits which exist by the a~priori estimates but we do not know their representation. In particular, we do not know that $\overline{|\partial_x s^{\alpha}|^2} =|\partial_x s^{\alpha}|^2$ or $\overline{u\,\partial_x s^{\alpha}} = u\, \partial_x s^{\alpha}$. Before proceeding to the proof, we will state a~version of the Aubin-Lions Lemma taken from \cite[Corollary 7.9]{MR3014456} in the case that the time derivative is a measure.

\begin{lem}\label{lem:aubin-lions-measure}
Let $p \in (1,\infty)$ and let $V_1$, $V_2$, $V_3$ be Banach spaces such that $V_1$ is compactly embedded into $V_2$ while $V_2$ is continuously embedded into $V_3$. Moreover, suppose that $V_1$ is reflexive while $V_3$ has a predual space, i.e. $V_3 = (V_3')^*$. Then, if $\{\ueps\}$ is bounded in $L^p(0,T ; V_1)$ and $\{\partial_t \ueps\}$ is bounded in $C(0,T; V_3')^*$, then $\{\ueps\}$ is precompact in $L^p(0,T; V_2)$.
\end{lem}

\begin{proof}[Proof of Proposition \ref{prop:basic_convergences}] We divide the proof into several steps.\\

\underline{Proof of \ref{conv:sum_strong}.} We observe that $\partial_t \seps^{1+\frac{\alpha}{2}} = (1+\frac{\alpha}{2}) \seps^{\frac{\alpha}{2}}\, \partial_t \seps$. Let $\varphi \in C_c^{\infty}([0,T]\times B_R)$ with $B_R = (-R,R)$ and $R>0$ fixed. Then,
$$
\frac{1}{1+\frac{\alpha}{2}} \int_{0}^T \int_{\R} \partial_t \seps^{1+\frac{\alpha}{2}}\, \varphi \diff x \diff t \leq \|\partial_t \seps\|_{L^2(0,T; H^{-1}(\R))}\, \|  \seps^{\frac{\alpha}{2}}\|_{L^2(0,T; H^{1}(\R))}\, \|\varphi \|_{C(0,T; W^{1,\infty}_0(B_R))},
$$
so by the continuous embedding of $H^2_0(B_R)$ into $W^{1,\infty}_0(B_R)$ we deduce that $\{\partial_t \seps^{1+\frac{\alpha}{2}} \}$ is bounded in $(C(0,T; H^2_0(B_R)))^*$. Similarly, ${\partial_x \seps^{1+\frac{\alpha}{2}}}$ is bounded in $L^2(0,T; H^1(\Rd))$. We apply Lemma \ref{lem:aubin-lions-measure} with $V_1 = H^1(B_R)$, $V_2 = L^2(B_R)$, $V_3 = H^{-2}_0(B_R)$ and $p=2$ to obtain a subsequence such that $\seps^{1+\frac{\alpha}{2}} \to g$ a.e. on $[0,T]\times B_R$ for each $R>0$. By the diagonal method, we obtain a subsequence such that $\seps^{1+\frac{\alpha}{2}} \to g$ a.e. on $[0,T]\times \R$ so that $\seps \to g^{\frac{1}{1+\frac{\alpha}{2}}}$ and since the limits of a.e. convergence and weak convergence coincide we deduce $g^{\frac{1}{1+\frac{\alpha}{2}}} = s$. Having a.e. convergence, we obtain the convergence in $L^1(0,T;L^1_{\loc}(\R))$ by the dominated convergence which can be upgraded to the convergence in $L^1((0,T)\times\R)$ by the moment estimate \ref{est:moment}. Finally, the convergence in $L^p((0,T)\times\R)$ follows by interpolation.\\

\underline{Proof of \ref{conv:u_v_weak}, \ref{conv:gradient_sum}, \ref{conv:time_derivative}.}  Clearly, \ref{conv:u_v_weak} follows directly from the bound \ref{est:conservation} (the weak convergence for $p=1$ is a consequence of the Dunford-Pettis theorem, the tail estimate \ref{est:moment} and the uniform estimate for $p>1$). For \ref{conv:gradient_sum}, there is a subsequence converging weakly to some $f\in L^2((0,T)\times\R)$ and then we identify $f = \partial_x \seps^{\alpha}$ since for all smooth $\varphi(t,x)$ we have $\int_0^T \int_{\R} \partial_x \varphi \, \seps^{\alpha} \diff x \diff t \to \int_0^T \int_{\R} \partial_x \varphi \, s^{\alpha} \diff x \diff t$ by $\seps^{\alpha} \to s^{\alpha}$ a.e. in \ref{conv:sum_strong} and the dominated convergence ($\{\seps\}$ is uniformly bounded by \ref{est:conservation}). An analogous argument works for $\partial_x \seps^{\frac{\alpha}{2}}$. Finally, \ref{conv:time_derivative} is a direct consequence of the estimate \ref{est:partial_t_neg_ss}.\\

\underline{Proof of \ref{conv:velocity_fields}.} We will only prove the convergence of $\{\partial_x K^{1,1} \ast \ueps\}$ (for the other components one argues similarly). Let $\chi_{R}$ be a smooth function which is 1 on $[-R,R]$ and 0 on the set $\R\setminus [-2R,2R]$. For a compact set $\Omega \subset \R$ we estimate
\begin{multline*}
\|\partial_x K^{1,1} \ast (\ueps- u) \|_{L^1((0,T)\times\Omega)} \leq \\ \leq  \|\partial_x K^{1,1} \ast ((\ueps- u)\,\chi_R) \|_{L^1((0,T)\times\Omega)} + \|\partial_x K^{1,1} \ast ((\ueps- u)\, (1-\chi_R)) \|_{L^1((0,T)\times\Omega)}. 
\end{multline*}
The first term converges to 0 for each $R>0$ (indeed, by a standard version of the Aubin-Lions lemma, it has a subsequence converging to 0 so we deduce that the whole sequence converges to 0; in order to apply the Aubin-Lions lemma, the spatial regularity is obtained from $\partial_x K^{1,1}, \partial_x^2 K^{1,1} \in L^{\infty}(\R)$, the temporal regularity is guaranteed by \ref{est:partial_t_neg_ss} and the whole reasoning is restricted to a compact set thanks to the function $\chi_R$ and $\Omega$ being bounded). The second term is arbitrarily small by the assumption $\partial_x K^{1,1} \in L^{\infty}(\R)$, the moment estimate \ref{est:moment} and Young's convolutional inequality. As soon as the convergence in $L^1((0,T)\times\Omega)$ is established, we can extend it to $L^p((0,T)\times\Omega)$ with $1<p<\infty$ by interpolation while the a.e. convergence follows by extracting a further subsequence.\\

\underline{Proof of \ref{conv:der_s_alpha_zero_in_vacuum}.} We first note the identity $\partial_x \seps^{\alpha} = \frac{1}{2} \seps^{\frac{\alpha}{2}} \, \partial_x \seps^{\frac{\alpha}{2}}$ and we observe that we can pass to the weak limit in this identity for instance in $L^1(0,T; L^1_{\loc}(\R))$ by \ref{conv:gradient_sum} and $\seps^{\frac{\alpha}{2}} \to s^{\frac{\alpha}{2}}$ strongly in $L^2(0,T; L^2_{\loc}(\R))$ (an easy consequence of the dominated convergence and \ref{conv:sum_strong}, \ref{est:conservation}). We obtain $\partial_x s^{\alpha} = \frac{1}{2} s^{\frac{\alpha}{2}} \, \partial_x s^{\frac{\alpha}{2}}$ which implies  $\partial_x s^{\alpha}=0$ a.e. on the set $\{s=0\}$. To conclude, we need to prove $\int_0^T \int_{\Omega} |\partial_x \seps^{\alpha}|\, \mathds{1}_{s=0} \diff x \diff t \to 0$ which follows by 
$$
\int_0^T \int_{\Omega} |\partial_x \seps^{\alpha}|\, \mathds{1}_{s=0} \leq \frac{1}{2} \|\seps^{\frac{\alpha}{2}}\, \mathds{1}_{s=0}\|_{L^2((0,T)\times\Omega)} \, \| \partial_x \seps^{\frac{\alpha}{2}} \|_{L^2((0,T)\times\R)}
$$ 
and $\|\seps^{\frac{\alpha}{2}}\, \mathds{1}_{s=0}\|_{L^2((0,T)\times\Omega)} \to 0$ by the dominated convergence as above.\\

\underline{Proof of \ref{conv:integrated_time_der_sum}.} We observe that the sequence $\{\int_{-\infty}^x \partial_t \seps \diff y\}$ is bounded in $L^2((0,T)\times\R)$. Indeed, summing up the PDEs in \eqref{eq:general_cross_diffusion_intro_viscosity} and integrating in space we obtain
$$
 \int_{-\infty}^x \partial_t \seps(t,y)\diff y = \frac{1}{\alpha} \partial_x \seps^{\alpha} + \ueps\, \partial_x \mathcal{V}^1[\ueps,\veps] + \veps \, \partial_x \mathcal{V}^2[\ueps,\veps] + \eps\, \partial_x \seps
$$
(there is a small issue when $\seps = 0$ and this is discussed in detail in \eqref{eq:integrated_PDE_for_seps_very_first_step_proof} later). Therefore, it is sufficient to use the estimates \ref{est:crucial_estimate_gradient_power_alpha}, \ref{est:conservation}, \ref{est:energy_derivatives_u_v_eps_signular} and Assumption \ref{ass:velocity_kernels_main}. Hence, the subsequence extracted in \ref{conv:sum_strong} has a further subsequence converging weakly to some limit in $L^2((0,T)\times\R)$. Using the convergence \ref{conv:sum_strong}, the limit can be identified to be $\int_{-\infty}^x \partial_t s(t,y) \diff y$. \\

\underline{Proof of \ref{conv:square_product_uv}, \ref{conv_logarithmic_fcns}, \ref{conv_nabla_s_alpha_squared_in_measures}, \ref{conv_product_u_derivative_s_alpha}.} For \ref{conv:square_product_uv}, this is a direct consequence of the estimate \ref{est:conservation} and the Banach-Alaoglu theorem (for $p=1$ we also need a tail estimate to use the Dunford-Pettis theorem which follows for all of the sequences there because of the $L^{\infty}$ estimate in \ref{est:conservation} and the moment estimate \ref{est:moment}). Similarly, for \ref{conv_logarithmic_fcns}, we argue as for \ref{conv:square_product_uv} using the analogous estimates \ref{est:ulogu_vlogv} and \ref{est:ulogu_vlogv_moment} (the case $p=1$ needs an application of the Dunford-Pettis theorem again). The convergences \ref{conv_nabla_s_alpha_squared_in_measures} and \ref{conv_product_u_derivative_s_alpha} are direct consequences of the Banach-Alaoglu theorem and estimates \ref{est:crucial_estimate_gradient_power_alpha} and \ref{est:crucial_estimate_gradient_power_alpha}, \ref{est:conservation}, respectively. The identity $\alpha\,s\,\overline{u\, \partial_x p(s)} = \overline{u\,\partial_x s^{\alpha}}$ in \ref{conv_product_u_derivative_s_alpha} is a simple consequence of the identity $\alpha\, \seps\, \ueps \, \partial_x p(\seps) = \ueps\, \partial_x \seps^{\alpha}$ in which we can pass to the weak limit in $L^1_{\loc}((0,T)\times\R)$ by \ref{conv:sum_strong} (we recall that the weak limits in $L^2_{\loc}((0,T)\times\R)$ and $L^1_{\loc}((0,T)\times\R)$ coincide). \\

\underline{Proof of \ref{conv:ulogu_vlogv_times_der}.} The weak compactness in $L^2((0,T)\times\R)$ of $\{\ueps\log \ueps \, \partial_x \seps^{\alpha}\}$, $\{ \veps \log \veps\, \partial_x \seps^{\alpha}\}$ is clear by \ref{est:crucial_estimate_gradient_power_alpha} and \ref{est:ulogu_vlogv}. For $\{\ueps\log \ueps \, \partial_x p(\seps) \}$ we write
\begin{equation}\label{eq:boundedness_in_L2_ulogu_parx_pseps}
\ueps\log \ueps \, \partial_x p(\seps)  = \ueps \log \ueps \, \seps^{\alpha-2}\, \partial_x \seps = \ueps \log \ueps \, \seps^{\frac{\alpha}{2}-1}\, \frac{2}{\alpha} \partial_x \seps^{\frac{\alpha}{2}},
\end{equation}
so by \ref{est:gradient_from_entropy}, we only need to show that $\{\ueps \log \ueps \, \seps^{\frac{\alpha}{2}-1}\}$ is bounded in $L^{\infty}((0,T)\times\R)$. This is clear when $\alpha \geq 2$ by \ref{est:conservation} and \ref{est:ulogu_vlogv}. If $\alpha <2$, we write $\ueps \log \ueps \, \seps^{\frac{\alpha}{2}-1} = \ueps^{\frac{\alpha}{2}}\,  \log \ueps\, (\ueps/\seps)^{1-\frac{\alpha}{2}}$ and we use again \ref{est:conservation} together with the fact that $[0,\infty) \ni x \mapsto x^{\frac{\alpha}{2}} \log x $ is continuous. The identity $\alpha\,s \, \overline{u\log u \, \partial_x p(s)}  = \overline{u\log u \, \partial_x s^{\alpha}}$ follows by writing $\alpha\,\seps \, {\ueps \log \ueps \, \partial_x p(\seps)}  = {\ueps \log \ueps \, \partial_x \seps^{\alpha}}$ and passing to the weak limit in $L^1_{\loc}((0,T)\times\R)$ by \ref{conv:sum_strong}. The argument for $\{\veps \log \veps \, \partial_x p(\seps) \}$ is the same.\\

\underline{Proof of \ref{conv:fractions}.} The weak compactness is the direct consequence of \ref{est:conservation} and the inequality $\ueps \leq \seps$. The identity ${s^{l-m}} \overline{\,\frac{u^k}{s^l}\,} = \overline{\,\frac{u^k}{s^m}\,} $ follows by writing $\seps^{l-m} \, \frac{\ueps^k}{\seps^l} = \frac{\ueps^k}{\seps^m}$  the limit, let $f_{\eps} = \frac{\ueps^k}{\seps^l}$ and passing to the limit in this identity using \ref{conv:sum_strong} and the weak compactness proved in the current step.\\

\underline{Proof of \ref{conv_product_u_derivative_s_alpha_divided_by_s}.} The four convergences follow by the Banach-Alaoglu theorem and the bounds \ref{est:crucial_estimate_gradient_power_alpha} and \ref{est:conservation}. For the first identity, we write $\frac{\ueps}{\seps}\, \partial_{x} \seps^{\alpha} + \frac{\veps}{\seps}\, \partial_{x} \seps^{\alpha} = \partial_x \seps^{\alpha}$ (which is immediate when $\seps>0$ while when $\seps = 0$ we use that $\partial_x \seps^{\alpha} = 2\,\seps^{\frac{\alpha}{2}}\, \partial_x \seps^{\frac{\alpha}{2}} = 0$ since both derivatives are $L^2((0,T)\times\R)$ functions) and we pass to the limit in the identity using \ref{conv:gradient_sum}. Similarly, for the second identity we write $\seps\, \frac{\ueps}{\seps}\, \partial_{x} \seps^{\alpha} = \ueps\, \partial_{x} \seps^{\alpha}$ and we pass to the limit. Both sides converge in $L^1_{\loc}((0,T)\times\R)$ using \ref{conv:sum_strong} with $p=2$ so we deduce $s\,\overline{\frac{u}{s}\,\partial_x s^{\alpha}} =\overline{u\,\partial_x s^{\alpha}}$.
\end{proof}

\subsection{A conservation law applicable for all $\alpha \in (0,\infty)$}\label{subsect:cons_laws_all_alpha}

The main novel point in this section is the observation that there is a nonlinear function of $\ueps$ which satisfies a certain conservation law, suitable for application of the Div-Curl Lemma (Theorem~\ref{thm:div_curl_lemma_version_CR}).

\begin{proposition}\label{prop:new_cons_law_log_general_alpha}
Let $\ueps$, $\veps$ solve \eqref{eq:general_cross_diffusion_intro_viscosity}, $\seps = \ueps + \veps$ and $
E_\eps = \ueps\,\log\ueps + \veps \, \log \veps.$
Then,
\begin{equation}\label{eq:PDE_cons_law_entropy_eps}
\partial_t E_{\eps} = \partial_x \big((E_{\eps} + \seps)\,\partial_x p(\seps) + \ueps  \log \ueps\, \partial_x \mathcal{V}^1[\ueps, \veps] + \veps  \log \veps\, \partial_x \mathcal{V}^2[\ueps, \veps]\big) + g_{\eps} + h_{\eps},
\end{equation}
where $g_{\eps}$, $h_{\eps}$ are defined as follows:
\begin{align*}
g_{\eps} &= -\partial_x p(\seps)\,  \partial_x \seps +  \ueps\,\partial^2_x \mathcal{V}^1[\ueps, \veps] +  \veps\,\partial^2_x \mathcal{V}^2[\ueps, \veps] - \eps\, \frac{|\partial_x \ueps|^2}{\ueps} - \eps\, \frac{|\partial_x \veps|^2}{\veps},\\
h_{\eps} &= \partial_x(\eps\, \partial_x \ueps\, (1+\log\ueps)) + \partial_x(\eps\, \partial_x \veps\, (1+\log\veps)).
\end{align*}
Moreover, $h_{\eps} \to 0$ in $L^2(0,T; H^{-1}(\R))$ and $\{g_{\eps}\}_{\eps \in (0,1)}$ is bounded in $L^1((0,T)\times\R)$. 
\end{proposition}
\begin{proof}
Multiplying the first equation by $1+\log \ueps$, we obtain by the product rule
\begin{align*}
\partial_t (\ueps\log\ueps) = & \, \partial_x(\ueps\,(1+\log\ueps)\,\partial_x p(\seps)) - \ueps\,\partial_x p(\seps)\, \frac{1}{\ueps}\, \partial_x \ueps \\
&+ (\partial_x \ueps \, \partial_x \mathcal{V}^1[\ueps, \veps] + \ueps \ \partial_x^2 \mathcal{V}^1[\ueps, \veps] )\, (1+\log\ueps)+ \eps\,\partial^2_x \ueps\, (1+\log\ueps).
\end{align*}
For the terms with the velocity $\mathcal{V}^1[\ueps, \veps]$ we have
\begin{align*}
&(\partial_x \ueps \, \partial_x \mathcal{V}^1[\ueps, \veps] + \ueps \, \partial_x^2 \mathcal{V}^1[\ueps, \veps] )\, (1+\log\ueps)  \\
&\qquad \qquad = \partial_x(\ueps \log \ueps)\,  \partial_x \mathcal{V}^1[\ueps, \veps] + \ueps\, (1+\log\ueps)\, \partial_x^2 \mathcal{V}^1[\ueps, \veps]\\
&\qquad \qquad = \partial_x(\ueps \log \ueps\,  \partial_x \mathcal{V}^1[\ueps, \veps]) + \ueps\, \partial_x^2 \mathcal{V}^1[\ueps, \veps],
\end{align*}
while for the term with viscosity we have
$$
\eps\,\partial^2_x \ueps\, (1+\log\ueps) = \partial_x(\eps\, \partial_x \ueps\, (1+\log\ueps))- \eps\, \frac{|\partial_x \ueps|^2}{\ueps}.
$$
Hence, we deduce
\begin{align*}
\partial_t (\ueps\log\ueps) =&\, \partial_x(\ueps\,(1+\log\ueps)\,\partial_x p(\seps) + \ueps \log \ueps \, \partial_x \mathcal{V}^1[\ueps, \veps]) -\partial_x p(\seps)\,  \partial_x \ueps \\
& + \ueps\,\partial^2_x \mathcal{V}^1[\ueps, \veps]  + \partial_x(\eps\, \partial_x \ueps\, (1+\log\ueps))- \eps\, \frac{|\partial_x \ueps|^2}{\ueps}.
\end{align*}
Analogously, we obtain for $\veps$:
\begin{align*}
\partial_t (\veps\log\veps) =&\, \partial_x(\veps\,(1+\log\veps)\,\partial_x p(\seps) + \veps \log \veps \, \partial_x \mathcal{V}^2[\ueps, \veps]) -\partial_x p(\seps)\,  \partial_x \veps \\
& + \veps\,\partial^2_x \mathcal{V}^2[\ueps, \veps]  + \partial_x(\eps\, \partial_x \veps\, (1+\log\veps))- \eps\, \frac{|\partial_x \veps|^2}{\veps}.
\end{align*}
Summing up equations for $\ueps\log\ueps$ and $\veps\log\veps$ we arrive at \eqref{eq:PDE_cons_law_entropy_eps}. The estimate on $\{g_{\eps}\}_{\eps \in (0,1)}$ follows directly from \ref{est:gradient_from_entropy} (since $\partial_x p(\seps)\, \partial_x \seps = \seps^{\alpha-2}\, |\partial_x \seps|^2 = |\partial_x \seps^{\frac{\alpha}{2}}|^2$), Assumption~\ref{ass:velocity_kernels_main} which implies that $\partial^2_x \mathcal{V}^1, \partial^2_x\mathcal{V}^2$ are bounded and \ref{est:entropy_derivatives_u_v_eps_signular}. To see the convergence $h_{\eps}\to0$, we write
\begin{equation}\label{eq:sequence_to_analyze_strong_conv_in_H_-1_entropic_variables}
\eps\, \partial_x \ueps\, (1+\log\ueps) = \eps\, \frac{\partial_x \ueps}{\sqrt{\ueps}}\, \sqrt{\ueps}\, (1+\log\ueps).
\end{equation}
Since $\{\sqrt{\eps} \frac{\partial_x \ueps}{\sqrt{\ueps}}\}_{\eps\in(0,1)}$ is bounded in $L^2((0,T)\times\R)$ by \ref{est:entropy_derivatives_u_v_eps_signular} and $\{\sqrt{\ueps}\, (1+\log\ueps)
\}_{\eps\in(0,1)}$ is bounded in $L^{\infty}((0,T)\times\R)$ by \ref{est:conservation}, the sequence in \eqref{eq:sequence_to_analyze_strong_conv_in_H_-1_entropic_variables} converges strongly to 0 in $L^2((0,T)\times\R)$. It follows that $\partial_x(\eps\, \partial_x \ueps\, (1+\log\ueps))\to 0$ strongly in $L^2(0,T; H^{-1}(\R))$ and since the same argument can be performed for $\partial_x(\veps\, \partial_x \veps\, (1+\log\veps))$, we conclude that $h_{\eps}\to0$ strongly in $L^2(0,T; H^{-1}(\R))$. 
\end{proof} 

\subsection{A conservation law applicable for $\alpha \in (0,2]$}\label{subsect:cons_law_u2_over_s}

Here, we provide another conservation law which is applicable only for $\alpha \in (0,2]$ (see Remark \ref{rem:integrability_of_derivative_sum_with_advection} for a discussion of the range). This conservation law can be used instead of \eqref{eq:PDE_cons_law_entropy_eps} to prove Proposition \ref{prop:the_case_of_support_on_the_line_identification_of_YM} (and consequently, Theorem \ref{thm:main}) slightly easier and we present this alternative argument in Section \ref{subsect:alternative_proof_alpha_at_most_3}. The quantity that turns out to be useful for us, $\frac{\ueps^2}{\seps}$, can be thought of as an integrand of the relative $L^2$ entropy $(1-\frac{\ueps}{\seps})^2\, \seps$ up to the terms whose limit can be identified. Note also that a similar relationship holds for $\frac{\ueps^{k+1}}{\seps^k}$ (see Remark \ref{rem:cons_law_uk+1_over_sk} below).
\begin{proposition}\label{prop:new_cons_law}
Let $\ueps$, $\veps$ solve \eqref{eq:general_cross_diffusion_intro_viscosity} and let $\seps = \ueps + \veps$. Then,
\begin{equation}\label{eq:PDE_for_usquared_over_s_main_statement}
\partial_t \left(\frac{\ueps^2}{\seps}\right) = \partial_x\left(\frac{1}{\alpha}\,\frac{\ueps^2}{\seps^2}\, \partial_x \seps^{\alpha} + f_{\eps} \right) + g_{\eps} +  h_{\eps},
\end{equation}
where $f_{\eps}$, $g_{\eps}$, $h_{\eps}$ are defined as follows:
\begin{align*}
f_{\eps} &= \frac{\ueps^2}{\seps}\,  \partial_x \mathcal{V}^1[\ueps, \veps] - \frac{\ueps^3}{3 \, \seps^2}\, \partial_x (\mathcal{V}^1[\ueps, \veps] - \mathcal{V}^2[\ueps, \veps]),\\
g_{\eps} &= \frac{\ueps^2}{\seps^2}\,\left(1-\frac{2\,\ueps}{3\,\seps} \right)\, \partial_x \big(\seps\,  \partial_x (\mathcal{V}^1[\ueps, \veps]-\mathcal{V}^2[\ueps, \veps])\big)-2\,\eps\,\seps\,\left(\frac{\partial_x \ueps \, \seps - \ueps \, \partial_x \seps}{\seps^2} \right)^2,\\
h_{\eps} &= 2\,\eps\, \partial_x\left(\frac{\ueps}{\seps} \, \partial_x \ueps\right)- \eps\, \partial_x\left(\frac{\ueps^2}{\seps^2}\, \partial_x \seps \right).
\end{align*}
Moreover, $h_{\eps} \to 0$ strongly in $L^2(0,T; H^{-1}(\R))$ and, if $\alpha\leq 2$, the sequence $\{g_{\eps}\}_{\eps\in(0,1)}$ is bounded in $L^{1}(0,T; L^1_{\loc}(\R))$ (see Remark~\ref{rem:integrability_of_derivative_sum_with_advection} below for a discussion of the condition $\alpha\leq2$).
\end{proposition} 

\begin{proof}[Proof of Proposition \ref{prop:new_cons_law}] We compute using \eqref{eq:general_cross_diffusion_intro_viscosity}:
\begin{equation}\label{eq:derivation_new_cons_law_step0_time_der}
\begin{split}
&\partial_t \left(\frac{\ueps^2}{\seps}\right) = \frac{2\,\ueps\, \partial_t \ueps \, \seps - \ueps^2\, \partial_t \seps }{\seps^2}\\
&=\,\frac{2\,\ueps \, \seps}{\seps^2} \, \left[\partial_x \ueps \, \partial_x p(\seps) + \ueps\, \partial_{x}^2 p(\seps) + \partial_x(\ueps \, \partial_x \mathcal{V}^1[\ueps, \veps]) + \eps\, \partial_{x}^2 \ueps \right]\\
&\phantom{=}\,-\frac{\ueps^2}{\seps^2}\left[\partial_x \seps\, \partial_x p(\seps) + \seps\, \partial_x^2 p(\seps) + \partial_x(\ueps \, \partial_x \mathcal{V}^1[\ueps, \veps]) + \partial_x(\veps \, \partial_x \mathcal{V}^2[\ueps, \veps]) + \eps\,\partial_x^2 \seps \right].
\end{split}
\end{equation}
The terms with $p(\seps)$ can be nicely written in the divergence form
\begin{equation}\label{eq:derivation_new_cons_law_step1_terms_with_pse}
\begin{split}
\frac{2\,\ueps \, \seps}{\seps^2} \, &\left[\partial_x \ueps \, \partial_x p(\seps) + \ueps\, \partial_{x}^2 p(\seps)\right] -\frac{\ueps^2}{\seps^2}\,\left[\partial_x \seps\, \partial_x p(\seps) + \seps\, \partial_x^2 p(\seps)\right]\\
&= \frac{\seps}{\seps^2}\, \partial_x \ueps^2 \, \partial_x p(\seps) + \frac{\seps}{\seps^2}\, \ueps^2 \, \partial_{x}^2 p(\seps) -\frac{\ueps^2}{\seps^2}\,\partial_x \seps\, \partial_x p(\seps)\\
&= \partial_x\left( \frac{\ueps^2}{\seps} \right)\, \partial_x p(\seps) + \frac{ \ueps^2}{\seps}\, \partial_{x}^2 p(\seps) = \partial_x\left(\frac{\ueps^2}{\seps}\, \partial_{x} p(\seps)\right) =  \frac{1}{\alpha}\,\partial_x\left(\frac{\ueps^2}{\seps^2}\, \partial_x \seps^{\alpha}\right),
\end{split}
\end{equation}
where in the last line we used computation \eqref{eq:pressure_term_rewritten_derivative_intro}. Now, we consider the terms with the velocity fields $\mathcal{V}^1[\ueps, \veps]$, $\mathcal{V}^2[\ueps, \veps]$. First, we have
\begin{align*}
2\, \frac{\ueps}{\seps} \, \partial_x(\ueps \, \partial_x \mathcal{V}^1[\ueps, \veps])  
&= 
2\, \frac{\ueps}{\seps} \, \partial_x \ueps\, \partial_x \mathcal{V}^1[\ueps, \veps] + 2\, \frac{\ueps^2}{\seps} \, \partial^2_x \mathcal{V}^1[\ueps, \veps] \\
&= 
\frac{\partial_x \ueps^2}{\seps}\,  \partial_x \mathcal{V}^1[\ueps, \veps] +2\, \frac{\ueps^2}{\seps} \, \partial^2_x \mathcal{V}^1[\ueps, \veps]\\
&= \partial_x\left(\frac{\ueps^2}{\seps}\,  \partial_x \mathcal{V}^1[\ueps, \veps]\right) + \frac{\ueps^2}{\seps^2}\, \partial_x \seps\,  \partial_x \mathcal{V}^1[\ueps, \veps] + \frac{\ueps^2}{\seps}\,  \partial^2_x \mathcal{V}^1[\ueps, \veps]\\
&=  \partial_x\left(\frac{\ueps^2}{\seps}\,  \partial_x \mathcal{V}^1[\ueps, \veps]\right) + \frac{\ueps^2}{\seps^2}\, \partial_x \big(\seps\,  \partial_x \mathcal{V}^1[\ueps, \veps] \big).
\end{align*}
For the term $- \frac{\ueps^2}{\seps^2}\,\partial_x(\veps \, \partial_x \mathcal{V}^2[\ueps, \veps])$ we first write
\begin{equation}\label{eq:term_vf_V2_split_v_for_s_u_to_combine_with_2nd_term}
- \frac{\ueps^2}{\seps^2}\,\partial_x(\veps \, \partial_x \mathcal{V}^2[\ueps, \veps]) = - \frac{\ueps^2}{\seps^2}\,\partial_x(\seps \, \partial_x \mathcal{V}^2[\ueps, \veps]) + \frac{\ueps^2}{\seps^2}\,\partial_x(\ueps \, \partial_x \mathcal{V}^2[\ueps, \veps]),
\end{equation}
so that the second term in \eqref{eq:term_vf_V2_split_v_for_s_u_to_combine_with_2nd_term} can be combined with $
 - \frac{\ueps^2}{\seps^2}\, \partial_x(\ueps \, \partial_x \mathcal{V}^1[\ueps, \veps])$ which yields
\begin{align*}
 &- \frac{\ueps^2}{\seps^2}\, \partial_x(\ueps \, \partial_x (\mathcal{V}^1[\ueps, \veps] - \mathcal{V}^2[\ueps, \veps] ) )  \\
&=  - \frac{\partial_x \ueps^3}{3 \, \seps^2}\, \partial_x (\mathcal{V}^1[\ueps, \veps] - \mathcal{V}^2[\ueps, \veps])
  - \frac{\ueps^3}{\seps^2}\,  \partial^2_x (\mathcal{V}^1[\ueps, \veps] - \mathcal{V}^2[\ueps, \veps])\\
&=  - \partial_x\left(\frac{\ueps^3}{3 \, \seps^2}\, \partial_x (\mathcal{V}^1[\ueps, \veps] - \mathcal{V}^2[\ueps, \veps])\right) -\frac{2}{3}\, \frac{\ueps^3}{\seps^3} \, \partial_x \seps \, \partial_x(\mathcal{V}^1[\ueps, \veps] - \mathcal{V}^2[\ueps, \veps])\\
&\,\phantom{ = } - \frac{2}{3}\,\frac{\ueps^3}{\seps^2}\,  \partial^2_x (\mathcal{V}^1[\ueps, \veps] - \mathcal{V}^2[\ueps, \veps])\\
&=  - \partial_x\left(\frac{\ueps^3}{3 \, \seps^2}\, \partial_x (\mathcal{V}^1[\ueps, \veps] - \mathcal{V}^2[\ueps, \veps])\right) - \frac{2}{3}\,\frac{\ueps^3}{\seps^3}\, \partial_x \big(\seps\,  \partial_x (\mathcal{V}^1[\ueps, \veps] - \mathcal{V}^2[\ueps, \veps])\big).
\end{align*} 
Hence, the terms with the velocity fields $\mathcal{V}^1[\ueps, \veps]$, $\mathcal{V}^2[\ueps, \veps]$ sum up to
\begin{equation}\label{eq:derivation_new_cons_law_step2_adv_terms}
\begin{split}
2\, \frac{\ueps}{\seps} \, \partial_x(&\ueps \, \partial_x \mathcal{V}^1[\ueps, \veps]) - \frac{\ueps^2}{\seps^2}\, \partial_x(\ueps \, \partial_x \mathcal{V}^1[\ueps, \veps]) - \frac{\ueps^2}{\seps^2}\,\partial_x(\veps \, \partial_x \mathcal{V}^2[\ueps, \veps]) \\ 
=\, &\partial_x\left(\frac{\ueps^2}{\seps}\,  \partial_x \mathcal{V}^1[\ueps, \veps]\right) - \partial_x\left(\frac{\ueps^3}{3 \, \seps^2}\, \partial_x (\mathcal{V}^1[\ueps, \veps] - \mathcal{V}^2[\ueps, \veps])\right)\\
&+\frac{\ueps^2}{\seps^2}\,\left(1-\frac{2\,\ueps}{3\,\seps} \right)\, \partial_x \big(\seps\,  \partial_x (\mathcal{V}^1[\ueps, \veps]-\mathcal{V}^2[\ueps, \veps])\big).
\end{split}
\end{equation}
Regarding the terms with $\eps\, \partial^2_x \ueps$ and $\eps\, \partial^2_x \seps$ we can write
\begin{align*}
2\,\eps \, \frac{\ueps}{\seps} \, \partial^2_x \ueps &= 
2\, \eps\, \partial_x\left(\frac{\ueps}{\seps} \, \partial_x \ueps\right) - 2\, \eps\, \partial_x \ueps \, \frac{\seps \, \partial_x \ueps  - \ueps\, \partial_x \seps}{\seps^2} \\
&= 2\, \eps\, \partial_x\left(\frac{\ueps}{\seps} \, \partial_x \ueps\right) - 2\, \eps\, \frac{|\partial_x \ueps|^2}{\seps} + 2\, \eps\, \frac{\ueps\,\partial_x \ueps\, \partial_x \seps}{\seps^2},
\end{align*}
\begin{align*}
-\eps\, \frac{\ueps^2}{\seps^2}\, \partial^2_x \seps =&  
-\eps\, \partial_x\left(\frac{\ueps^2}{\seps^2}\, \partial_x \seps \right) + \eps\, \partial_x \seps\, \frac{2\,\ueps\, \partial_x \ueps \, \seps^2 - 2\, \seps\, \partial_x \seps \, \ueps^2}{\seps^4}\\
=&-\eps\, \partial_x\left(\frac{\ueps^2}{\seps^2}\, \partial_x \seps \right) + 2\,\eps\, \frac{\ueps\, \partial_x \ueps\, \partial_x \seps}{\seps^2} - 2\,\eps\, \frac{\ueps^2\, |\partial_x\seps|^2}{\seps^3}.
\end{align*}
We observe that in the two expressions above, the terms not in the divergence form can be summed up to form $-2\,\eps\,\seps\,\left(\frac{\partial_x \ueps \, \seps - \ueps \, \partial_x \seps}{\seps^2} \right)^2$ so we get
\begin{equation}\label{eq:derivation_new_cons_law_step3_eps_terms}
\begin{split}
2\,\eps \, \frac{\ueps}{\seps} \, \partial^2_x \ueps -&\eps\, \frac{\ueps^2}{\seps^2}\, \partial^2_x \seps \\ &= 2\, \eps\, \partial_x\left(\frac{\ueps}{\seps} \, \partial_x \ueps\right)-\eps\, \partial_x\left(\frac{\ueps^2}{\seps^2}\, \partial_x \seps \right) -2\,\eps\,\seps\,\left(\frac{\partial_x \ueps \, \seps - \ueps \, \partial_x \seps}{\seps^2} \right)^2 .
\end{split}
\end{equation} 
Plugging \eqref{eq:derivation_new_cons_law_step1_terms_with_pse}, \eqref{eq:derivation_new_cons_law_step2_adv_terms} and \eqref{eq:derivation_new_cons_law_step3_eps_terms} into \eqref{eq:derivation_new_cons_law_step0_time_der}, we obtain \eqref{eq:PDE_for_usquared_over_s_main_statement}.\\

Now, $h_{\eps}\to 0$ in $L^2(0,T; H^{-1}(\R))$ because $\frac{\ueps}{\seps}\leq 1$ and $\eps \, \partial_x \ueps, \eps \, \partial_x \seps \to 0$ strongly in $L^2((0,T)\times\R)$ by \ref{est:energy_derivatives_u_v_eps_signular}. Regarding the bound on $\{g_{\eps}\}_{\eps\in(0,1)}$, the first term is controlled because $\seps, \partial_x \seps$ are at least in $L^{1}(0,T; L^1_{\loc}(\R))$ by \ref{est:conservation}, \ref{est:partial_x_seps_without_any_power} and $\partial_x \mathcal{V}^i[\ueps,\veps]$, $\partial^2_x \mathcal{V}^i[\ueps,\veps]$ are bounded by Assumption \ref{ass:velocity_kernels_main}. For the second term in $g_{\eps}$, we estimate
$$
\eps\,\seps\,\left(\frac{\partial_x \ueps \, \seps - \ueps \, \partial_x \seps}{\seps^2} \right)^2 \leq 2\, \eps\, \frac{|\partial_x \ueps|^2}{\seps} + 2\, \eps\, \frac{|\partial_x \seps|^2}{\seps} \leq 2\, \eps\, \frac{|\partial_x \ueps|^2}{\ueps} + 2\, \eps\, \frac{|\partial_x \seps|^2}{\seps}
$$
and both terms are bounded in $L^1((0,T)\times\R)$ by \ref{est:entropy_derivatives_u_v_eps_signular}.
\end{proof}

\begin{rem}\label{rem:integrability_of_derivative_sum_with_advection} We remark that the restriction $\alpha\leq2$ is only needed to ensure that $\{g_{\eps}\}_{\eps \in (0,1)}$ is bounded in $L^{1}_{\loc}((0,T)\times\R)$. This is achieved by proving that $\{\partial_x \seps\}_{\eps\in(0,1)}$ is bounded in $L^{1}(0,T; L^1_{\loc}(\R))$. If such a bound could be established for $\alpha > 2$, the conclusion of Proposition~\ref{prop:new_cons_law} would also hold in this case. Furthermore, the proof of Proposition~\ref{prop:the_case_of_support_on_the_line_identification_of_YM} presented in Section~\ref{subsect:alternative_proof_alpha_at_most_3} would also be valid, yielding an alternative (and simpler) proof of Theorem~\ref{thm:main}. It is not clear whether the range $\alpha \leq 2$ is optimal, and this question appears to be delicate even for the pure porous medium equation $\partial_t s = \partial_x^2 s^{\alpha}$ on $\mathbb{R}$ without advection (see, e.g., \cite{MR2342955, MR4195738}). A naive approach to prove that $\partial_x s \in L^1_{\loc}((0,T)\times\mathbb{R})$ in this case is to use the identity
\begin{equation}\label{eq:multiplying_negative_powers_s_to_control_gradient_s_l1}
-{\alpha\, \kappa} \int_0^t \int_{\R} s^{\alpha+\kappa - 2}\, |\partial_x s|^2 \diff x \diff \tau = \frac{1}{1+\kappa} \int_{\R} (s(t,x)^{1+\kappa} - s(0,x)^{1+\kappa}) \diff x,
\end{equation}
where $\kappa \in (-1,0]$. By splitting the spatial integral into the regions $|x|\leq1$ and $|x|~>~1$, the terms on the right-hand side can be controlled using the mass $\int_{\mathbb{R}} s(t,x) \diff x$ and a sufficiently high moment $\int_{\mathbb{R}} s(t,x)\,|x|^{\sigma} \diff x$, since $1+\kappa \in (0,1]$. Here we use that, for the porous medium equation, all moments are controlled provided the initial data is bounded and has finite moments of all orders. We conclude that $\partial_x s\in L^2((0,T)\times\R)$ provided $\alpha + \kappa - 2 < 0$, which can be achieved when $\alpha < 3$. This does not seem to be optimal, since for the Barenblatt's solution $B(t,x) = t^{-1/(\alpha+1)}\, F(x\,t^{-1/(\alpha+1)})$ with $F(y) = (C-y^2)_+^{1/(\alpha-1)}$, a~direct computation shows that $B(t,x)\in L^1((\tau,T)\times\R)$ for all $\tau>0$. Regarding our problem~\eqref{eq:general_cross_diffusion_intro_viscosity}, the argument exploiting \eqref{eq:multiplying_negative_powers_s_to_control_gradient_s_l1} cannot be adapted, since due to the advection we obtain an additional term on the (RHS) of \eqref{eq:multiplying_negative_powers_s_to_control_gradient_s_l1}, which is roughly $\int_0^t \int_{\R} \seps^{\kappa}\,|\partial_x \seps| \diff x \diff \tau$. By writing $\seps^{\kappa}\,|\partial_x \seps| = \seps^{\frac{\kappa+\alpha-2}{2}}\,|\partial_x \seps|\, \seps^{\frac{\kappa-\alpha+2}{2}}$, we see that this term can be controlled by $\int_0^t \int_{\R} \seps^{\alpha+\kappa - 2}\, |\partial_x \seps|^2 \diff x \diff \tau$ only if $\kappa - \alpha + 2 \geq 0$, i.e. $\alpha \leq 2$.
\end{rem}

\begin{rem}\label{rem:cons_law_uk+1_over_sk}
The same computation as in the proof of Proposition~\ref{prop:new_cons_law} shows that for any $\theta>0$ we have
\begin{equation*}
\partial_t \left(\frac{\ueps^{\theta+1}}{\seps^{\theta}}\right) = \partial_x\left(\frac{1}{\alpha}\,\frac{\ueps^{\theta+1}}{\seps^{\theta+1}}\, \partial_x \seps^{\alpha} + f_{\eps,\theta} \right) + g_{\eps,\theta} +  h_{\eps,\theta},
\end{equation*}
where $f_{\eps,\theta}$, $g_{\eps,\theta}$, $h_{\eps,\theta}$ are defined as follows:
\begin{align*}
f_{\eps,\theta} =\,& \frac{\ueps^{\theta+1}}{\seps^{\theta}}\,  \partial_x \mathcal{V}^1[\ueps, \veps] - \frac{\theta}{\theta+2}\,\frac{\ueps^{\theta+2}}{\seps^{\theta+1}}\, \partial_x (\mathcal{V}^1[\ueps, \veps] - \mathcal{V}^2[\ueps, \veps]),\\
g_{\eps,\theta} =\,& \theta\,\frac{\ueps^{\theta+1}}{\seps^{\theta+1}}\,\left(1-\frac{(\theta+1)}{(\theta+2)}\,\frac{\ueps}{\seps} \right)\, \partial_x \big(\seps\,  \partial_x (\mathcal{V}^1[\ueps, \veps]-\mathcal{V}^2[\ueps, \veps])\big)\\
&-\theta\,(\theta+1)\,\eps\,\seps\, \frac{\ueps^{\theta-1}}{\seps^{\theta-1}}\,\left(\frac{\partial_x \ueps \, \seps - \ueps \, \partial_x \seps}{\seps^2} \right)^2,\\
h_{\eps,\theta} =\,& (\theta+1)\,\eps\, \partial_x\left(\frac{\ueps^{\theta}}{\seps^{\theta}} \, \partial_x \ueps\right)-  \theta\, \eps\, \partial_x\left(\frac{\ueps^{\theta+1}}{\seps^{\theta+1}}\, \partial_x \seps \right).
\end{align*}
These quantities satisfy the same bounds as in Proposition~\ref{prop:new_cons_law}. This yields a whole family of conservation laws, which is applicable only for $\alpha \leq 2$, for the same reason as explained in Remark~\ref{rem:integrability_of_derivative_sum_with_advection}.
\end{rem}

\subsection{The limit of $|\partial_x \seps^{\alpha} - \partial_x s^{\alpha}|^2$}\label{subsect:dissipation_measure_gradient_squared}
Finally, we will need the following identification of the limit of $|\partial_x \seps^{\alpha} - \partial_x s^{\alpha}|^2$. 
\begin{proposition}\label{prop:weak_compact_partial_x_seps}
Let $u, v, s$ be the limits as in \ref{conv:sum_strong}, \ref{conv:u_v_weak} and let $\overline{u\, \partial_x s^{\alpha}}$ be the weak limit as in \ref{conv_product_u_derivative_s_alpha}. The sequence $\{ |\partial_x \seps^{\alpha} - \partial_x s^{\alpha}|^2 \}_{\eps\in(0,1)}$ converges weakly in $\mathcal{M}_{\loc}([0,T]\times\R)$ to
$$
|\partial_x \seps^{\alpha} - \partial_x s^{\alpha}|^2 \weak \alpha\,(\overline{u\, \partial_x s^{\alpha}} - u\, \partial_x s^{\alpha})\,  \partial_x (\mathcal{V}^2[u,v] - \mathcal{V}^1[u,v]).
$$
\end{proposition}
\begin{proof}
The integrated PDE in space for $\seps$ (obtained by summing up \eqref{eq:general_cross_diffusion_intro_viscosity}, using the computation in \eqref{eq:pressure_term_rewritten_derivative_intro} and integrating in space) reads
\begin{equation}\label{eq:integrated_PDE_for_seps_very_first_step_proof}
 \int_{-\infty}^x \partial_t \seps(t,y)\diff y = \frac{1}{\alpha} \partial_x \seps^{\alpha} + \ueps\, \partial_x \mathcal{V}^1[\ueps,\veps] + \veps \, \partial_x \mathcal{V}^2[\ueps,\veps] + \eps\, \partial_x \seps.
\end{equation}
Note that when summing up equations in \eqref{eq:general_cross_diffusion_intro_viscosity} we used $\ueps/\seps + \veps/\seps = 1$. This is not true when $\seps = 0$ because then, by our definition, $\ueps/\seps = \veps/\seps = 0$. However, if $\seps = 0$, then $\frac{1}{\alpha} \partial_x \seps^{\alpha} = \seps^{\alpha-1}\,\partial_x \seps = \seps^{\frac{\alpha}{2}}\,\seps^{\frac{\alpha}{2} - 1}\, \partial_x \seps$ and we know that $\seps^{\frac{\alpha}{2} - 1}\, \partial_x \seps \in L^2((0,T)\times\R)$ by \ref{est:crucial_estimate_gradient_power_alpha} so it is finite a.e. and so we obtain $\frac{1}{\alpha} \partial_x \seps^{\alpha} = 0$ a.e. on $\{\seps = 0\}$. Hence, \eqref{eq:integrated_PDE_for_seps_very_first_step_proof} is still satisfied. We will omit variables $(t,y)$ in what follows. Writing $\veps = \seps - \ueps$ we obtain
\begin{equation}\label{eq:integrated_in_x_PDE_for_seps}
\int_{-\infty}^x \partial_t  \seps \diff y = \frac{1}{\alpha} \partial_x \seps^{\alpha} + \ueps\, \partial_x (\mathcal{V}^1[\ueps,\veps] - \mathcal{V}^2[\ueps,\veps]) + \seps \, \partial_x \mathcal{V}^2[\ueps,\veps] + \eps\, \partial_x \seps.
\end{equation}
Passing to the limit $\eps \to 0$ in \eqref{eq:integrated_in_x_PDE_for_seps} (say, in the sense of distributions) is easy because all nonlinear terms are converging by \ref{conv:sum_strong}, \ref{conv:gradient_sum} and \ref{conv:velocity_fields} so we obtain
\begin{equation}\label{eq:integrated_in_x_PDE_for_s}
 \int_{-\infty}^x \partial_t s \diff y = \frac{1}{\alpha} \partial_x s^{\alpha} + u\, \partial_x (\mathcal{V}^1[u,v]-\mathcal{V}^2[u,v]) + s \, \partial_x \mathcal{V}^2[u,v].
\end{equation}
We multiply \eqref{eq:integrated_in_x_PDE_for_seps} by $\partial_x \seps^{\alpha}$ and \eqref{eq:integrated_in_x_PDE_for_s} by $\partial_x s^{\alpha}$ to obtain two identities
\begin{equation}\label{eq:integrated_in_x_PDE_for_seps_mult_by_der}
\begin{split}
\left[ \int_{-\infty}^x \partial_t \seps \diff y \right] \, \partial_x \seps^{\alpha} = \frac{1}{\alpha} |\partial_x \seps^{\alpha}|^2 + \ueps\, \partial_x \seps^{\alpha}\,  &\partial_x (\mathcal{V}^1[\ueps,\veps] - \mathcal{V}^2[\ueps,\veps]) \\ &+ \seps \, \partial_x \seps^{\alpha}\, \partial_x \mathcal{V}^2[\ueps,\veps] + \eps\, \partial_x \seps\, \partial_x \seps^{\alpha},
\end{split}
\end{equation}
\begin{equation}\label{eq:integrated_in_x_PDE_for_s_mult_by_der}
\left[ \int_{-\infty}^x \partial_t s \diff y \right] \, \partial_x s^{\alpha} = \frac{1}{\alpha} |\partial_x s^{\alpha}|^2 + u\, \partial_x s^{\alpha}\,  \partial_x (\mathcal{V}^1[u,v] - \mathcal{V}^2[u,v]) \,+ s \, \partial_x s^{\alpha}\, \partial_x \mathcal{V}^2[u,v].
\end{equation}
The idea is to pass to the limit $\eps \to 0$ in \eqref{eq:integrated_in_x_PDE_for_seps_mult_by_der} and compare the resulting terms with \eqref{eq:integrated_in_x_PDE_for_s_mult_by_der}. First,
\begin{equation}\label{eq:result_about_convergence_product_integral_time_derivative_and_power_of_s}
\left[\int_{-\infty}^x \partial_t  \seps \diff y\right] \, \partial_x \seps^{\alpha} \to \left[\int_{-\infty}^x \partial_t s \diff y \right] \, \partial_x s^{\alpha} \mbox{ in } \mathcal{D}'((0,T)\times\R).
\end{equation}
Indeed, for any $\varphi \in C_c^{\infty}((0,T)\times\R)$ we have
\begin{multline*}
\int_0^T\int_{\R} \varphi \left[ \int_{-\infty}^x \partial_t \seps \diff y\right] \, \partial_x \seps^{\alpha} \diff x \diff t =\\ = -\int_0^T\int_{\R} \varphi \, \partial_t \seps\, \seps^{\alpha} \diff x \diff t -\int_0^T\int_{\R} \partial_x \varphi \left[ \int_{-\infty}^x \partial_t \seps \diff y\right] \, \seps^{\alpha} \diff x \diff t.
\end{multline*}
The first term converges because we can write $\partial_t \seps\, \seps^{\alpha}$ as $\frac{1}{\alpha+1} \partial_t \seps^{\alpha+1}$, then pass the $\partial_t$ derivative on the test function $\varphi$ and use the dominated convergence ($\seps^{1+\alpha}$ converges a.e. by \ref{conv:sum_strong}, it is uniformly bounded by \ref{est:conservation} while $\varphi$ has compact support). The second term uses weak convergence of $\left[ \int_{-\infty}^x \partial_t \seps \diff y\right]$ in \ref{conv:integrated_time_der_sum} and strong convergence of $\seps^{\alpha}\to s^{\alpha}$ in $L^2(\supp \varphi)$ (since $\supp \varphi$ is compact, the $L^{\infty}((0,T)\times\R)$ bound on $\{\seps\}$ in \ref{est:conservation} and its a.e. convergence in \ref{conv:sum_strong} is sufficient to apply the dominated convergence). Hence,
\begin{multline*}
\lim_{\eps\to0} \int_0^T\int_{\R} \varphi \left[ \int_{-\infty}^x \partial_t \seps \diff y\right] \, \partial_x \seps^{\alpha} \diff x \diff t = \\ = -\frac{1}{1+\alpha} \int_0^T\int_{\R} \varphi \, \partial_t s^{1+\alpha} \diff x \diff t -\int_0^T\int_{\R} \partial_x \varphi \left[ \int_{-\infty}^x \partial_t s \diff y\right] \, s^{\alpha} \diff x \diff t.
\end{multline*}
We can now integrate by parts backwards (i.e. invert all operations above) to arrive at \eqref{eq:result_about_convergence_product_integral_time_derivative_and_power_of_s} (rigorously, to write $\partial_t s^{1+\alpha} = (1+\alpha)\, \partial_t s \, s^{\alpha}$ we use that $\partial_t s \in L^2(0,T; H^{-1}(\R))$ by \ref{est:partial_t_neg_ss} while $s^{\alpha} \in L^2(0,T; H^1_{\loc}(\R))$ by \ref{est:crucial_estimate_gradient_power_alpha} and \ref{est:conservation} so their product makes sense locally and we can argue by an approximation argument; we need to work here locally as $s^{\alpha}$ may not be in $L^2((0,T)\times\R)$ for small $\alpha$). \\

We now explain how to pass to the limit in the sense of distributions in the terms on the (RHS) of \eqref{eq:integrated_in_x_PDE_for_seps_mult_by_der}. For the first we use \ref{conv_nabla_s_alpha_squared_in_measures}, for the second we exploit \ref{conv:velocity_fields} and \ref{conv_product_u_derivative_s_alpha}, for the third one we use \ref{conv:sum_strong}, \ref{conv:gradient_sum} and \ref{conv:velocity_fields},  while for the last one we observe by \ref{est:crucial_estimate_gradient_power_alpha} and~\ref{est:energy_derivatives_u_v_eps_signular}
$$
\eps\, \int_0^T \int_{\R} |\partial_x \seps\, \partial_x \seps^{\alpha}| \diff x \diff t \leq \eps^{\frac{1}{2}}\, \| \eps^{\frac{1}{2}} \,\partial_x \seps \|_{L^2((0,T)\times\R)} \, \| \partial_x \seps^{\alpha} \|_{L^2((0,T)\times\R)} \to 0.
$$ 
Hence, passing to the limit in \eqref{eq:integrated_in_x_PDE_for_seps_mult_by_der} in the sense of distributions we obtain
$$
\left[ \int_{-\infty}^x \partial_t s \diff y \right] \, \partial_x s^{\alpha} = \frac{1}{\alpha} \overline{|\partial_x s^{\alpha}|^2} + \overline{u\, \partial_x s^{\alpha}}\,  \partial_x (\mathcal{V}^1[u,v] - \mathcal{V}^2[u,v]) + s \, \partial_x s^{\alpha}\, \partial_x \mathcal{V}^2[u,v]
$$
and comparing it with \eqref{eq:integrated_in_x_PDE_for_s_mult_by_der} we obtain
\begin{equation}\label{eq:identity_between_weak_limits_lemma_proof_difference}
\frac{1}{\alpha}\left( \overline{|\partial_x s^{\alpha}|^2} -  |\partial_x s^{\alpha}|^2\right)  =  (\overline{u\, \partial_x s^{\alpha}} - u\, \partial_x s^{\alpha})\,  \partial_x (\mathcal{V}^2[u,v] - \mathcal{V}^1[u,v]).
\end{equation}
To conclude the proof, we write
$$
|\partial_x\seps^{\alpha} - \partial_x s^{\alpha}|^2 = | \partial_x\seps^{\alpha}| ^2 - 2\, \partial_x\seps^{\alpha}\,\partial_x s^{\alpha} + |\partial_x^{\alpha} s^{\alpha}|^2.
$$
Since $2\, \partial_x\seps^{\alpha}\,\partial_x s^{\alpha} \weak  2 \, |\partial_x s^{\alpha}|^2$ in $L^1((0,T)\times\R)$ (because of \ref{conv:gradient_sum} and $\partial_x s^{\alpha} \in L^2((0,T)\times\R)$), the desired representation of the limit follows from \eqref{eq:identity_between_weak_limits_lemma_proof_difference}. 
\end{proof}

\section{Young measure representations and compensated compactness}\label{sect:young_measures_compensated_compactness}

\subsection{Young measures}\label{subsect:young_measures} We recall that we define $\frac{\ueps}{{\seps}}=0$ when $\seps = 0$ and $\frac{u}{s} = 0$ when $s=0$ (this does not matter much because when it comes to identification of the limit $\frac{\ueps}{\seps}\partial_x \seps^{\alpha}$, the case $s(t,x)=0$ is easy and we will restrict the reasoning almost immediately to the case $s(t,x)>0$, see Step 1 in Proposition \ref{prop:support_YM_is_line}).\\

For each $(t,x)$ we define the measure $\pi_{t,x}$ on $\R \times \R$ to be the Young measure (see \cite{MR6023} for the original paper by Young, who addressed the lack of strong convergence in problems of the calculus of variations; \cite{MR1036070} for the work of Ball, who introduced this concept into the study of PDEs; and \cite[Chapter~6]{MR1452107} and \cite[Chapter~4]{MR3821514} for more comprehensive treatments) of the sequence
\begin{equation}\label{eq:sequence}
\left(\frac{\ueps}{{\seps}} - \frac{u}{{s}},  \partial_x s_{\eps}^{\alpha} -  \partial_x s^{\alpha} \right).
\end{equation}
To say that $\pi_{t,x}$ is the Young measure of the sequence in \eqref{eq:sequence} means that:
\begin{itemize}
\item for any Carathéodory function $f(t,x,\lambda_1, \lambda_2): [0,T]\times\R\times\R^2 \to \R$ (i.e. measurable in $(t,x)$ and continuous in $(\lambda_1, \lambda_2)$) we have that 
$$
(t,x) \mapsto \int_{\R^2} f(t,x,\lambda_1, \lambda_2) \diff \pi_{t,x}(\lambda_1, \lambda_2) \mbox{ is (Lebesgue) measurable,}
$$
\item for any Carathéodory function $f(t,x,\lambda_1, \lambda_2): [0,T]\times\R\times\R^2 \to \R$ such that $$
\Big\{f\Big(t, x, \frac{\ueps}{{\seps}} - \frac{u}{{s}}, \partial_x s_{\eps}^{\alpha} -  \partial_x s^{\alpha} \Big)\Big\}_{\eps\in(0,1)}
$$
is weakly compact in $L^1_{\loc}((0,T)\times\R)$, we have
\begin{equation}\label{eq:representation_weak_limits_general_sequence_caratheodory_integrand}
f\left(t,x,\frac{\ueps}{{\seps}} - \frac{u}{{s}}, \partial_x s_{\eps}^{\alpha} -  \partial_x s^{\alpha} \right) \weak  \int_{\R^2} f(t,x,\lambda_1, \lambda_2) \diff \pi_{t,x}(\lambda_1, \lambda_2)
\end{equation}
weakly in $L^1_{\loc}((0,T)\times\R)$ after passing to a converging subsequence. In particular, if we pass to a subsequence such that $f\Big(t, x, \frac{\ueps}{{\seps}} - \frac{u}{{s}}, \partial_x s_{\eps}^{\alpha} -  \partial_x s^{\alpha} \Big) \weak  \overline{f}$, then by extracting a further subsequence satisfying \eqref{eq:representation_weak_limits_general_sequence_caratheodory_integrand}, we deduce that $\overline{f}$ is given by \eqref{eq:representation_weak_limits_general_sequence_caratheodory_integrand}. 
\end{itemize}
According to \cite[Theorem 6.2]{MR1452107}, the Young measure of a sequence $\{z_\eps\}$ exists whenever $\sup_{\eps\in(0,1)} \int_0^T \int_{\R} g(|z_\eps|) \diff x \diff t < \infty$ for some $g:[0,\infty) \to [0,\infty]$ continuous, nondecreasing and $\lim_{t\to\infty} g(t) = \infty$ which is clearly satisfied in our case by \ref{est:crucial_estimate_gradient_power_alpha} in Proposition \ref{prop:uniform_estimates_viscosity} and~$\frac{\ueps}{{\seps}}\leq 1$.\\

We will also need some information when the sequence $\Big\{f\Big(\frac{\ueps}{{\seps}} - \frac{u}{{s}}, \partial_x s_{\eps}^{\alpha} -  \partial_x s^{\alpha} \Big)\Big\}_{\eps\in(0,1)}$ is only bounded in $L^1_{\loc}((0,T)\times\R)$. This is the argument taken from \cite[Lemma 2.1]{MR3567640} simplified to our setting.
\begin{lem}\label{lem:bound_by_the_YM_from_below_only_L1_estimate}
Suppose that $f:\R^2\to [0,\infty)$ is continuous and $\Big\{f\Big(\frac{\ueps}{{\seps}} - \frac{u}{{s}}, \partial_x s_{\eps}^{\alpha} -  \partial_x s^{\alpha} \Big)\Big\}_{\eps\in(0,1)}$ is bounded in $L^1_{\loc}((0,T)\times\R)$. If $\widetilde{f}$ is its weak$^*$ limit (in the sense of $\mathcal{M}_{\loc}([0,T]\times\R)$), then
\begin{equation}\label{eq:estimate_YM_in_case_of_only_L^1_estimate}
\widetilde{f} \geq  \int_{\R^2} f(\lambda_1, \lambda_2) \diff \pi_{t,x}(\lambda_1, \lambda_2) \mbox{ in } \mathcal{D}'((0,T)\times\R).
\end{equation}
\end{lem}
\begin{proof}
Let $\varphi \in C_c^{\infty}((0,T)\times\R)$ with $\varphi \geq 0$. By a diagonal method, we choose a subsequence such that $\Big\{f\Big(\frac{\ueps}{{\seps}} - \frac{u}{{s}}, \partial_x s_{\eps}^{\alpha} -  \partial_x s^{\alpha} \Big) \, \mathds{1}_{ f(\frac{\ueps}{{\seps}} - \frac{u}{{s}}, \partial_x s_{\eps}^{\alpha} -  \partial_x s^{\alpha})\leq k}\Big\}$ is weakly converging in $L^1_{\loc}((0,T)\times\R)$ to its Young measure representation \eqref{eq:representation_weak_limits_general_sequence_caratheodory_integrand} for each $k \in \N$. Then,
\begin{align*}
\int_0^T \int_{\R} &\varphi(t,x) \diff \widetilde{f}(t,x) = \lim_{\eps \to 0} \int_0^T \int_{\R} \varphi(t,x)\, f\Big(\frac{\ueps}{{\seps}} - \frac{u}{{s}},\, \partial_x s_{\eps}^{\alpha} -  \partial_x s^{\alpha} \Big) \diff x \diff t\\
&\geq \lim_{\eps \to 0}  \int_0^T \int_{\R} \varphi(t,x)\, f\Big(\frac{\ueps}{{\seps}} - \frac{u}{{s}}, \partial_x s_{\eps}^{\alpha} -  \partial_x s^{\alpha} \Big) \, \mathds{1}_{ f(\frac{\ueps}{{\seps}} - \frac{u}{{s}},\, \partial_x s_{\eps}^{\alpha} -  \partial_x s^{\alpha} )\leq k} \diff x \diff t \\
&= \int_0^T \int_{\R}  \int_{\R^2} \varphi(t,x) \, f(\lambda_1, \lambda_2)\, \mathds{1}_{f(\lambda_1, \lambda_2)\leq k} \diff \pi_{t,x}(\lambda_1, \lambda_2) \diff x \diff t.
\end{align*}
By monotone convergence we can pass to the limit $k \to \infty$ and this concludes the proof.
\end{proof}
From \eqref{eq:estimate_YM_in_case_of_only_L^1_estimate} and the estimate on $\{\partial_x \seps^{\alpha}\}_{\eps\in(0,1)}$ (\ref{est:crucial_estimate_gradient_power_alpha} in Proposition \ref{prop:uniform_estimates_viscosity}), we also observe that for a.e. $(t,x)$, $\pi_{t,x}$ has the following integrability
$$
 \int_{\R^2} (|\lambda_1|^k+1)\,|\lambda_2|^2 \diff \pi_{t,x}(\lambda_1, \lambda_2) < \infty \mbox{ for all $k \in \N$}.
$$

\subsection{Div-Curl Lemma and compensated compactness}\label{subsect:div-curl-lemma}

The idea is now to find some constraints on $\pi_{t,x}$. The main tool will be a version of Div-Curl Lemma \cite[p.54]{MR1034481} (see also \cite{MR506997, MR584398} for the original references) whose original statement is recalled below:
\begin{thm}[Div-Curl Lemma]\label{thm:div_curl_lemma}
Let $\{\bf a_\eps\}$, $\{\bf b_\eps\}$ be two sequences converging weakly in $L^2_{\loc}(\Omega;\R^k)$ to some ${\bf a}$, ${\bf b}$ for some $\Omega \subset \R^k$. Suppose moreover that $\{\DIV {\bf a_\eps} \}$ and $\{\CURL  {\bf b_\eps} \}$ lie in a compact subset of $H^{-1}_{\loc}(\Omega)$. Then, $a_\eps \cdot b_\eps \to a \cdot b$ as $\eps\to0$ in the sense of distributions.
\end{thm}

Typically, for two continuity equations
\begin{align*}
\partial_t U_{\eps} &= \partial_x F_{\eps} + S_{\eps},\\
\partial_t V_{\eps} &= \partial_x G_{\eps} + R_{\eps},
\end{align*}
with $\{U_{\eps}\}$, $\{V_{\eps}\}$, $\{F_{\eps}\}$, $\{G_{\eps}\}$ bounded in $L^2_{\loc}((0,T)\times \R)$ and $\{S_{\eps}\}$, $\{R_{\eps}\}$ strongly compact in $H^{-1}_{\loc}((0,T)\times\R)$, one applies Theorem \ref{thm:div_curl_lemma} with 
$$
a_\eps = (U_{\eps}, - F_{\eps}), \qquad b_{\eps} = (G_{\eps}, V_{\eps}).
$$
We will apply these ideas twice and for the second application we need to slightly relax the assumptions of Theorem~\ref{thm:div_curl_lemma}. The following variant is proved in \cite{MR2769903}.

\begin{thm}\label{thm:div_curl_lemma_version_CR}
Let $\{\bf a_\eps\}$, $\{\bf b_\eps\}$ be two sequences converging weakly in $L^2(\Omega;\R^k)$ to some ${\bf a}$, ${\bf b}$ for a bounded Lipschitz set $\Omega \subset \R^k$. Suppose moreover that $\{\DIV {\bf a_\eps} \}$ and $\{\CURL  {\bf b_\eps} \}$ lie in a compact subset of $(W^{1,\infty}_0(\Omega))^*$ and that the sequence $\{{\bf a_\eps} \cdot {\bf b_\eps}  \}$ is uniformly integrable in $L^1(\Omega)$:
$$
\forall \delta>0 \quad \exists \kappa>0: \quad   A\subset \Omega, |A| \leq \kappa \implies \sup_{\eps \in (0,1)} \int_{A} |{\bf a_\eps} \cdot {\bf b_\eps}| \diff x \leq \delta.
$$
Then, ${\bf a_\eps} \cdot {\bf b_\eps} \to {\bf a} \cdot {\bf b}$ as $\eps\to0$ in the sense of distributions.
\end{thm}

\subsection{Existence of solutions in terms of $\pi_{t,x}$}\label{subsect:existence_of_slns_in_lang_of_yms}

The following proposition shows that in order to prove that \eqref{eq:general_cross_diffusion_intro_viscosity} converges to \eqref{eq:general_cross_diffusion_intro_short_velocity}, it is sufficient to prove that the integral $ \int_{\R^2} \lambda_1\, \lambda_2 \diff \pi_{t,x}(\lambda_1,\lambda_2) = 0$.  
\begin{proposition}\label{prop:main_thm_by_Young_measures}
Let $\pi_{t,x}$ be the Young measure of the sequence \eqref{eq:sequence}. Let $u$, $s$, $\overline{\frac{u}{s}\,\partial_x s^{\alpha}}$ and $\overline{\frac{v}{s}\,\partial_x s^{\alpha}}$ be limits as in \ref{conv:sum_strong}, \ref{conv:u_v_weak} and \ref{conv_product_u_derivative_s_alpha_divided_by_s}. Then, for a.e. $(t,x) \in [0,\infty)\times\R$
\begin{equation}\label{eq:equivalence_convergence_in_the_weak_formulation_double_integral_wrt_pi_is_zero}
\overline{\frac{u}{s}\,\partial_x s^{\alpha}} - \frac{u}{s}\,\partial_x s^{\alpha} =  \int_{\R^2} \lambda_1\, \lambda_2 \diff \pi_{t,x}(\lambda_1,\lambda_2).
\end{equation}
In particular, suppose that for a.e. $(t,x) \in [0,\infty)\times\R$
\begin{equation}\label{eq:condition_passing_to_the_limit}
 \int_{\R^2} \lambda_1\, \lambda_2 \diff \pi_{t,x}(\lambda_1,\lambda_2) = 0.
\end{equation}
Then, $\overline{\frac{u}{s}\,\partial_x s^{\alpha}} = \frac{u}{s}\,\partial_x s^{\alpha}$, $\overline{\frac{v}{s}\,\partial_x s^{\alpha}} = \frac{v}{s}\,\partial_x s^{\alpha}$ and $(u,v)$ is a solution to \eqref{eq:general_cross_diffusion_intro_short_velocity} in the sense of Definition \ref{def:weak_sol}. 
\end{proposition}

\begin{proof}
We write 
$$
\frac{\ueps}{{\seps}} \, \partial_x \seps^{\alpha} - \frac{u}{s} \, \partial_x s^{\alpha} = \left(\frac{\ueps}{{\seps}} - \frac{u}{s} \right)\, \left(\partial_x \seps^{\alpha} -  \partial_x s^{\alpha}\right) +  \frac{u}{s}\, \left(\partial_x \seps^{\alpha} -  \partial_x s^{\alpha}\right) + \partial_x s^{\alpha}\,\left(\frac{\ueps}{{\seps}} - \frac{u}{s} \right).
$$
The last two terms converge in $L^2((0,T)\times\R)$ to 0 by \ref{conv:fractions} and \ref{conv:gradient_sum} (to be fully rigorous, in the second term we only know by \ref{conv:fractions} that $\frac{\ueps}{{\seps}} - \frac{u}{s} \weak 0$ on $\{s >0\}$ but on $\{s=0\}$ we have $\partial_x s^{\alpha} = 0$ by \ref{conv:der_s_alpha_zero_in_vacuum}). Moreover, the sequence $\big\{ \left(\frac{\ueps}{{\seps}} - \frac{u}{s} \right)\, \left(\partial_x \seps^{\alpha} -  \partial_x s^{\alpha}\right)\big\}_{\eps\in(0,1)}$ is weakly compact in $L^1_{\loc}((0,T)\times\R)$ (by the $L^2((0,T)\times\R)$ bound) so applying \eqref{eq:representation_weak_limits_general_sequence_caratheodory_integrand} we obtain
$$
\frac{\ueps}{{\seps}} \, \partial_x \seps^{\alpha} - \frac{u}{s} \, \partial_x s^{\alpha} \weak  \int_{\R^2} \lambda_1\, \lambda_2 \diff \pi_{t,x}(\lambda_1,\lambda_2).
$$
This proves \eqref{eq:equivalence_convergence_in_the_weak_formulation_double_integral_wrt_pi_is_zero}. Next, if \eqref{eq:condition_passing_to_the_limit} is satisfied then by \eqref{eq:equivalence_convergence_in_the_weak_formulation_double_integral_wrt_pi_is_zero}, $\overline{\frac{u}{s}\,\partial_x s^{\alpha}} = \frac{u}{s}\,\partial_x s^{\alpha}$. By \ref{conv_product_u_derivative_s_alpha_divided_by_s}, $\overline{\frac{v}{s}\,\partial_x s^{\alpha}} = \partial_x s^{\alpha}-\overline{\frac{u}{s}\,\partial_x s^{\alpha}}$ and we also have $\frac{v}{s}\,\partial_x s^{\alpha} = \partial_x s^{\alpha}-\frac{u}{s}\,\partial_x s^{\alpha}$ (this is clear when $s>0$ and when $s=0$ we have $\partial_x s^{\alpha}=0$ by \ref{conv:der_s_alpha_zero_in_vacuum}) so that we deduce $\overline{\frac{v}{s}\,\partial_x s^{\alpha}} = \frac{v}{s}\,\partial_x s^{\alpha}$. Now, the fact that $(u,v)$ is a weak solution to \eqref{eq:general_cross_diffusion_intro_short_velocity} is straightforward since $\frac{\ueps}{{\seps}} \, \partial_x \seps^{\alpha}$ and $\frac{\veps}{{\seps}} \, \partial_x \seps^{\alpha}$ are the only nonlinear terms in \eqref{eq:general_cross_diffusion_intro_viscosity}.
\end{proof}

\section{The support of $\pi_{t,x}$ and the proof of Theorem \ref{thm:main}}\label{sect:proof_main_result} 
By Proposition \ref{prop:main_thm_by_Young_measures}, Theorem \ref{thm:main} will be proved if we establish \eqref{eq:condition_passing_to_the_limit}. The main idea is to characterize the support of $\pi_{t,x}$.

\subsection{The support of $\pi_{t,x}$ is (almost always) at most a line}\label{subsect:support_is_a_line} We define
\begin{equation}\label{eq:definition_quantity_F_necessary_later}
\mathcal{F}(t,x) :=  \alpha\,s\,\partial_x (\mathcal{V}^2[u,v] - \mathcal{V}^1[u,v]),
\end{equation}
where $u, v, s$ are limits in \ref{conv:sum_strong}, \ref{conv:u_v_weak} and we will prove the following.
\begin{proposition}\label{prop:support_YM_is_line}
Let $\pi_{t,x}$ be the Young measure of the sequence \eqref{eq:sequence}. Then, 
\begin{itemize}
\item if $s(t,x)=0$ then $\pi_{t,x}$ is supported at most on $\R\times \{0\}$,
\item if $\mathcal{F}(t,x) = 0$ and $s(t,x)>0$ then $\overline{\frac{u}{s}\,\partial_x s^{\alpha}} - \frac{u}{s}\,\partial_x s^{\alpha} = \int_{\R^2} \lambda_1\, \lambda_2 \diff \pi_{t,x}(\lambda_1,\lambda_2)=0$,
\item if $s(t,x)>0$ and $\mathcal{F}(t,x) \neq 0$, the $\supp \pi_{t,x}$ is contained in $\{0\}\times\R$ or in the set $\{(\lambda_1, \mathcal{F}\,\lambda_1): \lambda_1\in\R\}$.
\end{itemize}

\end{proposition}
Note that only the last scenario (when $\supp \pi_{t,x}$ is contained in $\{(\lambda_1, \mathcal{F}\,\lambda_1)\}$) does not lead directly to \eqref{eq:condition_passing_to_the_limit} and needs further investigation in Section \ref{subsect:the_case_that_YM_supp_on_a_nonzero_line}. 
\begin{proof}[Proof of Proposition \ref{prop:support_YM_is_line}] We split the proof into three steps.\\

\underline{Step 1. The case $s(t,x)=0$.} We apply \eqref{eq:representation_weak_limits_general_sequence_caratheodory_integrand} with $f(t,x,\lambda_1,\lambda_2) = \mathds{1}_{s(t,x)=0}\,|\lambda_2|$ which is allowed since by \ref{conv:der_s_alpha_zero_in_vacuum}, $\mathds{1}_{s=0}\, |\partial_x \seps^{\alpha}-\partial_x s^{\alpha}| \to 0$ strongly in $L^1((0,T)~\times~\R)$. Hence,
$$
|\partial_x \seps^{\alpha} - \partial_x s^{\alpha}| \, \mathds{1}_{s(t,x)=0} \weak  \int_{\R^2} |\lambda_2|\, \mathds{1}_{s(t,x)=0} \diff \pi_{t,x}(\lambda_1,\lambda_2) =0. 
$$
We conclude that for $(t,x)$ such that $s(t,x)=0$, $\pi_{t,x}$ is supported at most on $\R\times \{0\}$.

\underline{Step 2. The case $s(t,x)>0$ and $\mathcal{F}(t,x) = 0$.} In addition to considering the case, we will also prove that for a.e. $(t,x)$ such that $s(t,x)>0$ we have
\begin{equation}\label{eq:final_identity_from_1st_application_div_curl_lemma}
\left(\frac{u}{s}\, \partial_x s^{\alpha} - \overline{\frac{u}{s}\,\partial_x s^{\alpha}} \right) = \mathcal{F}\, \left(\frac{u^2}{s^2} - \overline{\,\frac{u^2}{s^2}\,}\right),
\end{equation}
where $\overline{\,\frac{u^2}{s^2}\,}$ is the weak limit from \ref{conv:fractions}. To this end, we want to apply Theorem~\ref{thm:div_curl_lemma} locally on the space $(0,T)\times\R \subset \R^2$ (locally means that we apply it on every compact set which is sufficient to identify the limits) with variables $(t,x)$ and with vector fields
$$
{\bf a_\eps} = \left(\ueps, -\frac{1}{\alpha}\,\frac{\ueps}{\seps}\,\partial_x \seps^{\alpha} -\ueps \, \partial_x \mathcal{V}^1[\ueps,\veps] \right),\qquad
{\bf b_\eps} = \left(\frac{1}{\alpha}\,\frac{\veps}{\seps}\,\partial_x \seps^{\alpha} +\veps \, \partial_x \mathcal{V}^2[\ueps,\veps], \veps \right). 
$$
We have
$$
\DIV {\bf a_\eps} = \eps\, \partial^2_x \ueps, \qquad \qquad \CURL {\bf b_\eps} = -\eps\, \partial^2_x \veps, 
$$
and both of them are compact in $H^{-1}((0,T)\times\R)$ because by \ref{est:energy_derivatives_u_v_eps_signular}, $\eps\,\partial_x\ueps, \eps\,\partial_x \veps\to 0$ strongly in $L^2((0,T)\times\R)$. According to \ref{conv:u_v_weak}, \ref{conv:velocity_fields} and \ref{conv_product_u_derivative_s_alpha_divided_by_s}, the vector fields converge weakly in $L^2_{\loc}((0,T)\times\R)$ to 
$$
{\bf a} = \left(u, -\frac{1}{\alpha}\,\overline{\frac{u}{s}\,\partial_x s^{\alpha}} -u \, \partial_x \mathcal{V}^1[u,v] \right), \qquad \qquad  {\bf b} = \left(\frac{1}{\alpha}\,\overline{\frac{v}{s}\,\partial_x s^{\alpha}} + v \, \partial_x \mathcal{V}^2[u,v], v \right).
$$
Using identities $v = s-u$ and $\overline{\frac{u}{s}\,\partial_x s^{\alpha}}+ \overline{\frac{v}{s}\,\partial_x s^{\alpha}} = \partial_x s^{\alpha}$ (see \ref{conv_product_u_derivative_s_alpha_divided_by_s} in Proposition \ref{prop:uniform_estimates_viscosity}), we can compute 
\begin{align*}
{\bf a} \cdot {\bf b} =\,& u\, \left(\frac{1}{\alpha}\,\overline{\frac{v}{s}\,\partial_x s^{\alpha}} + v \, \partial_x \mathcal{V}^2[u,v]\right) - \left(\frac{1}{\alpha}\,\overline{\frac{u}{s}\,\partial_x s^{\alpha}} +u \, \partial_x \mathcal{V}^1[u,v] \right) \,v\\
=\,& \frac{1}{\alpha}\,\left(u\, \partial_x s^{\alpha} - u\, \overline{\frac{u}{s}\,\partial_x s^{\alpha}} - \overline{\frac{u}{s}\,\partial_x s^{\alpha}} \,(s-u) \right)  + u\,v \, \partial_x(\mathcal{V}^2[u,v]- \mathcal{V}^1[u,v])\\
=\,& \frac{1}{\alpha}\,\left(u\, \partial_x s^{\alpha} - \overline{\frac{u}{s}\,\partial_x s^{\alpha}} \,s \right)  + u\,v \, \partial_x(\mathcal{V}^2[u,v]- \mathcal{V}^1[u,v]). 
\end{align*}
Moreover, by \ref{conv:velocity_fields} and \ref{conv:square_product_uv} in Proposition \ref{prop:uniform_estimates_viscosity}
$$
{\bf a_\eps}\cdot {\bf b_\eps} = \ueps \, \veps \, \partial_x (\mathcal{V}^2[\ueps,\veps] - \mathcal{V}^1[\ueps,\veps]) \weak \overline{u\,v} \, \partial_x (\mathcal{V}^2[u,v] - \mathcal{V}^1[u,v]).
$$
Hence, the Div-Curl Lemma (Theorem \ref{thm:div_curl_lemma}) implies
$$
\frac{1}{\alpha}\,\left(u\, \partial_x s^{\alpha} - \overline{\frac{u}{s}\,\partial_x s^{\alpha}} \,s \right) = (\overline{u\,v} - u\,v) \, \partial_x (\mathcal{V}^2[u,v] - \mathcal{V}^1[u,v]).
$$
Finally, we observe that by \ref{conv:square_product_uv} we know that the sequence $\{\ueps^2\}$ converges weakly to some $\overline{u^2}$. Writing $v=s-u$ and using the strong convergence of $\{\seps\}$, we have $\overline{u\,v} - u\,v = u^2 - \overline{u^2}$ and we can simplify
$$
\frac{1}{\alpha}\,\left(u\, \partial_x s^{\alpha} - \overline{\frac{u}{s}\,\partial_x s^{\alpha}} \,s \right) = (u^2 - \overline{u^2}) \, \partial_x (\mathcal{V}^2[u,v] - \mathcal{V}^1[u,v]).
$$
By Step 1, we consider this pointwise identity at $(t,x)$ such that $s(t,x)>0$. Dividing by $s$ and multiplying by $\alpha$ we obtain \eqref{eq:final_identity_from_1st_application_div_curl_lemma} (we also used here that for $(t,x)$ such that $s(t,x)>0$, we have $\overline{u^2/s^2} = \frac{1}{s^2}\,\overline{u^2}$, cf. \ref{conv:fractions}). We now observe that when $\mathcal{F} = 0$, $\frac{u}{s}\, \partial_x s^{\alpha} = \overline{\frac{u}{s}\,\partial_x s^{\alpha}}$ so we are done by \eqref{eq:equivalence_convergence_in_the_weak_formulation_double_integral_wrt_pi_is_zero}.\\

\underline{Step 3. The case $s(t,x)>0$ and $\mathcal{F} \neq 0$.} We note that identity \eqref{eq:final_identity_from_1st_application_div_curl_lemma} is valid for a.e. $(t,x)$ such that $s(t,x)>0$ (we did not use $\mathcal{F} = 0$ in its derivation). We now observe that
\begin{equation}\label{eq:difference_weak_limit_fraction_squared_in_terms_of_YM}
\overline{\,\frac{u^2}{s^2}\,} - \frac{u^2}{s^2} =  \int_{\R^2} \lambda_1^2 \diff \pi_{t,x}(\lambda_1, \lambda_2).
\end{equation}
Indeed, we write
$$
\frac{\ueps^2}{\seps^2} - \frac{u^2}{s^2} = \left(\frac{\ueps}{\seps}-\frac{u}{s}\right)^2 +2 \left(\frac{\ueps}{\seps}-\frac{u}{s}\right)\, \frac{u}{s}.
$$
The last term on the (RHS) converges to 0 by \ref{conv:fractions} (again, by \ref{conv:fractions} we know the weak limit of $\frac{\ueps}{\seps}$ only for $s>0$ but for $s=0$ we can use that $\frac{u}{s}=0$). Hence, by \eqref{eq:representation_weak_limits_general_sequence_caratheodory_integrand} we obtain~\eqref{eq:difference_weak_limit_fraction_squared_in_terms_of_YM}.\\

Now, we combine \eqref{eq:equivalence_convergence_in_the_weak_formulation_double_integral_wrt_pi_is_zero}, \eqref{eq:final_identity_from_1st_application_div_curl_lemma} and \eqref{eq:difference_weak_limit_fraction_squared_in_terms_of_YM} to obtain
\begin{equation}\label{eq:equality_between_2_YM_integrals_lambda1_lambda12}
\int_{\R^2} \lambda_1\,\lambda_2 \diff \pi_{t,x}(\lambda_1, \lambda_2)= \mathcal{F}\,\int_{\R^2} \lambda_1^2 \diff \pi_{t,x}(\lambda_1, \lambda_2).
\end{equation}
Moreover, from Proposition \ref{prop:weak_compact_partial_x_seps} and Lemma \ref{lem:bound_by_the_YM_from_below_only_L1_estimate} we obtain
$$
\int_{\R^2} \lambda_2^2 \diff \pi_{t,x}(\lambda_1, \lambda_2) \leq \alpha\,(\overline{u\, \partial_x s^{\alpha}} - u\, \partial_x s^{\alpha})\,  \partial_x (\mathcal{V}^2[u,v] - \mathcal{V}^1[u,v]).
$$
By \ref{conv_product_u_derivative_s_alpha_divided_by_s} in Proposition \ref{prop:uniform_estimates_viscosity}, we can write $\overline{u\, \partial_x s^{\alpha}} - u\, \partial_x s^{\alpha} = s\,(\overline{\frac{u}{s}\, \partial_x s^{\alpha}} - \frac{u}{s}\, \partial_x s^{\alpha})$ for $s>0$ so that using \eqref{eq:equivalence_convergence_in_the_weak_formulation_double_integral_wrt_pi_is_zero} and the definition of $\mathcal{F}$ in \eqref{eq:definition_quantity_F_necessary_later} we deduce
\begin{equation}\label{eq:inequality_between_2_YM_integrals_lambda2_lambda12}
\int_{\R^2} \lambda_2^2 \diff \pi_{t,x}(\lambda_1, \lambda_2) \leq \mathcal{F}\, \int_{\R^2} \lambda_1\,\lambda_2 \diff \pi_{t,x}(\lambda_1, \lambda_2).
\end{equation}
Now, \eqref{eq:equality_between_2_YM_integrals_lambda1_lambda12} and \eqref{eq:inequality_between_2_YM_integrals_lambda2_lambda12} imply
$$
\int_{\R^2} \lambda_1^2 \diff \pi_{t,x}(\lambda_1, \lambda_2)  \, \int_{\R^2} \lambda_2^2 \diff \pi_{t,x}(\lambda_1, \lambda_2) \leq  \left|\int_{\R^2} \lambda_1\,\lambda_2 \diff \pi_{t,x}(\lambda_1, \lambda_2)\right|^2
$$
and so, by the Cauchy-Schwarz inequality we finally arrive at 
\begin{equation}\label{eq:equality_Cauchy_Schwarz_between_YM_FINAL}
\left|\int_{\R^2} \lambda_1\,\lambda_2 \diff \pi_{t,x}(\lambda_1, \lambda_2)\right|^2 = \int_{\R^2} \lambda_1^2 \diff \pi_{t,x}(\lambda_1, \lambda_2)  \,\int_{\R^2} \lambda_2^2 \diff \pi_{t,x}(\lambda_1, \lambda_2).  
\end{equation}
It is well-known that the equality in the Cauchy-Schwarz inequality occurs only in some very particular cases which we will exploit. More precisely, we will prove that either $\int_{\R^2} \lambda_1^2 \diff \pi_{t,x}(\lambda_1, \lambda_2)= 0$ (which clearly implies that $\supp \pi_{t,x} \subset \{0\}\times\R$) or $\lambda_2 = \mathcal{F}\, \lambda_1$ on $\supp \pi_{t,x}$. Indeed, if $\int_{\R^2} \lambda_1^2 \diff \pi_{t,x}(\lambda_1, \lambda_2)>0$
\begin{align*}
\int_{\R^2} |\lambda_2 - \mathcal{F}\,\lambda_1|^2 &\diff \pi_{t,x}(\lambda_1, \lambda_2) =\\
&= \int_{\R^2} \lambda_2^2\diff \pi_{t,x}(\lambda_1, \lambda_2) + \mathcal{F}^2 \int_{\R^2}  \,\lambda_1^2\diff \pi_{t,x}(\lambda_1, \lambda_2)  - 2\,\mathcal{F} \int \lambda_1\,\lambda_2  \diff \pi_{t,x}(\lambda_1, \lambda_2)\\
&= \int_{\R^2} \lambda_2^2\diff \pi_{t,x}(\lambda_1, \lambda_2) - \mathcal{F}\, \int \lambda_1\,\lambda_2  \diff \pi_{t,x}(\lambda_1, \lambda_2)\\
&= \int_{\R^2} \lambda_2^2\diff \pi_{t,x}(\lambda_1, \lambda_2) - \frac{\left|\int_{\R^2} \lambda_1\,\lambda_2 \diff \pi_{t,x}(\lambda_1, \lambda_2)\right|^2}{\int_{\R^2} \lambda_1^2\diff \pi_{t,x}(\lambda_1, \lambda_2)} = 0,
\end{align*}
where we used \eqref{eq:equality_between_2_YM_integrals_lambda1_lambda12} twice and the equality \eqref{eq:equality_Cauchy_Schwarz_between_YM_FINAL}. The proof is concluded.
\end{proof}

\subsection{The case of $\pi_{t,x}$ supported on $\{(\lambda_1, \mathcal{F}\,\lambda_1): \lambda_1\in\R\}$ and proof of Theorem \ref{thm:main}.}\label{subsect:the_case_that_YM_supp_on_a_nonzero_line} We now consider the possibility that $\pi_{t,x}$ is supported on the line $\{(\lambda_1, \mathcal{F}\,\lambda_1): \lambda_1\in\R\}$ and show that the support is in fact $(0,0)$. 

\begin{proposition}\label{prop:the_case_of_support_on_the_line_identification_of_YM}
Let $\pi_{t,x}$ be the Young measure of the sequence \eqref{eq:sequence} and suppose that $(t,x)$ is such that $\pi_{t,x}$ is supported on the line $\{(\lambda_1, \mathcal{F}\,\lambda_1): \lambda_1\in\R\}$ with $\mathcal{F}$ defined by \eqref{eq:definition_quantity_F_necessary_later}, $\mathcal{F}\neq 0$ and $s(t,x)>0$. Then, $\int_{\R^2} \lambda_1^2 \diff \pi_{t,x}(\lambda_1, \lambda_2) =0$ so that $\pi_{t,x} = \delta_{(0,0)}$. 
\end{proposition} 
We will need two technical lemmas.
\begin{lem}\label{lem:upper_bound_on_the_x-y_log_xy}
Let $x,y \geq 0$. Then, ${2\, x\,y}\,(x-y)\, \log\left(\frac{x}{y}\right) \leq {(x-y)^2\,(x+y)}$.
\end{lem}
\begin{proof}
If $x=0$ or $y=0$, the inequality is trivial so we assume $x,y>0$. Moreover, both sides do not change when we swap $x$ and $y$ so we may assume that $x\geq y$. We let $t =\frac{x}{y}$ so dividing the inequality by $y^3$, it is equivalent to
$$
2\,t\,(t-1)\log t \leq {(t-1)^2\,(t+1)}
$$
for $t \geq 1$. If $t = 1$, the inequality is satisfied so we divide by $(t-1)$ and we have to prove $2t \log t \leq (t-1)\,(t+1) = t^2-1$. Let $f(t) = t^2-1-2\,t \log t$. We have $f(1) = 0$ and $f'(t) =2t - 2 - 2\log t$ so for $t\geq 1$ we have $f'(t)\geq 0$ as desired.
\end{proof}
\begin{lem}\label{lem:general_identity_to_replace_partialx_ps_with_just_u} Let $f:[0,\infty)^2 \to \R$ be a continuous function, $\{\ueps\}$, $\{\veps\}$ be the sequences of solutions to \eqref{eq:general_cross_diffusion_intro_viscosity} under Assumption \ref{ass:velocity_kernels_main}. Suppose that $f(\ueps, \veps) \weak \overline{f(u,v)}$, $f(\ueps, \veps)\, \partial_x \seps^{\alpha} \weak \overline{f(u,v)\, \partial_x s^{\alpha}}$, $\ueps\, f(\ueps, \veps) \weak \overline{u\, f(u,v)}$, $f(\ueps, \veps)\, \frac{\ueps}{\seps} \weak \overline{f(u,v)\, \frac{u}{s}}$ weakly in $L^1_{\loc}((0,T)\times\R)$. Let $(t,x)$ be a point as in Proposition~\ref{prop:the_case_of_support_on_the_line_identification_of_YM} such that $s(t,x)>0$, $\mathcal{F}\neq 0$, $\lambda_2 = \mathcal{F}\, \lambda_1$ for $\pi_{t,x}$-a.e. $(\lambda_1, \lambda_2)$. Then, at $(t,x)$, we have  
\begin{equation}\label{eq:limits_relationship_supp_restricted_line_before_logarithmic_analysis_lemma}
s\, \left(\overline{f(u,v) \, \partial_x s^{\alpha}} - \overline{ f\left(u,v\right)}\, \partial_x s^{\alpha} \right)= \mathcal{F} \, (\overline{ f(u, v) \, u} - \overline{f(u,v)}\, u).
\end{equation}
\end{lem}
\begin{proof}
Since $f$ is continuous, we can extend it to a continuous function on $\R^2$ by letting $f(u,v) = 0$ for $u,v\leq 0$, $f(u,v)=f(0,v)$ for $u\leq 0$, $v\geq 0$ and $f(u,v)=f(u,0)$ for $v\leq 0$, $u\geq 0$. Note that
$$
\Delta(t,x):= \left| f(\ueps, \veps) - f\left(\frac{\ueps}{\seps}\, s, \frac{\veps}{\seps}\,s\right) \right|
$$
is uniformly bounded in $L^{\infty}((0,T)\times\R)$ and $\Delta(t,x) \leq \omega_{f}(|\seps-s|) \to 0$ a.e. with $\omega_f$ being a modulus of continuity of $f$ on some bounded set $[0,M]^2$ such that $\seps \leq M$ on $[0,T]\times\R$. By the dominated convergence, $\Delta \to 0$ strongly in $L^2_{\loc}((0,T)\times\R)$, so we can replace $f(\ueps, \veps)$ with $ f\left(\frac{\ueps}{\seps}\, s, \frac{\veps}{\seps}\,s\right)$ in all the sequences whose limits appear in \eqref{eq:limits_relationship_supp_restricted_line_before_logarithmic_analysis_lemma} (this uses that $\{\partial_x \seps^{\alpha}\}$ and $\{\ueps\}$ are bounded in $L^2((0,T)\times\R)$ by \ref{est:crucial_estimate_gradient_power_alpha} and \ref{est:conservation}). We write
$$
f\left(\frac{\ueps}{\seps}\, s, \frac{\veps}{\seps}\,s\right) = 
f\left(\left(\frac{\ueps}{\seps} - \frac{u}{s}\right)\, s + u, \left(-\frac{\ueps}{\seps}+\frac{u}{s}\right)\,s + v\right).
$$
Therefore,
$$
f\left(\frac{\ueps}{\seps}\, s, \frac{\veps}{\seps}\,s\right) \, (\partial_x \seps^{\alpha} - \partial_x s^{\alpha}) \weak
\int_{\R^2} f(\lambda_1\,s + u,  - \lambda_1 s + v)\, \lambda_2  \diff \pi_{t,x}(\lambda_1, \lambda_2).
$$
Similarly, 
$$
f\left(\frac{\ueps}{\seps}\, s, \frac{\veps}{\seps}\,s\right) \, \left(\frac{\ueps}{\seps} - \frac{u}{s}\right) \weak 
\int_{\R^2} f(\lambda_1\,s + u, - \lambda_1 s + v)\,  \lambda_1 \diff \pi_{t,x}(\lambda_1, \lambda_2).
$$
The claim follows by using $\lambda_2 = \mathcal{F}\,\lambda_1$ and observing that $s\, \overline{f(u, v)\, \frac{u}{s}} = \overline{f(u,v)\, u}$ (which can be deduced by considering the equality $\seps \, f(\ueps, \veps)\, \frac{\ueps}{\seps} = f(\ueps, \veps)\, \ueps$ and using the strong convergence of $\{\seps\}$ from \ref{conv:sum_strong}).
\end{proof}
\begin{proof}[Proof of Proposition \ref{prop:the_case_of_support_on_the_line_identification_of_YM}] We split the reasoning for several steps.\\

\underline{Step 1. Identity between weak limits.} We will prove that:
\begin{equation}\label{eq:proof_identification_YM_step1_algebraic_identity_between_weak_limits_logarithmic_variables}
\overline{u\, \partial_x s^{\alpha}}\, \overline{E}  - u\, \overline{E  \,\partial_x s^{\alpha}} + s\, \overline{u  \,\partial_x s^{\alpha}} - u\, s \,\partial_x s^{\alpha}  + \mathcal{F}\, ( \overline{u\, v\log v} - u\, \overline{v  \log v} )  = 0,
\end{equation}
where $\overline{u\, \partial_x s^{\alpha}}$, $\overline{E}$, $\overline{E  \,\partial_x s^{\alpha}}$, $\overline{u\, v\log v}$, $\overline{v  \log v}$ are weak limits as in \ref{conv_logarithmic_fcns}, \ref{conv_product_u_derivative_s_alpha}, \ref{conv:ulogu_vlogv_times_der}. We apply Theorem~\ref{thm:div_curl_lemma_version_CR} to the vector fields
\begin{align*}
{\bf a_\eps} &= \left(\ueps, -\ueps\, \partial_x p(\seps) -\ueps \, \partial_x \mathcal{V}^1[\ueps,\veps] \right),\\
{\bf b_\eps} &= \left((E_{\eps} + \seps)\,\partial_x p(\seps) + \ueps  \log \ueps\, \partial_x \mathcal{V}^1[\ueps, \veps] + \veps  \log \veps\, \partial_x \mathcal{V}^2[\ueps, \veps], E_{\eps} \right). 
\end{align*}
Let us explain why the assumptions of Theorem \ref{thm:div_curl_lemma_version_CR} are satisfied. Clearly, both $\{\bf a_\eps\}$ and $\{\bf b_\eps\}$ converge weakly in $L^2_{\loc}((0,T)\times\R)$ to some ${\bf a}$ and ${\bf b}$ by
\ref{conv:sum_strong}, \ref{conv:u_v_weak}, \ref{conv:velocity_fields}, \ref{conv_logarithmic_fcns}, \ref{conv_product_u_derivative_s_alpha}, and \ref{conv:ulogu_vlogv_times_der} in Proposition~\ref{prop:basic_convergences}. Moreover, $\DIV {\bf a_\eps} =\eps\, \partial^2_x \ueps$ converges strongly to 0 in $H^{-1}((0,T)\times\R)$ because, by \ref{est:energy_derivatives_u_v_eps_signular}, $\eps\,\partial_x\ueps, \eps\,\partial_x \veps\to 0$ strongly in $L^2((0,T)\times\R)$. Regarding $\CURL {\bf b_\eps}$, we have $\CURL {\bf b_\eps} = - g_{\eps} - h_{\eps}$, where $\{g_{\eps}\}$ and $\{h_{\eps}\}$ are defined in Proposition~\ref{prop:new_cons_law_log_general_alpha}. We know that $\{g_{\eps}\}$ is bounded in $L^1((0,T)\times\R)$, which compactly embeds into $W^{-1,p}_{\loc}((0,T)\times\R)$ for $p<2$ \cite[Theorem 6, p.~7]{MR1034481}, while $\{h_{\eps}\}$ is strongly compact in $H^{-1}((0,T)\times\R)$. Combining these two facts, $\{\CURL {\bf b_\eps}\}$ is strongly compact in $(W^{1,\infty}_0(\Omega))^*$ for each bounded $\Omega \subset [0,T]\times\R$. Finally, the product $\{{\bf a_\eps} \cdot {\bf b_\eps}\}$ is uniformly integrable in $L^1(\Omega)$ for $\Omega$ as above. Indeed, writing ${\bf a_\eps} = (a^1_{\eps}, a^2_{\eps})$ and ${\bf b_\eps} = (b^1_{\eps}, b^2_{\eps})$, $\{a^1_{\eps}\}$ and $\{b^2_{\eps}\}$ are bounded in $L^{\infty}((0,T)\times\R)$ by \ref{est:conservation} and \ref{est:ulogu_vlogv}, while $\{a^2_{\eps}\}$ and $\{b^1_{\eps}\}$ are bounded at least in $L^{2}_{\loc}((0,T)\times\R)$ by \ref{conv_product_u_derivative_s_alpha}, \ref{conv:ulogu_vlogv_times_der}, and \ref{est:ulogu_vlogv}. It follows that $\{{\bf a_\eps} \cdot {\bf b_\eps}\}$ is bounded in $L^2(\Omega)$, which implies uniform integrability in $L^1(\Omega)$. \\

We compute
\begin{align*}
{\bf a_\eps} \cdot {\bf b_\eps} =\,& \ueps\, ((E_{\eps} + \seps)\,\partial_x p(\seps) + \ueps  \log \ueps\, \partial_x \mathcal{V}^1[\ueps, \veps] + \veps  \log \veps\, \partial_x \mathcal{V}^2[\ueps, \veps]) \\
&-  (\ueps\, \partial_x p(\seps) +\ueps \, \partial_x \mathcal{V}^1[\ueps,\veps])\, E_{\eps}\\
=\,& - \ueps (E_{\eps}\, \partial_x \mathcal{V}^1[\ueps,\veps]  \!-\! \seps \,\partial_x p(\seps) \!-\!  \ueps  \log \ueps\, \partial_x \mathcal{V}^1[\ueps, \veps] \!-\! \veps  \log \veps\, \partial_x \mathcal{V}^2[\ueps, \veps])\\
=\,& - \ueps\, \veps\log\veps\, \partial_x (\mathcal{V}^1[\ueps,\veps]-\mathcal{V}^2[\ueps,\veps]) + \ueps \,  \seps \,\partial_x p(\seps).
\end{align*}
By \ref{conv:sum_strong}, \ref{conv:velocity_fields}, \ref{conv_logarithmic_fcns} and \ref{conv_product_u_derivative_s_alpha}, its limit in the sense of distributions equals
$$
{\bf a_\eps} \cdot {\bf b_\eps}  \weak
-\,\overline{u\, v\log v}\, \partial_x (\mathcal{V}^1[u,v]-\mathcal{V}^2[u,v]) + s\, \overline{u  \,\partial_x p(s)}.
$$
The convergences above together with \ref{conv:u_v_weak}, \ref{conv:ulogu_vlogv_times_der} allow us to identify also products of the weak limits ${\bf a} = (a^1, a^2)$, ${\bf b} = (b^1, b^2)$ as 
$$
a^1\, b^1 = u\,( \overline{E  \,\partial_x p(s)} + s \,\partial_x p(s) + \overline{u \log u}\, \partial_x \mathcal{V}^1[u,v] + \overline{v  \log v} \, \partial_x \mathcal{V}^2[u,v] ),
$$
$$
{a^2} \, b^2 = -(\overline{u\, \partial_x p(s)} +u \, \partial_x \mathcal{V}^1[u,v] ) \, \overline{E}.
$$
Expanding $\overline{E}$, we obtain
$$
{\bf a} \cdot {\bf b} = 
-\overline{u\, \partial_x p(s)}\, \overline{E}  + u\, \overline{E  \,\partial_x p(s)} + u\, s \,\partial_x p(s) - u\, \overline{v  \log v} \, \partial_x (\mathcal{V}^1[u,v] - \mathcal{V}^2[u,v]).
$$
By Theorem \ref{thm:div_curl_lemma_version_CR}
\begin{multline*}
\overline{u\, \partial_x p(s)}\, \overline{E}  - u\, \overline{E  \,\partial_x p(s)} + s\, \overline{u  \,\partial_x p(s)} - u\, s \,\partial_x p(s) \\ + ( u\, \overline{v  \log v} - \overline{u\, v\log v}) \, \partial_x (\mathcal{V}^1[u,v] - \mathcal{V}^2[u,v])  = 0.
\end{multline*}
By \ref{conv_product_u_derivative_s_alpha} and \ref{conv:ulogu_vlogv_times_der}, $\alpha\,s\, \overline{u\, \partial_x p(s)} = \overline{u\, \partial_x s^{\alpha}}$ and $\alpha\,s\,\overline{E  \,\partial_x p(s)} = \overline{E  \,\partial_x s^{\alpha}}$. We also have $\alpha\,s \, u\, \partial_x p(s)= u\, \partial_x s^{\alpha}$ by a direct computation. Plugging in these formulas, multiplying by $\alpha\,s$ and recalling the definition of the constant $\mathcal{F}$ in \eqref{eq:definition_quantity_F_necessary_later} we obtain \eqref{eq:proof_identification_YM_step1_algebraic_identity_between_weak_limits_logarithmic_variables}.\\

\underline{Step 2. An identity between weak limits without $\partial_x s^{\alpha}$.} We now eliminate $\partial_x s^{\alpha}$ from \eqref{eq:proof_identification_YM_step1_algebraic_identity_between_weak_limits_logarithmic_variables} by proving
\begin{equation}\label{eq:proof_identification_YM_step2_algebraic_identity_between_weak_limits_logarithmic_variables}
(\overline{E} + s) \,(\overline{u^2} -  u^2) + s\,( \overline{u\, v\log v} - u\, \overline{v  \log v})  - u\,(\overline{u\, E} -  u\, \overline{E})  = 0,
\end{equation}
where $\overline{u^2}$ is the weak limit as in \ref{conv:square_product_uv}. To this end, by Lemma \ref{lem:general_identity_to_replace_partialx_ps_with_just_u} applied to the function $f(u,v)=u\log u + v\log v$ (the assumptions of the lemma are satisfied by \ref{conv_logarithmic_fcns} and \ref{conv:ulogu_vlogv_times_der}), we have
$$
s\, \overline{E  \,\partial_x s^{\alpha}} = 
s\, \overline{E}  \,\partial_x s^{\alpha} +  \mathcal{F} \,(\overline{u\, E} -  u\, \overline{E}).
$$
We multiply by $u$ so that the first two terms in \eqref{eq:proof_identification_YM_step1_algebraic_identity_between_weak_limits_logarithmic_variables} (multiplied by $s$) can be written as
$$
s\,\overline{u\, \partial_x s^{\alpha}}\, \overline{E}  - s\,u\, \overline{E  \,\partial_x s^{\alpha}}= s\, \overline{E}\,(\overline{u\, \partial_x s^{\alpha}} - {u\, \partial_x s^{\alpha}} ) -\mathcal{F} \, u\,(\overline{u\, E} -  u\, \overline{E}).
$$
Hence, we multiply \eqref{eq:proof_identification_YM_step1_algebraic_identity_between_weak_limits_logarithmic_variables} by $s$ to obtain
$$
s\, (\overline{E} + s) \, (\overline{u  \,\partial_x s^{\alpha}} - u \,\partial_x s^{\alpha}) + \mathcal{F}\,s\,( \overline{u\, v\log v} - u\, \overline{v  \log v})  -\mathcal{F} \, u\,(\overline{u\, E} -  u\, \overline{E})  = 0.
$$
Applying Lemma \ref{lem:general_identity_to_replace_partialx_ps_with_just_u} with $f(u,v)=u$ we have
$$
s\, \left(\overline{u \, \partial_x s^{\alpha}} - u\, \partial_x s^{\alpha} \right)= \mathcal{F} \, (\overline{ u^2 } - u^2).
$$
Since $\mathcal{F}\neq 0$, we obtain \eqref{eq:proof_identification_YM_step2_algebraic_identity_between_weak_limits_logarithmic_variables} as desired.\\

\underline{Step 3. Decomposition of \eqref{eq:proof_identification_YM_step2_algebraic_identity_between_weak_limits_logarithmic_variables} for two nonnegative terms.} We define
\begin{equation}\label{eq:definition_term_X_identification_YM_ueps} 
X := (\overline{E} - E) \,(\overline{u^2} -  u^2),
\end{equation}
$$
Y := (E+ s) \,(\overline{u^2} -  u^2) + s\,( \overline{u\, v\log v} - u\, \overline{v  \log v})  - u\,(\overline{u\, E} -  u\, \overline{E}),
$$
where $E = u \log u + v\log v$. We claim that $X \geq 0$ and $Y \geq 0$.\\

\underline{Step 3.1. $X\geq 0$.} Let $\{\mu_{t,x}\}_{t,x}$ be the Young measure of the sequence $\{\ueps\}_{\eps\in(0,1)}$, i.e. whenever $f:\R\to \R$ is continuous and $\{f(\ueps)\}_{\eps\in(0,1)}$ is weakly compact in $L^1_{\loc}((0,T)\times\R)$, we have
\begin{equation}\label{eq:YM_representation_of_weak_limits_with_just_ueps}
f(\ueps) \weak \int_{[0,\infty)} f(\lambda) \diff \mu_{t,x} \mbox{ weakly in } L^1_{\loc}((0,T)\times\R),
\end{equation}
see \cite[Chapter~6]{MR1452107}. We can apply \eqref{eq:YM_representation_of_weak_limits_with_just_ueps} with $f(\lambda) = \lambda \log(\lambda)$ and $f(\lambda) = \lambda^2$ by \ref{est:conservation} and \ref{est:ulogu_vlogv}. Now, since $\mu_{t,x}$ is a probability measure and $[0,\infty) \mapsto \lambda\log \lambda$ is convex, by Jensen's inequality
\begin{equation}\label{eq:weak_limit_ulogu_controlled_from_below_by_ulogu}
\overline{ u\log u} \! = \! \int_{[0,\infty)} \! \lambda \log \lambda \diff \mu_{t,x}(\lambda) \! \geq \! \left(\int_{[0,\infty)} \! \lambda \diff \mu_{t,x}(\lambda)\right)  \log \left( \int_{[0,\infty)}\! \lambda \diff \mu_{t,x}(\lambda) \right) \!= \!u\log u.
\end{equation}
Similarly, we prove $\overline{u^2} \geq u^2$, $\overline{v\log v} \geq v\log v$ and we conclude $X \geq 0$. \\

\underline{Step 3.2. $Y\geq 0$.} $Y$ is the weak limit of the sequence $Y_{\eps} = Y^1_{\eps} + Y^2_{\eps}$, where 
$$
Y^1_{\eps}:= (E+ s) \,(\ueps^2 -  u^2), \qquad Y^2_{\eps}:= s\,( u_{\eps}\, v_{\eps} \log v_{\eps} - u\, v_{\eps}  \log v_{\eps})  - u\,({u_{\eps}\, E_{\eps}} -  u\, {E_{\eps}}).
$$
First, we can write $Y^2_{\eps}$ as 
\begin{align*}
Y^2_{\eps} &= s\,( u_{\eps}\, v_{\eps} \log v_{\eps} - u\, v_{\eps}  \log v_{\eps})  - u\,(u_{\eps}^2 \log \ueps + \ueps\, \veps \log \veps -  u\, \ueps \log \ueps - u\, \veps \log \veps)\\
&= - (\ueps-u)\, u\, \ueps\log \ueps + (\ueps -u)\,v\, \veps \log \veps =: Y^{2,1}_{\eps} + Y^{2,2}_{\eps}.
\end{align*}
For the term $Y^{2,1}_{\eps}$, we decompose $\log \ueps = \log u + \log \frac{\ueps}{u}$, $\ueps=(\ueps-u) + u$ and we use Lemma \ref{lem:upper_bound_on_the_x-y_log_xy} to obtain
\begin{align*}
Y^{2,1}_{\eps} &= - (\ueps-u)\, u\, \ueps\log \ueps = - (\ueps-u)\, u\, \ueps\log u - (\ueps-u)\, u\, \ueps\log \frac{\ueps}{u}\\
&\geq  - (\ueps-u)^2\,  u\,\log u - (\ueps-u)\, u^2 \log u - \frac{\ueps+u}{2} (\ueps-u)^2.
\end{align*}
Note that there is no issue with $\log\ueps$ or $\log u$ when $\ueps=0$ or $u=0$, respectively, as the term is multiplied by $u\,\ueps$. Similarly, for $Y_{\eps}^{2,2}$ we have
\begin{align*}
Y^{2,2}_{\eps} &= (\ueps -u)\,v\, \veps \log \veps = (\seps -s)\,v\, \veps \log \veps - (\veps-v)\,v\, \veps \log \veps\\
&= (\seps -s)\,v\, \veps \log \veps - (\veps-v)\,v\, \veps \log \frac{\veps}{v} - (\veps-v)\,v\, \veps \log v\\
&= (\seps -s)\,v\, \veps \log \veps - (\veps-v)\,v\, \veps \log \frac{\veps}{v} - (\veps-v)^2\,v \log v - (\veps-v)\,v^2 \log v\\
&\geq (\seps -s)\,v\, \veps \log \veps - \frac{\veps+v}{2}(\veps-v)^2 - (\veps-v)^2\,v \log v - (\veps-v)\,v^2 \log v.
\end{align*}
In the terms with $(\veps-v)^2$, we use the identity
$$
(\veps-v)^2 = (\ueps-u)^2 + (\seps-s)\,(\veps-v+u-\ueps),
$$
so we can write
\begin{align*}
Y^{2,2}_{\eps} \geq\, & (\seps -s)\,v\, \veps \log \veps 
- \frac{\veps+v}{2}(\ueps-u)^2 
- (\ueps-u)^2\,v \log v\\ 
&-\Big(\frac{\veps+v}{2} + v \log v\Big)\,(\seps-s)\,(\veps-v+u-\ueps)
- (\veps-v)\,v^2 \log v.
\end{align*}
We collect all terms in the estimates for $Y^{2,1}_{\eps}$ and $Y^{2,2}_{\eps}$ converging at least weakly to 0 
\begin{align*}
Z_{\eps}:= \,&  - (\ueps-u)\, u^2 \log u + (\seps -s)\,v\, \veps \log \veps\\
& -\Big(\frac{\veps+v}{2} + v \log v\Big)\,(\seps-s)\,(\veps-v+u-\ueps)
- (\veps-v)\,v^2 \log v,
\end{align*}
so that we can finally write
\begin{align*}
Y^2_{\eps} &\geq - (\ueps-u)^2\,  u\,\log u  - \frac{\ueps+u}{2} (\ueps-u)^2 
- \frac{\veps+v}{2}(\ueps-u)^2 
- (\ueps-u)^2\,v \log v + Z_{\eps}\\
& = - (\ueps-u)^2\, E - \frac{\seps+s}{2} \, (\ueps-u)^2 + Z_{\eps}. 
\end{align*}
Combining with $Y^1_{\eps} = (E+s)\,(\ueps^2-u^2) = (E+s)\,(\ueps-u)^2 + 2\,u\,(E+s)\,(\ueps-u)$, we obtain
$$
Y_{\eps} = Y^1_{\eps} + Y^2_{\eps} \geq Z_{\eps} + \frac{s-\seps}{2}(\ueps-u)^2 + 2\,u\,(E+s)\,(\ueps-u).
$$
We define $\widetilde{Z_{\eps}} := Z_{\eps} + \frac{s-\seps}{2}(\ueps-u)^2 + 2\,u\,(E+s)\,(\ueps-u)$ and we observe that $\widetilde{Z_{\eps}}  \weak 0$ weakly in $L^1_{\loc}((0,T)\times\R)$ (because it contains either terms composed of $(\seps-s)$ which converge strongly by \ref{conv:sum_strong} or terms composed of $(\ueps-u)$, $(\veps-v)$ converging weakly by \ref{conv:u_v_weak}, multiplied by some functions which are bounded independently of $\eps$). Let $\varphi \in C_c^{\infty}((0,T)\times\R)$, $\varphi \geq 0$. By $Y_{\eps} - \widetilde{Z_{\eps}} \geq 0$, we have 
$$
\int_0^T \int_{\R} (Y_{\eps} - \widetilde{Z_{\eps}} )\, \varphi \diff x \diff t \geq 0 \implies \int_0^T \int_{\R} Y\, \varphi \diff x \diff t \geq 0.
$$ 
Since $\varphi \geq 0$ is arbitrary, we deduce ${Y} \geq 0$ as desired.\\

\underline{Step 4. Conclusion.} In Step 3 we decomposed \eqref{eq:proof_identification_YM_step2_algebraic_identity_between_weak_limits_logarithmic_variables} for two nonnegative terms $X$, $Y$. From \eqref{eq:proof_identification_YM_step2_algebraic_identity_between_weak_limits_logarithmic_variables} we have $X+Y = 0$ so we deduce $X=Y=0$. From \eqref{eq:definition_term_X_identification_YM_ueps}, this means $\overline{E}=E$ or $\overline{u^2} = u^2$. We first show that $\overline{E}=E$ implies $\overline{u^2} = u^2$ so that we can assume simply $\overline{u^2} = u^2$. Indeed, by \eqref{eq:weak_limit_ulogu_controlled_from_below_by_ulogu}, we know that $d_u:= \overline{u\log u} - u\log u \geq 0$, $d_{v} := \overline{v \log v} - v\log v \geq0$. Since $\overline{E}=E$, we obtain $d_u = d_ v = 0$, that is $\overline{u\log u} = u\log u$ and $\overline{v \log v} = v\log v$.\\

Let $\{\mu_{t,x}\}$ be the Young measure of the sequence $\{\ueps\}_{\eps\in(0,1)}$ as in Step 3.1. We claim that $\overline{u\log u} = u\log u$ implies that $\mu_{t,x} = \delta_{u(t,x)}$. Indeed, if $u(t,x) = 0$, we have $0 = u(t,x) = \int_{[0,\infty)} \lambda \diff \mu_{t,x}(\lambda)$ so that $\mu_{t,x} = \delta_{0} = \delta_{u(t,x)}$ as claimed. If $u(t,x)>0$, we have for each $\lambda \in [0,\infty)$ by the Taylor's expansion with an integral remainder of $f(\lambda) = \log \lambda$ with $f'(\lambda) =\log u + 1$, $f''(\lambda)=\frac{1}{u}$ 
$$
\lambda \, \log \lambda = u\, \log u + (\lambda - u) \,( \log u + 1) + \int_{u}^{\lambda} \frac{\lambda - \tau}{\tau} \diff \tau.
$$
Integrating this identity with respect to $\mu_{t,x}(\lambda)$ and using $\overline{u\log u} = u\log u$ we obtain
$$
\int_{[0,\infty)} \int_{u}^{\lambda} \frac{\lambda - \tau}{\tau} \diff \tau \diff \mu_{t,x}(\lambda) = 0.
$$
The integrand is nonnegative and if there existed any $\lambda \in \supp \mu_{t,x}$ such that $\lambda \neq u(t,x)$, we would gain a positive contribution to the integral. It follows that $\mu_{t,x} = \delta_{u(t,x)}$ as claimed. Since $\mu_{t,x} = \delta_{u(t,x)}$, we deduce $\overline{u^2} = \int_{[0,\infty)} \lambda^2 \diff \mu_{t,x}(\lambda) = u^2$ as desired.\\

Hence, we have $\overline{u^2}=u^2$. At any point $(t,x)$ such that $s(t,x)>0$, we have $s^2\, \overline{\, \frac{u^2}{s^2}\,}  = \overline{u^2}$ cf. \ref{conv:fractions} so that
$$
\left(\frac{\ueps}{\seps}-\frac{u}{s}\right)^2 \weak \overline{\,\frac{u^2}{s^2}\,} -  \frac{u^2}{s^2} = \frac{1}{s^2}\left(\overline{u^2} - u^2 \right).
$$
Since $\overline{u^2}=u^2$, $\left(\frac{\ueps}{\seps}-\frac{u}{s}\right)^2 \weak 0$ and we finally obtain $\int_{\R^2} \lambda_1^2 \diff \pi_{t,x}(\lambda_1,\lambda_2) =0$ as desired.
\end{proof}

\subsection{Proof of Theorem \ref{thm:main}}\label{sect:main_proof_after_two_steps}
Let $(u,v)$ be the limit in Proposition \ref{prop:basic_convergences}. In order to prove that $(u,v)$ is a solution to \eqref{eq:general_cross_diffusion_intro_short_velocity}, according to Proposition \ref{prop:main_thm_by_Young_measures}, we only need to check that for a.e. $(t,x)$ $\int_{\R^2} \lambda_1\, \lambda_2 \diff \pi_{t,x}(\lambda_1,\lambda_2) = 0$. By Proposition \ref{prop:support_YM_is_line} this is the case for all points $(t,x)$ except for those such that $\pi_{t,x}$ is supported on the line $\lambda_2 = \mathcal{F}\,\lambda_1$. However, Proposition \ref{prop:the_case_of_support_on_the_line_identification_of_YM} shows that if this is the case than $\pi_{t,x}= \delta_{(0,0)}$ so we again obtain $\int_{\R^2} \lambda_1\, \lambda_2 \diff \pi_{t,x}(\lambda_1,\lambda_2) = 0$. \\

We pass to the proof of modes of convergences. Clearly, the weak/weak$^*$ convergences of $\{\ueps\}$, $\{\veps\}$ and strong convergence of $\{\seps\}$ follow directly from Proposition \ref{prop:basic_convergences}. We now discuss the strong convergence of $\{\partial_x \seps^{\alpha}\}$. This argument is nowadays standard and has been used e.g. in \cite{MR4712820, MR4179253, MR4880211}. Multiplying the PDE for $\seps$ by $\seps^{\alpha}$ we have for each $\eps>0$ and $t\in[0,T]$
\begin{equation}\label{eq:energy_identity_before_sending_eps_to_0}
\begin{split}
& \int_{\R} \frac{ \seps^{\alpha+1}(t,x)\! -\! (s^0)^{\alpha+1}(x) }{\alpha+1} \diff x\! +\! \frac{1}{\alpha} \int_0^t \int_{\R} |\partial_x \seps^{\alpha}|^2 \diff x \diff \tau\! +\! \eps \int_0^t \int_{\R} \partial_x \seps \, \partial_x \seps^{\alpha} \diff x \diff \tau\\&= -\int_0^t \int_{\R} \seps\, \frac{\ueps}{\seps}\, \partial_x \seps^{\alpha}\, \partial_x \mathcal{V}^1[\ueps, \veps] \diff x \diff \tau -\int_0^t \int_{\R} \seps\, \frac{\veps}{\seps}\, \partial_x \seps^{\alpha} \, \partial_x \mathcal{V}^2[\ueps, \veps] \diff x \diff \tau.
\end{split}
\end{equation}
We also have this identity for $s$ because we can use $s^{\alpha}$ as a test function in Definition~\ref{def:weak_sol} as $\partial_x s^{\alpha} \in L^2((0,T)\times\R)$. More precisely, we extend $s(t,x) = s^0(x)$ for $t<0$, we test the PDE with $((s\ast \eta_{\delta}\ast \omega_{\kappa} \, \psi)^{\alpha} \,\psi) \ast \eta_{\delta}\ast \omega_{\kappa}$, where $\psi$ is smooth and compactly supported, $\{\eta_{\delta}\}_{\delta \in (0,1)}$ is a mollification in time while $\{\omega_{\kappa}\}_{\kappa \in (0,1)}$ is a mollification in space. After summing up and obtaining an identity for $\frac{1}{\alpha+1} \partial_t \int_{\R} (s\ast \eta_{\delta}\ast \omega_{\kappa} \, \psi)^{\alpha+1} \diff x$, we integrate in time and we can pass to the limit $\delta, \kappa \to 0$ and $\psi \to 1$ on $\R$ using the estimates on $\partial_x s^{\alpha}$ and $s |x|$. We obtain
\begin{equation}\label{eq:energy_identity_after_sending_eps_to_0}
\begin{split}
&\int_{\R} \frac{ s^{\alpha+1}(t,x) - (s^0)^{\alpha+1}(x)}{\alpha+1}  \diff x +\frac{1}{\alpha} \int_0^t \int_{\R} |\partial_x s^{\alpha}|^2 \diff x \diff \tau \\&= -\int_0^t \int_{\R} s\, \frac{u}{s}\, \partial_x s^{\alpha}\, \partial_x \mathcal{V}^1[u, v] \diff x \diff \tau -\int_0^t \int_{\R} s\, \frac{v}{s}\, \partial_x s^{\alpha} \, \partial_x \mathcal{V}^2[u, v] \diff x \diff \tau.
\end{split}
\end{equation}
It is easy to see that
\begin{equation}\label{eq:convergence_energy_advection_term_with_u_claim}
\int_0^t \int_{\R} \seps \, \frac{\ueps}{\seps}\, \partial_x \seps^{\alpha}\, \partial_x \mathcal{V}^1[\ueps, \veps] \diff x \diff \tau \to \int_0^t \int_{\R} s\, \frac{u}{s}\, \partial_x s^{\alpha}\, \partial_x \mathcal{V}^1[u,v] \diff x \diff \tau \mbox{ as } \eps \to 0.
\end{equation}
Indeed, we split $\R$ for $[-R,R]$ and $\R \setminus [-R,R]$. For all $R>0$, we have covergence of the integral over $[-R,R]$  thanks to \ref{conv:sum_strong}, \ref{conv:velocity_fields} and because $\overline{\frac{\ueps}{\seps}\, \partial_x \seps^{\alpha}} = \frac{u}{s}\, \partial_x s^{\alpha}$. On $\R \setminus [-R, R]$ we use that the integral is small
$$
\left|\int_0^t \int_{\R \setminus [-R,R]} \ueps\, \partial_x \seps^{\alpha}\, \partial_x \mathcal{V}^1[\ueps, \veps] \diff x \diff \tau \right| \leq \|\partial_x \mathcal{V}^1\|_{L^{\infty}}\, \|\partial_x \seps^{\alpha}\|_{L^2((0,T)\times\R)}\, \frac{\|\ueps\, |x|^{\frac{1}{2}}\|_{L^2((0,T)\times\R)}}{R^{\frac{1}{2}}}.
$$
This estimate is uniform in $\eps$ by Assumption \ref{ass:velocity_kernels_main} as well as \ref{est:moment}, \ref{est:crucial_estimate_gradient_power_alpha} and \ref{est:conservation}. Similarly,
$$
\left|\int_0^t \int_{\R \setminus [-R,R]} u\, \partial_x s^{\alpha}\, \partial_x \mathcal{V}^1[u,v] \diff x \diff \tau \right| \leq \frac{C}{R^{\frac{1}{2}}}
$$
for some constant $C$. This proves the convergence \eqref{eq:convergence_energy_advection_term_with_u_claim} by taking the limit $R\to\infty$. Analogously, 
\begin{equation}\label{eq:convergence_energy_advection_term_with_v_claim}
\int_0^t \int_{\R} \seps \, \frac{\veps}{\seps}\, \partial_x \seps^{\alpha}\, \partial_x \mathcal{V}^1[\ueps, \veps] \diff x \diff \tau \to \int_0^t \int_{\R} s\, \frac{v}{s}\, \partial_x s^{\alpha}\, \partial_x \mathcal{V}^1[u,v] \diff x \diff \tau \mbox{ as } \eps \to 0.
\end{equation}
We now observe that there is a subsequence (not relabelled) such that for a.e. $t \in [0,T]$
\begin{equation}\label{eq:convergence_energies_ae_time}
\frac{1}{\alpha+1} \int_{\R} \left( \seps^{\alpha+1}(t,x) - (s^0)^{\alpha+1}(x) \right) \diff x \to \frac{1}{\alpha+1} \int_{\R} \left( s^{\alpha+1}(t,x) - (s^0)^{\alpha+1}(x) \right) \diff x.
\end{equation}
Indeed, by the strong convergence \ref{conv:sum_strong} we have $\int_0^T \int_\R | \seps^{\alpha+1} - s^{\alpha+1}| \diff x \diff t \to 0$ so that $\int_{\R}  \seps^{\alpha+1}(t,x) \diff x \to \int_{\R} s^{\alpha+1}(t,x) \diff x$ in $L^1(0,T)$ and the existence of a subsequence follows. Therefore, by comparing \eqref{eq:energy_identity_before_sending_eps_to_0} and \eqref{eq:energy_identity_after_sending_eps_to_0}, and exploiting \eqref{eq:convergence_energy_advection_term_with_u_claim}, \eqref{eq:convergence_energy_advection_term_with_v_claim}, \eqref{eq:convergence_energies_ae_time} and that $\eps\, \int_0^t \int_{\R} \partial_x \seps \, \partial_x \seps^{\alpha} \diff x \diff \tau \geq 0$, we deduce that for a.e. $t \in [0,T]$
$$
\limsup_{\eps \to 0}  \int_0^t \int_{\R} |\partial_x \seps^{\alpha}|^2 \diff x \diff \tau \leq  \int_0^t \int_{\R} |\partial_x s^{\alpha}|^2 \diff x \diff \tau.
$$
It follows that $\limsup_{\eps \to 0} \int_0^t \int_{\R} |\partial_x \seps^{\alpha} -\partial_x s^{\alpha}|^2 \diff x \diff \tau \leq 0$ for a.e. $t \in (0,T)$. We conclude since $T>0$ is arbitrary.\\

We now prove the strong convergence of $\{\ueps\}$, $\{\veps\}$ in $L^p(\mathcal{S})$ and a.e. on $\mathcal{S}$, where 
$$
\mathcal{S}:= [0,T]~\times~\R \setminus \{(t,x): s(t,x) > 0 \mbox{ and } 
\partial_x\mathcal{V}^1[u,v] = \partial_x\mathcal{V}^2[u,v]\}.
$$
Let $\{\mu_{t,x}\}_{t,x}$ be the Young measure of $\{\ueps\}$, i.e. whenever $f(t,x,\lambda):[0,T]\times\R \times \R\to \R$ is a Carathéodory function (measurable in $(t,x)$ and continuous in $\lambda$) and $\{f(t,x,\ueps)\}$ is weakly compact in $L^1_{\loc}((0,T)\times\R)$,
$
f(t,x,\ueps) \weak \int_{[0,\infty)} f(t,x,\lambda) \diff \mu_{t,x} \mbox{ weakly in } L^1_{\loc}((0,T)\times\R),
$
see \cite[Chapter~6]{MR1452107}. We will prove that for $(t,x) \in \mathcal{S}$ we have $\mu_{t,x} = \delta_{u(t,x)}$, which implies the claim since $|\ueps - u|^p \weak \int_{[0,\infty)} |\lambda - u(t,x)|^p \diff \mu_{t,x}(\lambda) = 0$ for $p\geq1$. We consider two cases.
\begin{itemize}
\item If $s(t,x)=0$, we also have $u(t,x) = 0$ and
$
|\ueps(t,x)| \leq |\seps(t,x)| \to 0.
$
But $|\ueps(t,x)| \weak \int_{[0,\infty)} |\lambda| \diff \mu_{t,x}(\lambda)$ so that $\mu_{t,x} = \delta_{0} = \delta_{u(t,x)}$. 
\item When $s(t,x)>0$ and $\partial_x\mathcal{V}^1[u,v] \neq \partial_x \mathcal{V}^2[u,v]$, by Propositions \ref{prop:support_YM_is_line} and \ref{prop:the_case_of_support_on_the_line_identification_of_YM} we have $\pi_{t,x}((\R \setminus \{0\}) \times \R) = 0$, so that $\lambda_1 = 0$ for a.e.-$\pi_{t,x}$ $(\lambda_1, \lambda_2)$. In particular, by \ref{conv:fractions} and $\overline{\, \frac{u^2}{s^2} \,} - \frac{u^2}{s^2} = \int_{\R^2} \lambda_1^2 \diff \pi_{t,x}(\lambda_1, \lambda_2)$,
$$
 \overline{u^2} = s^2 \, \overline{\, \frac{u^2}{s^2} \,} =s^2 \left( \int_{\R^2} \lambda_1^2 \diff \pi_{t,x}(\lambda_1, \lambda_2) + \frac{u^2}{s^2} \right) = u^2,
$$
so $\int_{[0,\infty)} |\lambda - u(t,x)|^2 \diff \mu_{t,x}(\lambda) = 0$ which implies $\mu_{t,x} = \delta_{u(t,x)}$. 
\end{itemize}
The convergence of $\{\veps\}$ follows from $\veps = \seps-\ueps$ and the strong convergence of $\{\seps\}$.

\begin{rem}
One may ask whether the strong convergence of $\{\ueps\}$ and $\{\veps\}$ can be extended to the whole domain $[0,T]\times\R$. The compensated compactness approach fails at points where $\partial_x\mathcal{V}^1[u,v] = \partial_x\mathcal{V}^2[u,v]$ (which implies $\mathcal{F}=0$ with $\mathcal{F}$ defined in \eqref{eq:definition_quantity_F_necessary_later}), since \eqref{eq:proof_identification_YM_step1_algebraic_identity_between_weak_limits_logarithmic_variables} reduces to the trivial identity $0=0$ once $\{\partial_x \seps^{\alpha}\}$ is known to converge strongly. Similarly, in the alternative argument of Section~\ref{subsect:alternative_proof_alpha_at_most_3}, \eqref{eq:2nd_application_comp_compactn_identity_step_1} degenerates to the same trivial identity. A natural approach would be to exploit the energy $(u,v)\mapsto \int_{\R} (u \log u + v \log v) \diff x$, in analogy with the use of the dissipation of $s \mapsto \int_{\R} s \log s\diff x$ to prove the strong convergence of $\{\partial_x \seps^{\alpha}\}$. The main difficulty is that, although all the terms in the PDE for the limit $u$ have been identified, it remains unclear how to justify rigorously testing this PDE by $\log u$ (and similarly testing the equation for $v$ by $\log v$), since $u$ and $v$ possess only $L^p((0,T)\times\R)$ regularity. This argument could also show strong convergence of $\{\partial_x \seps^{\frac{\alpha}{2}}\}$.
\end{rem}
\subsection{Proof of Proposition \ref{prop:the_case_of_support_on_the_line_identification_of_YM} for $\alpha\in(0,2]$.}\label{subsect:alternative_proof_alpha_at_most_3} Here, we provide an alternative proof of Proposition \ref{prop:the_case_of_support_on_the_line_identification_of_YM} (and so, Theorem \ref{thm:main} as well) which requires less algebraic manipulations since it uses the conservation law for $\frac{\ueps^2}{\seps}$ from Proposition \ref{prop:new_cons_law} instead of the one for $\ueps\log\ueps+\veps\log\veps$ from Proposition \ref{prop:new_cons_law_log_general_alpha} and the nonlinearities involved are easier to handle. However, this argument works only for $\alpha \leq 2$ as explained in Section \ref{subsect:cons_law_u2_over_s}. \\ 

\underline{Step 1: Identity between weak limits.} We will prove that
\begin{equation}\label{eq:2nd_application_comp_compactn_identity_step_1}
\begin{split}
 \frac{1}{3}\, \overline{\,\frac{u^4}{s^2}\,} \, \partial_x (\mathcal{V}^1[u,v] - &\mathcal{V}^2[u,v]) \\ &= \frac{1}{\alpha}\, \overline{\frac{u}{s} \partial_x s^{\alpha}} \, \overline{\, \frac{u^2}{s}\,}- \frac{1}{\alpha}\,u\, \overline{\frac{u^2}{s^2}\, \partial_x s^{\alpha}} + \frac{1}{3}\, u\, \overline{\,\frac{u^3}{s^2}\,} \partial_x (\mathcal{V}^1[u,v] - \mathcal{V}^2[u,v]),
\end{split}
\end{equation}
where $\overline{\,\frac{u^4}{s^2}\,}$, $\overline{\frac{u}{s} \partial_x s^{\alpha}}$, $\overline{\, \frac{u^2}{s}\,}$, $ \overline{\frac{u^2}{s^2}\, \partial_x s^{\alpha}}$, $\overline{\,\frac{u^3}{s^2}\,}$ are the weak limits defined in \ref{conv:fractions} and \ref{conv_product_u_derivative_s_alpha_divided_by_s}. As in Section \ref{subsect:the_case_that_YM_supp_on_a_nonzero_line}, we apply Theorem~\ref{thm:div_curl_lemma_version_CR} to the vector fields
\begin{align*}
{\bf a_\eps} &= \left(\ueps, -\frac{1}{\alpha}\, \frac{\ueps}{\seps} \partial_x \seps^{\alpha} -\ueps \, \partial_x \mathcal{V}^1[\ueps,\veps] \right),\\
{\bf b_\eps} &= \left(\frac{1}{\alpha}\,\frac{\ueps^2}{\seps^2}\, \partial_x \seps^{\alpha} + \frac{\ueps^2}{\seps}\,  \partial_x \mathcal{V}^1[\ueps, \veps] - \frac{\ueps^3}{3 \, \seps^2}\, \partial_x (\mathcal{V}^1[\ueps, \veps] - \mathcal{V}^2[\ueps, \veps]),\, \frac{\ueps^2}{\seps} \right). 
\end{align*}
Let us explain why the assumptions of Theorem \ref{thm:div_curl_lemma_version_CR} are satisfied. Clearly, both $\{\bf a_\eps\}$ and $\{\bf b_\eps\}$ converge weakly in $L^2_{\loc}((0,T)\times\R)$ to some ${\bf a}$ and ${\bf b}$ by
\ref{conv:sum_strong}, \ref{conv:u_v_weak}, \ref{conv:velocity_fields}, \ref{conv:fractions}, and \ref{conv_product_u_derivative_s_alpha_divided_by_s} in Proposition~\ref{prop:basic_convergences}. Moreover, $\DIV {\bf a_\eps} =\eps\, \partial^2_x \ueps$ converges strongly to 0 in $H^{-1}((0,T)\times\R)$ because, by \ref{est:energy_derivatives_u_v_eps_signular}, $\eps\,\partial_x\ueps, \eps\,\partial_x \veps\to 0$ strongly in $L^2((0,T)\times\R)$. Regarding $\CURL {\bf b_\eps}$, we have $\CURL {\bf b_\eps} = - g_{\eps} - h_{\eps}$, where $\{g_{\eps}\}$ and $\{h_{\eps}\}$ are defined in Proposition~\ref{prop:new_cons_law}. We know that $\{g_{\eps}\}$ is bounded in $L^1_{\loc}([0,T]\times\R)$, which compactly embeds into $W^{-1,p}_{\loc}([0,T]\times\R)$ for $p<2$ \cite[Theorem 6, p.~7]{MR1034481}, while $\{h_{\eps}\}$ is strongly compact in $H^{-1}((0,T)\times\R)$. Combining these two facts, $\{\CURL {\bf b_\eps}\}$ is strongly compact in $(W^{1,\infty}_0(\Omega))^*$ for each bounded $\Omega \subset [0,T]\times\R$. Finally, the product $\{{\bf a_\eps} \cdot {\bf b_\eps}\}$ is uniformly integrable in $L^1(\Omega)$ for $\Omega$ as above. Indeed, writing ${\bf a_\eps} = (a^1_{\eps}, a^2_{\eps})$ and ${\bf b_\eps} = (b^1_{\eps}, b^2_{\eps})$, $\{a^1_{\eps}\}$ and $\{b^2_{\eps}\}$ are bounded in $L^{\infty}((0,T)\times\R)$ by \ref{est:conservation}, while $\{a^2_{\eps}\}$ and $\{b^1_{\eps}\}$ are bounded at least in $L^{2}_{\loc}((0,T)\times\R)$ by \ref{est:crucial_estimate_gradient_power_alpha}, \ref{est:conservation}, and Assumption~\ref{ass:velocity_kernels_main}. It follows that $\{{\bf a_\eps} \cdot {\bf b_\eps}\}$ is bounded in $L^2(\Omega)$, which implies uniform integrability in $L^1(\Omega)$.\\

We compute
\begin{align*}
{\bf a_\eps}\cdot {\bf b_\eps} =&\, \ueps \left(\frac{1}{\alpha}\,\frac{\ueps^2}{\seps^2}\, \partial_x \seps^{\alpha} + \frac{\ueps^2}{\seps}\,  \partial_x \mathcal{V}^1[\ueps, \veps] - \frac{\ueps^3}{3 \, \seps^2}\, \partial_x (\mathcal{V}^1[\ueps, \veps] - \mathcal{V}^2[\ueps, \veps])\right)\\ 
&-\left(\frac{1}{\alpha}\, \frac{\ueps}{\seps} \partial_x \seps^{\alpha} +\ueps \, \partial_x \mathcal{V}^1[\ueps,\veps] \right) \frac{\ueps^2}{\seps}\\
=&\,-\frac{\ueps^4}{3 \, \seps^2}\, \partial_x (\mathcal{V}^1[\ueps, \veps] - \mathcal{V}^2[\ueps, \veps]).
\end{align*}
By \ref{conv:velocity_fields} and \ref{conv:fractions}
$$
{\bf a_\eps}\cdot {\bf b_\eps} \weak -\frac{1}{3}\, \overline{\,\frac{u^4}{s^2}\,} \, \partial_x (\mathcal{V}^1[u,v] - \mathcal{V}^2[u,v]).
$$
Moreover, by \ref{conv:u_v_weak}, \ref{conv:velocity_fields}, \ref{conv:fractions}, \ref{conv_product_u_derivative_s_alpha_divided_by_s}, we can identify also products of the weak limits ${\bf a} = (a^1, a^2)$ and ${\bf b} = (b^1, b^2)$ as
$$
a^1\, b^1 = u\, \left(\frac{1}{\alpha}\,\overline{\frac{u^2}{s^2}\, \partial_x s^{\alpha}} + \overline{\,\frac{u^2}{s}\,}  \partial_x \mathcal{V}^1[u,v] - \frac{1}{3} \overline{\,\frac{u^3}{s^2}\,} \partial_x (\mathcal{V}^1[u,v] - \mathcal{V}^2[u,v]) \right),
$$
$$
a^2 \, b^2 = - \left(\frac{1}{\alpha}\, \overline{\frac{u}{s} \partial_x s^{\alpha}} +u \, \partial_x \mathcal{V}^1[u,v]\right)\, \overline{\, \frac{u^2}{s}\,}, 
$$
so we obtain 
$$
{\bf a}\cdot {\bf b} = 
-\frac{1}{\alpha}\, \overline{\frac{u}{s} \partial_x s^{\alpha}} \, \overline{\, \frac{u^2}{s}\,}+ \frac{1}{\alpha}\,u\, \overline{\frac{u^2}{s^2}\, \partial_x s^{\alpha}} - \frac{1}{3}\, u\, \overline{\,\frac{u^3}{s^2}\,} \partial_x (\mathcal{V}^1[u,v] - \mathcal{V}^2[u,v]).
$$
Therefore, \eqref{eq:2nd_application_comp_compactn_identity_step_1} follows directly from Theorem \ref{thm:div_curl_lemma_version_CR}.\\

\underline{Step 2. The identity away from vacuum.} We prove that for $(t,x)$ such that $s(t,x)>0$ we have
\begin{equation}\label{eq:2nd_application_comp_compactn_identity_step_2}
\frac{u}{s}\, \overline{\frac{u^2}{s^2}\, \partial_x s^{\alpha}} - \overline{\frac{u}{s} \partial_x s^{\alpha}} \, \overline{\, \frac{u^2}{s^2}\,}=  \frac{\mathcal{F}}{3}\, \left( \overline{\,\frac{u^4}{s^4}\,} -  \frac{u}{s}\, \overline{\,\frac{u^3}{s^3}\,} \right).
\end{equation}
Indeed, by \ref{conv:fractions} we have $s^2 \overline{\,\frac{u^4}{s^2}\,} =  \overline{\,\frac{u^4}{s^4}\,}$, $\overline{\, \frac{u^2}{s}\,} = s\,\overline{\, \frac{u^2}{s^2}\,}$, $ \overline{\,\frac{u^3}{s^2}\,} = s\, \overline{\,\frac{u^3}{s^3}\,}$ so that we transform \eqref{eq:2nd_application_comp_compactn_identity_step_1} into
\begin{multline*}
 \frac{s^2}{3}\, \overline{\,\frac{u^4}{s^4}\,} \, \partial_x (\mathcal{V}^1[u,v] - \mathcal{V}^2[u,v]) \\ = \frac{s}{\alpha}\, \overline{\frac{u}{s} \partial_x s^{\alpha}} \, \overline{\, \frac{u^2}{s^2}\,}- \frac{s}{\alpha}\,\frac{u}{s}\, \overline{\frac{u^2}{s^2}\, \partial_x s^{\alpha}} + \frac{s^2}{3}\, \frac{u}{s}\, \overline{\,\frac{u^3}{s^3}\,} \partial_x (\mathcal{V}^1[u,v] - \mathcal{V}^2[u,v]).
\end{multline*}
By using formula for $\mathcal{F}$ in \eqref{eq:definition_quantity_F_necessary_later}
$$
-\frac{s\,\mathcal{F}}{3\,\alpha}\, \overline{\,\frac{u^4}{s^4}\,} = \frac{s}{\alpha}\, \overline{\frac{u}{s} \partial_x s^{\alpha}} \, \overline{\, \frac{u^2}{s^2}\,}- \frac{s}{\alpha}\,\frac{u}{s}\, \overline{\frac{u^2}{s^2}\, \partial_x s^{\alpha}} - \frac{s\,\mathcal{F}}{3\,\alpha}\, \frac{u}{s}\, \overline{\,\frac{u^3}{s^3}\,}.
$$
Dividing by $\frac{s}{\alpha}$ and rearranging we obtain \eqref{eq:2nd_application_comp_compactn_identity_step_2}.\\

\underline{Step 3. Representation for the weak limits in terms of $\pi_{t,x}$.} We claim
\begin{equation}\label{eq:representation_weak_limit_u4_s4_step3}
 \overline{\,\frac{u^4}{s^4}\,} -  \frac{u}{s}\, \overline{\,\frac{u^3}{s^3}\,}  =
 \int_{\R^2} \left(\lambda_1^4 + 3\, \frac{u}{s}\, \lambda_1^3 + 3\, \frac{u^2}{s^2} \, \lambda_1^2\right) \diff \pi_{t,x}(\lambda_1,\lambda_2),
\end{equation}
\begin{equation}\label{eq:2nd_application_comp_compactness_limit_identification_step_3_mixed_fractions_gradients}
\begin{split}
\frac{u}{s}\, \overline{\frac{u^2}{s^2}\, \partial_x s^{\alpha}} - \overline{\frac{u}{s} \, \partial_x s^{\alpha}} \, \overline{\, \frac{u^2}{s^2}\,}
=& \int_{\R^2}\left( \frac{\mathcal{F}\,u}{s}\, \lambda_1^3   + \frac{\mathcal{F}u^2}{s^2}\, \lambda_1^2 \right) \diff \pi_{t,x}(\lambda_1,\lambda_2)\\
&- \mathcal{F}\, \left(\int_{\R^2} \lambda_1^2 \diff \pi_{t,x}(\lambda_1, \lambda_2)\right)^2.
\end{split}
\end{equation}
Let $q_{\eps} = \frac{\ueps}{\seps}$, $q=\frac{u}{s}$, $r_{\eps} = \partial_x \seps^{\alpha}$, $r=\partial_x s^{\alpha}$. We have
$$
q_{\eps}^3 = ((q_{\eps}-q) + q)^3 = (q_{\eps}-q)^3 + 3\,(q_{\eps}-q)^2\, q + 3\,(q_{\eps}-q)\, q^2 + q^3,
$$
so we obtain
$$
q_{\eps}^4 - q\, q_{\eps}^3 = q_{\eps}^3\, (q_{\eps}-q) = (q_{\eps}-q)^4 + 3\,(q_{\eps}-q)^3\, q + 3\,(q_{\eps}-q)^2\, q^2 + q^3\,(q_{\eps}-q).
$$
All the terms above are converging weakly$^*$ by \ref{conv:fractions} so a direct application of the representation formula \eqref{eq:representation_weak_limits_general_sequence_caratheodory_integrand} yields \eqref{eq:representation_weak_limit_u4_s4_step3}. Similarly,
\begin{align*}
q\,q_{\eps}^2\,r_{\eps} =&\,q\, (q_{\eps}-q + q)^2\, r_{\eps} = q\, (q_{\eps}-q)^2 \, r_{\eps} + 2\, q^2\, (q_{\eps}-q)\, r_{\eps} + q^3\, r_{\eps}
\\=&\, q\, (q_{\eps}-q)^2 \, (r_{\eps}-r) + q\, (q_{\eps}-q)^2\,r\\
&+ 2\, q^2\, (q_{\eps}-q) \, (r_{\eps}-r) +  2\, q^2\, (q_{\eps}-q) \,r + q^3\, (r_{\eps}-r) + q^3\,r,
\end{align*}
where all the terms are converging weakly at least in $L^2(0,T; L^2_{\loc}(\R))$ by \ref{conv:gradient_sum}, \ref{conv:fractions} and \ref{conv_product_u_derivative_s_alpha_divided_by_s}. Hence, by the representation formula \eqref{eq:representation_weak_limits_general_sequence_caratheodory_integrand},
$$
\frac{u}{s}\, \overline{\frac{u^2}{s^2}\, \partial_x s^{\alpha}} = \int_{\R^2}\left( \frac{u}{s}\, \lambda_1^2 \, \lambda_2 + \frac{u}{s}\,\partial_x s^{\alpha} \,  \lambda_1^2 +2 \, \frac{u^2}{s^2}\, \lambda_1\, \lambda_2 \right) \diff \pi_{t,x}(\lambda_1,\lambda_2) + \frac{u^3}{s^3}\, \partial_x s^{\alpha}.
$$
By using the assumption that $\lambda_2 = \mathcal{F}\,\lambda_1$ we simplify
\begin{equation}\label{eq:identification_of_the_limit_step_3_A_fraction_squared_times_gradient}
\frac{u}{s}\, \overline{\frac{u^2}{s^2}\, \partial_x s^{\alpha}} = \int_{\R^2}\left( \frac{\mathcal{F}\,u}{s}\, \lambda_1^3  + \frac{u}{s}\,\partial_x s^{\alpha} \,  \lambda_1^2 +2 \, \frac{\mathcal{F}u^2}{s^2}\, \lambda_1^2 \right) \diff \pi_{t,x}(\lambda_1,\lambda_2) + \frac{u^3}{s^3}\, \partial_x s^{\alpha}.
\end{equation}
Next, we consider the second term on the (LHS) in \eqref{eq:2nd_application_comp_compactness_limit_identification_step_3_mixed_fractions_gradients}:
\begin{align*}
\overline{\frac{u}{s} \, \partial_x s^{\alpha}} \, \overline{\, \frac{u^2}{s^2}\,} =& \left(\overline{\frac{u}{s}\, \partial_x s^{\alpha}} -\frac{u}{s}\, \partial_x s^{\alpha}\right) \, \overline{\, \frac{u^2}{s^2}\,} + \frac{u}{s}\, \partial_x s^{\alpha} \, \overline{\, \frac{u^2}{s^2}\,}\\
=& \left(\overline{\frac{u}{s}\, \partial_x s^{\alpha}} -\frac{u}{s}\, \partial_x s^{\alpha}\right) \,\left( \overline{\, \frac{u^2}{s^2}\,} - \frac{u^2}{s^2} \right)+  \left(\overline{\frac{u}{s}\, \partial_x s^{\alpha}} -\frac{u}{s}\, \partial_x s^{\alpha}\right)\, \frac{u^2}{s^2} \\
&+ \frac{u}{s}\, \partial_x s^{\alpha} \,\left( \overline{\, \frac{u^2}{s^2}\,} - \frac{u^2}{s^2}\right) + \frac{u^3}{s^3}\, \partial_x s^{\alpha}.
\end{align*}
Using explicit formulas \eqref{eq:equivalence_convergence_in_the_weak_formulation_double_integral_wrt_pi_is_zero}, \eqref{eq:difference_weak_limit_fraction_squared_in_terms_of_YM} and the assumption that $\lambda_2 = \mathcal{F}\,\lambda_1$ we obtain
\begin{equation}\label{eq:identification_of_the_limit_product_step_3_B}
\begin{split}
\overline{\frac{u}{s} \, \partial_x s^{\alpha}} \, \overline{\, \frac{u^2}{s^2}\,} = &\, \mathcal{F}\, \left(\int_{\R^2} \lambda_1^2 \diff \pi_{t,x}(\lambda_1, \lambda_2)\right)^2 \\&+ \left(\mathcal{F}\, \frac{u^2}{s^2} + \frac{u}{s}\, \partial_x s^{\alpha} \right) \int_{\R^2} \lambda_1^2 \diff \pi_{t,x}(\lambda_1, \lambda_2) + \frac{u^3}{s^3}\, \partial_x s^{\alpha}.
\end{split}
\end{equation}
Combining \eqref{eq:identification_of_the_limit_step_3_A_fraction_squared_times_gradient} and \eqref{eq:identification_of_the_limit_product_step_3_B} we finally obtain \eqref{eq:2nd_application_comp_compactness_limit_identification_step_3_mixed_fractions_gradients}.\\

\underline{Step 4: Conclusion.} We plug \eqref{eq:representation_weak_limit_u4_s4_step3} and \eqref{eq:2nd_application_comp_compactness_limit_identification_step_3_mixed_fractions_gradients} into \eqref{eq:2nd_application_comp_compactn_identity_step_2} to obtain
$$
\frac{1}{3}\,  \int_{\R^2} \lambda_1^4  \diff \pi_{t,x}(\lambda_1, \lambda_2) + \left(\int_{\R^2} \lambda_1^2 \diff \pi_{t,x}(\lambda_1, \lambda_2)\right)^2 = 0
$$
which concludes the proof.

\subsection{Extension to the nonconservative case \eqref{eq:general_cross_diffusion_intro_short_velocity_noncons}}\label{sect:extension_nonconservative}
We briefly outline the necessary modifications in order to treat the nonconservative system \eqref{eq:general_cross_diffusion_intro_short_velocity_noncons}.\\

\underline{Sections \ref{subsect:aprioriestimates} and \ref{subsect:basic_compactness}: A priori estimates and compactness.} The proofs of \ref{est:moment}, \ref{est:gradient_from_entropy}, and \ref{est:entropy_derivatives_u_v_eps_signular} are essentially identical. First, the mass $M$ is no longer fixed, but it can be controlled by a simple Gr\"onwall argument. Therefore, in what follows, we still assume that $\int_{\R} \seps \diff x \leq M$ for all $t\in[0,T]$. The identity \eqref{eq:entropy_bound_alpha_between_0_and_13} has an additional term on the (RHS)
$$
(\|G^1\|_{L^{\infty}} + \|G^2\|_{L^{\infty}}) \int_{\R} (\ueps |\log \ueps| + \veps |\log \veps|) \diff x,
$$
which can be estimated in terms of the integrals of $\ueps \log \ueps$, $\ueps\, |x|$ and $e^{-|x|}$ by Lemma~\ref{lem:control_negative_parts_of_log}. Similarly, the identity \eqref{eq:final_identity_for_evol_1st_moment_before_splitting_2_cases} gains an additional term on the (RHS) of the form
$$
(\|G^1\|_{L^{\infty}} + \|G^2\|_{L^{\infty}}) \int_{\R} \seps \,|x| \diff x.
$$
This still permits closing the Gr\"onwall-type argument for $\int_{\R} (\ueps|\log \ueps| + \veps|\log \veps| + \seps|x|) \diff x$ when $\alpha \leq 1$. For $\alpha >1$, the identity \eqref{eq:identity_estimates_alpha_geq_1_energy} gets an additional term on the (RHS) of the form $
(\|G^1\|_{L^{\infty}} + \|G^2\|_{L^{\infty}}) \int_{\R} \seps^{\alpha} \diff x,
$ 
so Gr\"onwall's inequality can be applied once again.\\

The main modification concerns the $L^{\infty}$ bound on $\seps$ obtained via Alikakos' method. Indeed, the key identity for $\partial_t \int_{\R} \seps^{\gamma} \diff x$ used there, namely \eqref{eq:evolution_of_the_gamma_moment_before_doing_the_interpolation_without_the_rubbish}, has an additional term $C\,\gamma\int_{\R} \seps^{\gamma} \diff x$ and its dependence on $\gamma$ requires some care. Nevertheless, this term still allows one to conclude that $\{\seps\}$ is bounded in $L^{\infty}(0,T; L^p(\R))$ for all $1 \leq p<\infty$ by the argument following \eqref{eq:evolution_of_the_gamma_moment_before_doing_the_interpolation_without_the_rubbish}. In particular, \eqref{eq:claim_s_to_power_3_halves_is_in_l2} remains valid. Proceeding to the iteration scheme, it relies on estimating the (RHS) of \eqref{eq:evolution_of_the_gamma_moment_before_doing_the_interpolation_without_the_rubbish}. The key observation is that the new term $\gamma\int_{\R} \seps^{\gamma} \diff x$ can be handled in the same way as the term $\frac{\gamma\,(\gamma-1)}{2\,\delta_1} \int_{\R} \seps^{-\alpha+\gamma+1} \diff x$, which already appears there. Repeating the computation in \eqref{eq:estimate_Linf_bound_integral_-alph+gamma_first_splitting}, we write 
$$
\int_{\R} \seps^{\gamma} \diff x =  \int_{\R} \seps^{\frac{\gamma+\alpha}{2}+1}\, \seps^{\frac{\gamma-\alpha}{2}-1} \diff x \leq \left\| \seps^{\frac{\gamma+\alpha}{2}+1}  \right\|_{L^{\infty}(\R)} \int_{\R}\seps^{\frac{\gamma-\alpha}{2}-1} \diff x. 
$$
The first factor $\| \seps^{\frac{\gamma+\alpha}{2}+1} \|_{L^{\infty}(\R)}$ has already been estimated in \eqref{eq:estimate_L_inf_norm_term_half_of_gamma_via_dissipation}. For the second factor, since $\frac{\gamma-\alpha}{2}-1<\frac{\gamma}{2}$, we can estimate $\int_{\R}\seps^{\frac{\gamma-\alpha}{2}-1} \diff x \leq (\int_{\R} \seps^{\frac{\gamma}{2}} \diff x)^{\beta}\, M^{1-\beta}$ with $\beta = \frac{\frac{\gamma-\alpha}{2}-2}{\frac{\gamma}{2}-1}$, analogously to \eqref{eq:Holder_to_adjust_the_exponent_to_gamma_over_2_for_alikakos}. Note that $\beta \in [\frac{4}{5},1]$ for $\gamma$ sufficiently large. Therefore, we obtain an inequality analogous to \eqref{eq:final_bound_to_start_alikakos}, which allows us to apply the Alikakos' method. The remaining estimates, as well as the convergence results in Section~\ref{subsect:basic_compactness}, are unaffected by the additional term.\\

\underline{Sections \ref{subsect:cons_laws_all_alpha} and \ref{subsect:cons_law_u2_over_s}: Conservation laws.} Only minor modifications are required here. In Proposition~\ref{prop:new_cons_law_log_general_alpha}, the function $g_{\eps}$ contains an additional term that reads $\ueps\, (1+\log \ueps)\, G^1(\seps) + \veps\, (1+\log \veps)\, G^2(\seps)$, which belongs to the desired space by \ref{est:ulogu_vlogv}. Similarly, in Proposition~\ref{prop:new_cons_law}, $g_{\eps}$ contains an additional term $\frac{2\,\ueps^2}{\seps}\, G^1(\seps) - \frac{\ueps^3}{\seps^2}\,G^1(\seps) - \frac{\ueps^2\, \veps}{\seps^2}\,G^2(\seps)$,  which again belongs to the desired space by \ref{est:conservation}.\\

\underline{Section \ref{subsect:dissipation_measure_gradient_squared}: The limit of $|\partial_x \seps^{\alpha} - \partial_x s^{\alpha}|^2$.} The identification of the limit relies on the PDE for $\seps$, integrated in space and multiplied by $\partial_x \seps^{\alpha}$. The additional terms to be handled are $\left[\int_{-\infty}^x \ueps\, G^1(\seps) \diff y \right]\, \partial_x \seps^{\alpha}$, $\left[\int_{-\infty}^x \veps\, G^2(\seps) \diff y \right]\, \partial_x \seps^{\alpha}$. A standard argument, based on integration by parts and the strong convergence of $\{\seps\}$, shows that these terms converge to $\left[\int_{-\infty}^x u\, G^1(s) \diff y \right]\, \partial_x s^{\alpha}$, $\left[\int_{-\infty}^x v\, G^2(s) \diff y \right]\, \partial_x s^{\alpha}$. Consequently, they do not contribute to the limit of $|\partial_x \seps^{\alpha} - \partial_x s^{\alpha}|^2$.  \\

\underline{Sections \ref{sect:young_measures_compensated_compactness}, \ref{subsect:support_is_a_line}, \ref{subsect:the_case_that_YM_supp_on_a_nonzero_line}, and \ref{subsect:alternative_proof_alpha_at_most_3}: Young measures and compensated compactness.} There are no changes needed there. The inclusion of the new terms in the function $g_{\eps}$ (in Sections \ref{subsect:cons_laws_all_alpha} and \ref{subsect:cons_law_u2_over_s}) above makes the additional terms completely invisible in the compensated compactness argument.\\

\underline{Section \ref{sect:main_proof_after_two_steps}: Proof of Theorem \ref{thm:main}.} The only required modification arises in the proof of the strong convergence of $\{\partial_x \seps^{\alpha}\}$. In \eqref{eq:energy_identity_before_sending_eps_to_0}, an additional term $\ueps\, G^1(\seps)\, \seps^{\alpha} + \veps\, G^2(\seps)\, \seps^{\alpha}$ appears, whose weak limit is $u\, G^1(s)\, s^{\alpha} + v\, G^2(s)\, s^{\alpha}$. This is precisely the extra term that arises in \eqref{eq:energy_identity_after_sending_eps_to_0}. Consequently, these new contributions cancel within the energy argument of Section~\ref{sect:main_proof_after_two_steps}, and the same conclusion follows.

\section*{Acknowledgements}
I would like to thank José A. Carrillo, Charles Elbar, Ansgar Jüngel, Alpár R. Mészáros, Guy Parker, and Yao Yao for fruitful discussions that significantly improved my understanding of the problem. I am grateful to Yurij Salmaniw for pointing out that the result extends naturally to the case of weighted pressure (Remark~\ref{rem:Darcy_law_weights}), which is of interest in mathematical biology. I also thank Filippo Santambrogio for the observation that, in order to obtain tightness of the sequence $\{\seps\}$, it suffices to bound $\int_{\R} \seps \, |x| \diff x$ for all $\alpha \in (0,\infty)$. In a~previous version of the manuscript, we used $\int_{\R} \seps \, |x| \diff x$ for $\alpha \leq \frac{1}{3}$ and $\int_{\R} \seps \, |x|^2 \diff x$ for $\alpha > \frac{1}{3}$. I~further thank Simon Schulz for suggesting additional references on $2\times 2$ hyperbolic systems and for noting that the present work exploits only three entropy-entropy flux pairs. I am also grateful to Jakub Woźnicki for pointing out that, in the second application of compensated compactness, \cite{MR2769903} can be used, which greatly simplifies the argument. Finally, I thank Yurij Salmaniw and Jethro Warnett for numerous remarks and corrections on a~preliminary version of this work, which greatly improved its readability.
\appendix

\bibliographystyle{abbrv}
\bibliography{fastlimit}
\end{document}